\documentclass[12pt,reqno]{amsart}

\usepackage{hyperref}
\usepackage[usenames]{color}
\usepackage[latin1]{inputenc} 
\usepackage{float}
\usepackage{bbm}

\usepackage[displaymath, mathlines]{lineno}

\usepackage[english]{babel} 

\usepackage{amssymb, amsmath,mathtools}
\usepackage{geometry,graphicx}
\graphicspath{{figures/}}
\usepackage{cite}
\usepackage{algorithm}
\usepackage{algpseudocode}
\usepackage{hhline}

\newcommand{\Cov}{\mathbb{C}\mathrm{ov}}
\newcommand{\e}{\mathrm{e}}
\newcommand{\dif}{\mathrm{d}}
\newcommand{\leb}{\mathrm{L}}

\theoremstyle{plain}
\newtheorem{theorem}{Theorem}[section] 

\newtheorem{lemma}[theorem]{Lemma} 

\newtheorem{remark}[theorem]{Remark}

\newtheorem{example}[theorem]{Example}

\numberwithin{equation}{section} 
\title[Linear differential equation with analytic uncertainties]{Second order linear differential equations with analytic uncertainties: stochastic analysis via the computation of the probability density function}

\author[Marc Jornet, Julia Calatayud, Olivier P. Le Ma\^itre, Juan Carlos Cort\'es]{Marc Jornet $^1$, Julia Calatayud $^1$, Olivier P. Le Ma\^itre $^2$, \\ Juan Carlos Cort\'es $^1$}

\begin{document}

\maketitle

\begin{center}
\noindent
\address{$^1$ Institut de Matem\`{a}tica Multidisciplinar, Universitat Polit\`{e}cnica de Val\`{e}ncia, \\
 Cam\'i de Vera s/n, 46022, Valencia, Spain. \\
\ \\
$^2$ CMAP, CNRS, INRIA, \'Ecole Polytechnique, Institut Polytechnique de Paris, 91128 Palaiseau, France. \\
\ \\
Emails: marjorsa@doctor.upv.es; jucagre@doctor.upv.es; olivier.le-maitre@polytechnique.edu; jccortes@imm.upv.es
}
\end{center}

\begin{abstract}
This paper concerns the analysis of random second order linear differential equations. Usually, solving these equations consists of computing the first statistics of the response process, and that task has been an essential goal in the literature. A more ambitious objective is the computation of the solution probability density function. We present advances on these two aspects in the case of general random non-autonomous second order linear differential equations with analytic data processes. 
The Fr\"obenius method is employed to obtain the stochastic solution in the form of a mean square convergent power series. We demonstrate that the convergence requires the boundedness of the random input coefficients. Further, the mean square error of the Fr\"obenius method is proved to decrease exponentially with the number of terms in the series, although not uniformly in time. 
Regarding the probability density function of the solution at a given time, which is the focus of the paper, we rely on the law of total probability to express it in closed-form as an expectation. For the computation of this expectation, a sequence of approximating density functions is constructed by reducing the dimensionality of the problem using the truncated power series of the fundamental set. We prove several theoretical results regarding the pointwise convergence of the sequence of density functions and the convergence in total variation. The pointwise convergence turns out to be exponential under a Lipschitz hypothesis. 
As the density functions are expressed in terms of expectations, we propose a crude Monte Carlo sampling algorithm for their estimation. This algorithm is implemented and applied on several numerical examples designed to illustrate the theoretical findings of the paper. After that, the efficiency of the algorithm is improved by employing the control variates method. Numerical examples corroborate the variance reduction of the Monte Carlo approach. \\ 
\\ 
\textit{Keywords:} random non-autonomous second order linear differential equation; mean square analytic solution; probability density function; Monte Carlo simulation; uncertainty quantification. \\
\\
\textit{AMS Classification 2010:} 37H10; 34F05; 60H10; 60H35; 65C05.
\end{abstract}

\section{Introduction}
Random differential equations are ordinary differential equations whose input coefficients are random quantities, in the form of random variables or stochastic processes (not to be confused with It\^o's stochastic differential equations, which are differential equations driven by Wiener noise). 
In this setting, there is a complete abstract probability space $(\Omega,\mathcal{F},\mathbb{P})$, where $\Omega$ is the sample space defined as the set of outcomes $\omega\in\Omega$, $\mathcal{F}\subseteq 2^\Omega$ is the $\sigma$-algebra of events, and $\mathbb{P}$ is the probability measure. 

The stochastic solution may be considered in two senses. One approach considers the random calculus that arises from the Lebesgue space $(\leb^p(\Omega),\|\cdot\|_p)$, $1\leq p\leq\infty$, where the norms are defined as $\|X\|_p=\mathbb{E}[|X|^p]^{1/p}$ for $p<\infty$, and $\|X\|_\infty=\inf\{C>0:\,|X|\leq C \text{ almost surely}\}$ (essential supremum), where $\mathbb{E}$ denotes the expectation operator defined as $\mathbb{E}[X]=\int_\Omega X\,\dif\mathbb{P}$. The limits in the definitions of continuity, differentiability, and Riemann integrability are considered in the topology of $\leb^p(\Omega)$. The space $(\leb^p(\Omega),\|\cdot\|_p)$, $1\leq p\leq\infty$, is a Banach space. Of particular importance is the case $p=2$, which gives rise to the Hilbert space $(\leb^2(\Omega),\langle \cdot,\cdot\rangle)$ of random variables with finite variance, whose inner product is defined as $\langle X,Y\rangle=\mathbb{E}[XY]$, $X,Y\in\leb^2(\Omega)$. The calculus in $\leb^2(\Omega)$ is referred to as the mean square calculus. A key feature of $\leb^2(\Omega)$ is that mean square convergence ensures convergence of the expectation and the variance. An alternative strategy to tackle random differential equations is the sample path approach, which considers the trajectories of the solution process by fixing each outcome $\omega\in\Omega$. An interesting result that links $\leb^p(\Omega)$ and sample path calculus states that every $\leb^p(\Omega)$ solution is also a sample path solution. For theoretical discussions on random differential equations and the types of stochastic solutions, we refer the reader to \cite{soong,neckel,strand,jc}.

Understanding the inherent stochastic nature of the solution is of primary importance. This is the focus of uncertainty quantification~\cite{smith}. The most common strategies for uncertainty quantification are Monte Carlo simulation~\cite{mc}, PC (polynomial chaos) expansions \cite{pc,xiu_llibre,olivier_llibre}, and perturbation methods~\cite{soong,xiu_llibre,liu2}. Some studies of different random differential equation problems providing a fair overview of the state-of-the-art literature can be found, for instance, in \cite{ricati_benito,tawil,dorini1,hussein1,pendol,nouri}.

In the case of random second order linear differential equations, important advances have been achieved for the computation of the first moments of the solution, via mean square calculus and the so-called Fr\"obenius method. The Fr\"obenius method consists in finding a mean square convergent power series solution, in analogy to the deterministic theory of ordinary differential equations. The general stochastic system is given by
\begin{equation}
 \begin{cases} \ddot{X}(t)+A(t)\dot{X}(t)+B(t)X(t)=0, \; t\in\mathbb{R}, \\ X(t_0)=Y_0, \;\dot{X}(t_0)=Y_1. \end{cases} 
 \label{problem}
\end{equation}
Here, $A(t)$ and $B(t)$ are stochastic processes and $Y_0$ and $Y_1$ are random variables on $(\Omega,\mathcal{F},\mathbb{P})$. The stochastic process $X(t)$ is the solution. We will assume that $A(t)$ and $B(t)$ are analytic stochastic processes on $(t_0-r,t_0+r)$, for $r>0$ fixed, in the mean square sense~\cite[p.~99]{soong}: $A(t)=\sum_{n=0}^\infty A_n (t-t_0)^n$ and $B(t)=\sum_{n=0}^\infty B_n (t-t_0)^n$ are two random power series in $\leb^2(\Omega)$, where $A_0,A_1,\ldots$, $B_0,B_1,\ldots$ are second order random variables. The expansions coincide with the Taylor series of $A(t)$ and $B(t)$. 

Airy, Hermite and Legendre differential equations are particular instances of~\eqref{problem}, which represent important stochastic models of Mathematical Physics. The rigorous analysis and construction of mean square solutions to these particular equations, using random power series, can be found in \cite{airy,hermite,legendre2}. We have proposed a generalization of these contributions to the general system~\eqref{problem} in \cite{frob1,frob2}.

In the recent paper~\cite{gpc2}, we investigated the resolution of random second order linear differential equations with PC-based methods. In~\cite{homotopy}, the authors proposed a homotopy technique to solve some particular random differential equations pertaining to the class given in~\eqref{problem}. Other solution techniques include variational iteration~\cite{variational} and Adomian decomposition~\cite{adomian}. Finally, a technique, analogous to the Fr\"obenius method but relying on the concept of differential transform, is proposed in \cite{dif1,dif2}.

A more ambitious objective is the computation of the probability density function of $X(t)$, denoted hereafter as $f_{X(t)}(x)=\frac{\dif (\mathbb{P}\circ X(t)^{-1})(x)}{\dif x}$. The probability density function is defined as a non-negative Borel measurable function characterized by $\mathbb{P}[X(t)\in \mathcal{C}]=\int_{\mathcal{C}} f_{X(t)}(x)\,\dif x$. Random variables having a probability density function are called absolutely continuous, meaning that their probability law is absolutely continuous with respect to the Lebesgue measure. The density function allows for calculating general statistics (expectation, variance, skewness, kurtosis, median, quantiles, mode, etc.) and confidence intervals via integration. 

In~\cite{constant}, the authors constructed approximations of the probability density functions of the solution to~\eqref{problem} when $A(t)$ and $B(t)$ do not vary stochastically in time, that is, when $A(t)=A$ and $B(t)=B$ are actually absolutely continuous random variables (autonomous case). A recent paper, \cite{ana}, presents the approximation of the probability density function of $X(t)$ when $A(t)=p(t;D)$ and $B(t)=q(t;D)$, that is to say, when both $A(t)$ and $B(t)$ depend on a unique absolutely continuous random variable $D$. This approach does not extend to the general problem~\eqref{problem} and certain theoretical points from that contribution are unclear. 

In this work, we provide an analysis of~\eqref{problem} via the Fr\"obenius method. The solution is expressed in the form of a mean square convergent power series, under $\leb^\infty(\Omega)$ convergence of $A(t)=\sum_{n=0}^\infty A_n (t-t_0)^n$ and $B(t)=\sum_{n=0}^\infty B_n (t-t_0)^n$ and mean square integrability of the initial data $Y_0$ and $Y_1$. The boundedness of the coefficients $A_0,A_1,\ldots$, $B_0,B_1,\ldots$ is necessary, as shown in examples of the paper. Truncation of unbounded supports of random coefficients can be carried out to assure the required boundedness. The bias error of the Fr\"obenius method is proved to decrease exponentially with the number of terms in the series. Therefore rapid approximations of the statistical moments of $X(t)$ can be derived. However, the exponential convergence is not uniform in time, and it may deteriorate as we move away from the initial instant $t_0$. Section~\ref{stoch_solution} considers these issues. An additional issue is the computation of the probability density function of $X(t)$. This is the main contribution of the paper in terms of novelty and length. Theoretically, the probability density function is given by a closed-form expression in terms of an expectation derived from the law of total probability and by exploiting the linearity of the problem. However, to evaluate it in practice, a dimensionality reduction of the problem is required. By truncating the power series, we construct a sequence of probability density functions that, under certain assumptions regarding Nemytskii operators, converges to the target density function pointwise. In this setting, the pointwise convergence of the densities implies convergence in $\leb^1(\mathbb{R})$ (total variation distance), and in fact, in $\leb^p(\mathbb{R})$, for $1\leq p<\infty$. The pointwise convergence rate is proved to be exponential under a certain Lipschitz condition, albeit being again not uniform in time. This theoretical analysis on the approximation of the probability density function is presented in Section~\ref{sec_dens}. As each approximating density function is expressed in terms of an expectation, they can be estimated via a Monte Carlo sampling strategy. The brute-force Monte Carlo method is implemented in the form of an algorithm, whose computational aspects are detailed in Section~\ref{comp_aspects}. The proposed algorithm is tested on several numerical examples in Section~\ref{num_examples}, to verify the theoretical findings of the paper and to illustrate computational aspects. In Section~\ref{sec_control_var}, we apply a variance reduction approach, the control variates method, to improve the efficiency of the crude Monte Carlo algorithm from the previous sections; the method is analyzed from the computational viewpoint and several test examples corroborate the gains. Finally, Section~\ref{concl} draws the main conclusions and points out potential lines of research for the future.

\section{Stochastic solution} \label{stoch_solution}

The initial value problem~\eqref{problem} was previously studied in the mean square sense. We start by recalling the mean square existence and uniqueness theorem proved in our recent contributions~\cite{frob1,frob2}. The proof uses fundamental results from deterministic power series extended to the random scenario, together with the basics of difference equations.

\begin{theorem} \label{nostre} \cite[Th.~2]{frob2}
Let $A(t)=\sum_{n=0}^\infty A_n (t-t_0)^n$ and $B(t)=\sum_{n=0}^\infty B_n (t-t_0)^n$ be two random series in the $\leb^\infty(\Omega)$ setting, for $t\in (t_0-r,t_0+r)$, being $r>0$ finite and fixed. Assume that the initial conditions $Y_0$ and $Y_1$ belong to $\leb^2(\Omega)$. Then the stochastic process $X(t)=\sum_{n=0}^\infty X_n(t-t_0)^n$, $t\in (t_0-r,t_0+r)$, where $X_0=Y_0$, $X_1=Y_1$ and for $n\geq 0$, $X_{n+2}=\frac{-1}{(n+2)(n+1)}\sum_{m=0}^n [(m+1)A_{n-m}X_{m+1}+B_{n-m}X_m]$,
is an analytic solution to the random initial value problem~\eqref{problem} in the mean square sense. Moreover, it is unique. Furthermore, by \cite[Subsection~3.4]{frob1}, if $Y_0$ and $Y_1$ are bounded random variables, then $X(t)$ is an analytic $\leb^\infty(\Omega)$ solution to~\eqref{problem}.
\end{theorem}

From this fundamental theorem we can extend the theory to a more general convergence measure, by considering $\leb^p(\Omega)$ convergence, $1\leq p\leq\infty$: if $A(t)$ and $B(t)$ are two random power series with convergence in $\leb^\infty(\Omega)$, for $t\in (t_0-r,t_0+r)$, and the initial conditions $Y_0$ and $Y_1$ belong to $\leb^p(\Omega)$, then the stochastic process $X(t)=\sum_{n=0}^\infty X_n(t-t_0)^n$ is the $\leb^p(\Omega)$ solution to~\eqref{problem} on $(t_0-r,t_0+r)$.

Regarding the rapidity of convergence of the power series $X(t)=\sum_{n=0}^\infty X_n (t-t_0)^n$ introduced in Theorem~\ref{nostre}, some theoretical estimates were obtained in \cite[Subsection~3.6]{frob1}, although no rate of convergence was derived. Fixed $r>0$ finite, given $\rho\coloneqq |t-t_0|<r$ and given an arbitrary $s$ such that $\rho<s<r$, the following estimate holds:
\[ \|X^N(t)-X(t)\|_2\leq K\left(r,s,\{\|A_i\|_\infty\}_{i=1}^\infty,\{\|B_i\|_\infty\}_{i=1}^\infty,\|Y_0\|_2,\|Y_1\|_2\right)\cdot \frac{(\rho/s)^{N+1}}{1-\rho/s}. \]
In general, the estimate holds for $p$-norms. The constant $K$ can be constructed as follows (see~\cite{frob1}):
\begin{enumerate}
\item[Step 1.] Given $u=(r+s)/2\in (s,r)$, choose a constant $C_u>0$ such that $\|A_i\|_\infty\leq C_u/u^i$ and $\|B_i\|_\infty\leq C_u/u^i$, $i\geq0$. Such a constant $C_u$ exists because $\sum_{i=0}^\infty \|A_i\|_\infty u^i<\infty$ and $\sum_{i=0}^\infty \|B_i\|_\infty u^i<\infty$.
\item[Step 2.] Pick an integer $n\geq0$ such that $\frac{ns}{(n+2)u}+\frac{C_u s}{n+2}+\frac{C_u s^2}{(n+2)(n+1)}<1$.
\item[Step 3.] Take $K=\max_{0\leq m\leq n} H_m s^m$, where $\{H_m\}_{m=0}^\infty$ satisfies the recursive equation: $H_0=\|Y_0\|_2$, $H_1=\|Y_1\|_2$ and for $m\ge 0$, $H_{m+2}=\left(\frac{m}{(m+2)u}+\frac{C_u}{m+2}\right)H_{m+1}+\frac{C_u}{(m+2)(m+1)}H_m$. 
\end{enumerate}

From the constructed $K$ and given a target error $\epsilon>0$, a truncation order $N$ satisfying $N>\log( \epsilon^{-1} K (1-\rho/s)^{-1})/\log(s/\rho)-1=\mathcal{O}(\log (\epsilon^{-1}))$ guarantees a root mean square error $\|X^N(t)-X(t)\|_2$ less than $\epsilon$. The number $s$ is arbitrary in $(\rho,r)$. Unfortunately, we are not aware of any method to choose the optimal $s\in (\rho,r)$ minimizing $N=\log( \epsilon^{-1} K (1-\rho/s)^{-1})/\log(s/\rho)$.

We stress several new consequences from these estimates. First, the rate of convergence of $\{X^N(t)\}_{N=0}^\infty$ towards $X(t)$ as $N\rightarrow\infty$ is exponential, for $t\in (t_0-r,t_0+r)$, because it is proportional to $(\rho/s)^N$. Second, the convergence rate may deteriorate severely for large $\rho=|t-t_0|$ and large norms of the random input coefficients. Indeed, $K$ is growing with $s>\rho$ and the norms.

The fact that the convergence rate deteriorates for large $|t-t_0|$ is clear. Assume that we have $\sum_{n=N}^\infty \|X_n\|_2 |t-t_0|^n<\epsilon$, for some target error $\epsilon>0$. As $|t-t_0|^n$ increases when $|t-t_0|$ grows, a larger $N$ is needed to achieve a root mean square error less than $\epsilon$. This fact may especially occur for $|t-t_0|\geq1$, as in this case $|t-t_0|^n$ does not tend to $0$ when $n\rightarrow\infty$, therefore a faster decay of the coefficients $\|X_n\|_2$ is needed to assure convergence.

The numerical experiments presented in~\cite{frob1} also permitted analyzing the behavior of convergence. As theoretically expected by our exposition, the most important phenomenon observed was that the convergence rate deteriorates severely when the distance $|t-t_0|$ increases, therefore making the Fr\"obenius method computationally intractable. This issue also occurs with PC-type methods, which require large orders for long-time integration~\cite{timedepgpc} (in this case, multi-element methods may be an alternative \cite{wan,augustin}). Nonetheless, for not too large $|t-t_0|$ (this ``not too large'' is problem-dependent), the Fr\"obenius method works very well. 

Having analyzed the convergence rate of the Fr\"obenius method, let us focus now on the assumptions of Theorem~\ref{nostre}. In~\cite{frob2}, the following open problem was raised: ``If there exists a point $t_1\in (t_0-r,t_0+r)$ such that $A(t_1)\notin\leb^\infty(\Omega)$ or $B(t_1)\notin\leb^\infty(\Omega)$, then there exist two initial conditions $Y_0,Y_1\in\leb^2(\Omega)$ such that \eqref{problem} has no mean square solution on $(t_0-r,t_0+r)$''. This problem implies that the hypotheses used in Theorem~\ref{nostre} are sharp, in the sense that counterexamples exist if any of them is relaxed. The following two examples of~\eqref{problem} with an unbounded input coefficient have no mean square solution $X(t)$. The arguments to prove that these examples have no solution follow the reasoning of \cite[Example, pp.~541--542]{strand}.

\begin{example}[Non-existence of mean square solution $X(t)$] \label{non_ex1} \normalfont Consider the initial value problem~\eqref{problem} with $A(t)=0$, $B(t)=Z$ and the initial conditions $X(t_0=0)=Y_0$ and $\dot{X}(t_0=0)=0$.
Let $Z<0$ be an unbounded random variable (for example, $Z=-U$, where $U$ follows an Exponential, Gamma, Poisson, etc. distribution). Suppose that for any initial condition $Y_0\in\leb^2(\Omega)$ there is a mean square solution $X(t)$. By \cite[Th.~3(a)]{strand}, every mean square solution to a random differential equation problem is a sample path solution. More specifically, there exists an equivalent stochastic process, product measurable, whose sample paths solve the deterministic counterpart of the problem almost surely. Therefore $X(t)$ is a sample path solution (we choose the appropriate representative of the equivalence class), with $X(t)=Y_0\cosh(\sqrt{-Z}\,t)$ for all $t\in\mathbb{R}$, almost surely. Fix $t\neq0$. Consider the random variable $T=\cosh(\sqrt{-Z}\,t)$. Notice that $\|T\|_\infty=\infty$. Consider the operator $\Delta:\leb^2(\Omega)\rightarrow\leb^2(\Omega)$, $\Delta(Y)=YT$. This operator is linear and continuous, as a consequence of the closed graph theorem. Hence, there is a constant $C>0$ such that $\|Y T\|_2\leq C\|Y\|_2$, for all $Y\in\leb^2(\Omega)$. In fact, this inequality holds for any random variable $Y$ (since, if $Y\notin\leb^2(\Omega)$, then $\|Y\|_2=\infty$). Let $Y=T^m$. We have $\|T^{m+1}\|_2\leq C\|T^m\|_2$, which yields $\|T^m\|_2\leq C^m$. That is, $\|T\|_{2m}\leq C$. Hence, $\|T\|_\infty=\lim_{m\rightarrow\infty} \|T\|_{2m}\leq C$, but this is a contradiction. Thus, we conclude that there must exist an initial condition $Y_0\in\leb^2(\Omega)$ such that the stochastic problem has no mean square solution. 

The case in which $Z>0$ is unbounded (let us suppose that $Z$ is Gamma distributed) may be tackled analogously, although with a subtlety. Proceeding again by contradiction, let us suppose that for any initial condition $Y_0\in\leb^2(\Omega)$ there exists a mean square solution $X(t)$. By \cite[Th.~3(a)]{strand}, $X(t)=Y_0\cos(\sqrt{Z}\,t)$ for all $t\in\mathbb{R}$, almost surely. In contrast with the previous case, now $\cos(\sqrt{Z}\,t)$ is bounded. As $X(t)$ is mean square differentiable, its mean square derivative must be given by $\dot{X}(t)=-Y_0\sqrt{Z}\sin(\sqrt{Z}\,t)$ \cite[p.~536]{doob}. Fix $t\neq0$ and let $T=-\sqrt{Z}\sin(\sqrt{Z}\,t)$. Now we do have that $\|T\|_\infty=\infty$, so the previous reasoning based on the closed graph theorem can be applied to deduce that there exists an initial condition $Y_0\in\leb^2(\Omega)$ such that $\dot{X}\notin\leb^2(\Omega)$. This is a contradiction.

The general case, in which $Z$ is an unbounded random variable, is easily addressed now (this includes, for instance, the case of Gaussian random variables). If $Z$ is unbounded, then it must be unbounded on the positive or negative axis. Let us suppose it unbounded on the positive axis (the other case is completely analogous). Take $\tilde{\Omega}\subseteq\Omega$ such that $\mathbb{P}[\tilde{\Omega}]>0$ and $Z(\omega)>0$ for each $\omega\in\tilde{\Omega}$. Consider the new probability subspace $(\tilde{\Omega},\mathcal{F}_{\tilde{\Omega}}=\mathcal{F}\cap 2^{\tilde{\Omega}},\mathbb{P}_{\tilde{\Omega}}=\mathbb{P}|_{\mathcal{F}_{\tilde{\Omega}}})$. We restate the random differential equation problem on this new probability space, where $Z>0$ is unbounded. The previous case thus applies. Therefore we are done since every mean square solution on $\Omega$ must also be a mean square solution on $\tilde{\Omega}$. This analysis terminates the example.
\end{example}

\begin{example}[Non-existence of mean square solution $X(t)$] \label{non_ex2} \normalfont
Let us consider now problem~\eqref{problem} with $A(t)=Z$, $B(t)=0$ and the initial conditions $X(t_0=0)=0$, $\dot{X}(t_0=0)=Y_1$.
Let $Z$ be any unbounded random variable. Suppose that for any initial condition $Y_1$, there exists a mean square solution $X(t)$. Let $Y(t)=\dot{X}(t)$, which satisfies $\dot{Y}(t)+Z Y(t)=0$, $Y(t_0=0)=Y_1$.
By \cite[Th.~3(a)]{strand}, $Y(t)=Y_1\e^{-Zt}$ for all $t\in\mathbb{R}$, almost surely. Fix $t\neq0$ and let $T=\e^{-Zt}$. The random variable $T$ is unbounded. Hence, the same reasoning from Example~\ref{non_ex1} based on the closed graph theorem applies again. We conclude that there must exist $Y_1\in\leb^2(\Omega)$ such that $Y(t)\notin \leb^2(\Omega)$, which is a contradiction, and we are done with this example.
\end{example}

The boundedness of the random input coefficients is crucial to obtain the Lipschitz condition demanded by the general existence and uniqueness theorem for random differential equations \cite[pp.~118--119]{soong}, \cite{strand}, \cite[Th.~4]{frob2}. In practice, to satisfy this mandatory boundedness, one may truncate the support to a large but bounded interval.

\section{Computation of the probability density function} \label{sec_dens}

We now turn to the computation of the probability density function of $X(t)$. Having clarified the conditions for the existence of the solution, we start by rewriting $X(t)$ in an alternative form.

\begin{theorem} \label{nova_expr}
Let $A(t)=\sum_{n=0}^\infty A_n (t-t_0)^n$ and $B(t)=\sum_{n=0}^\infty B_n (t-t_0)^n$ be two random series in the $\leb^\infty(\Omega)$ setting, for $t\in (t_0-r,t_0+r)$, being $r>0$ finite and fixed. Assume that the initial conditions $Y_0$ and $Y_1$ belong to $\leb^2(\Omega)$. Then the mean square analytic solution $X(t)$ can be expressed as $X(t)=Y_0 S_0(t)+Y_1 S_1(t)$, $t\in (t_0-r,t_0+r)$, where $S_0(t)$ and $S_1(t)$ are random power series solutions to~\eqref{problem} in $\leb^\infty(\Omega)$ for the \emph{deterministic} initial conditions
$S_0(t_0)=1$, $\dot{S}_0(t_0)=0$, and $S_1(t_0)=0$, $\dot{S}_1(t_0)=1$, respectively.
\end{theorem}

Notice that we write $X(t)$ as a linear combination of the fundamental set $\{S_0(t),S_1(t)\}$. This expression exploits the linearity of the problem. The processes $S_0(t)$ and $S_1(t)$ are random power series in $\leb^\infty(\Omega)$, 
\[ S_0(t)=\sum_{n=0}^\infty S_{0,n}(t-t_0)^n, \quad S_1(t)=\sum_{n=0}^\infty S_{1,n}(t-t_0)^n, \]
whose coefficients satisfy a difference equation as in Theorem~\ref{nostre}; for $S_0(t)$ it comes $S_{0,0}=1$, $S_{0,1}=0$, and for $n\geq0$,
$
 S_{0,n+2}=\frac{-1}{(n+2)(n+1)}\sum_{m=0}^n [(m+1)A_{n-m}S_{0,m+1}+B_{n-m}S_{0,m}] 
 $,
while for $S_1(t)$ we have $S_{1,0}=0$, $S_{1,1}=1$, and 
$
 S_{1,n+2}=\frac{-1}{(n+2)(n+1)}\sum_{m=0}^n [(m+1)A_{n-m}S_{1,m+1}+B_{n-m}S_{1,m}]
 $, for $n\geq0$.

The following lemma is necessary to compute the probability density function $f_{X(t)}(x)$. Its proof is a consequence of the law of total probability \cite[Ch.~6]{totalProb1}, \cite[Def.~7.11]{totalProb2}. 

\begin{lemma} \label{convo}
Let $U$ be an absolutely continuous random variable, independent of the random vector $(Z_1,Z_2)$, where $Z_1\neq0$ almost surely. Then $Z_1 U+Z_2$ is absolutely continuous, with density function $f_{Z_1 U+Z_2}(z)=\mathbb{E}[f_U((z-Z_2)/Z_1)/|Z_1|]$.
\end{lemma}

This lemma provides an alternative to the random variable transformation method \cite[Th.~1]{ana}, in the case of affine mappings. It does not require that the random quantities have an absolutely continuous probability law, a fact that presents advantages from the practical perspective. The drawback is that we need independence between $U$ and $(Z_1,Z_2)$ to represent the probability density function as an expectation. The expectation can be approximated via sampling-based statistical methods, as discussed later on.
The following theorem, which derives the probability density function of $X(t)$, is a straightforward consequence of Lemma~\ref{convo}.

\begin{theorem} \label{the_densitat}
Let $A(t)=\sum_{n=0}^\infty A_n (t-t_0)^n$ and $B(t)=\sum_{n=0}^\infty B_n (t-t_0)^n$ be two random series in the $\leb^\infty(\Omega)$ setting, for $t\in (t_0-r,t_0+r)$, being $r>0$ finite and fixed. Suppose that the initial conditions $Y_0$ and $Y_1$ belong to $\leb^2(\Omega)$. If $S_0(t)\neq0$ almost surely, if $Y_0$ is absolutely continuous, with density function $f_{Y_0}$, and it is independent of the rest of random input parameters of~\eqref{problem}, then the mean square solution $X(t)$ has for probability density function
\begin{equation}
 f_{X(t)}(x)=\mathbb{E}\left[f_{Y_0}\left(\frac{x-Y_1 S_1(t)}{S_0(t)}\right)\frac{1}{|S_0(t)|}\right]. 
 \label{dens_Xt}
\end{equation}
\end{theorem}
An important issue with expression~\eqref{dens_Xt} is that $S_0(t)$ and $S_1(t)$ are given by infinite series, therefore truncated approximations are needed. We have to justify that it is legitimate to reduce dimensionality and use truncated random power series for $S_0(t)$ and $S_1(t)$. In what follows, we denote by $S_0^N(t)$ and $S_1^N(t)$ the $N$-th partial sums of $S_0(t)$ and $S_1(t)$, respectively, which converge in $\leb^\infty(\Omega)$ for each $t$. Let $X^N(t)=Y_0 S_0^N(t)+Y_1 S_1^N(t)$ be a truncation of the solution $X(t)$, which converges in the mean square sense for each time $t$.

In the study of random differential equation problems with no closed-form expression of the solution process but only an infinite expansion, one usually constructs an approximating sequence of stochastic processes with reduced dimensionality and computable probability density function. Thus, one obtains an approximating sequence of probability density functions, which hopefully presents rapid convergence to the target density function. Moreover, the discontinuity and non-differentiability points of the target density function are captured with no difficulty. In the literature, one may find applications of this type of strategy with power series and Karhunen-Lo\`eve developments \cite{pendol,parabolic}, finite difference schemes \cite{tawil}, and PC expansions~\cite{symmetry}.

If $S_0^N(t)\neq0$ almost surely, the probability density function of $X^N(t)$ is
\begin{equation} f_{X^N(t)}(x)=\mathbb{E}\left[f_{Y_0}\left(\frac{x-Y_1 S_1^N(t)}{S_0^N(t)}\right)\frac{1}{|S_0^N(t)|}\right]. 
 \label{dens_Xnt}
\end{equation}
This expression involves a maximum of $2N+3$ random variables ($Y_1$, $S_{0,n}$ and $S_{1,n}$ for $0\leq n\leq N$). Thus, the expectation can be computed by numerical integration (in the case of absolutely continuous random input coefficients), or by a Monte Carlo procedure \cite[pp.~53--54]{xiu_llibre}, by sampling realizations of $Y_1$, $S_0^N(t)$ and $S_1^N(t)$. This is the same strategy as the one followed in our recent paper~\cite{pendol}. The approach based on numerical integration would be feasible only in the case of small $N$ and $A(t)=p(t;D)$, $B(t)=q(t;D)$ ($D$ random), as in~\cite{ana}, otherwise it is impractical. This is because the integration dimension relies on the dimension of the random space (the total number of input random variables). The Monte Carlo strategy can cope with high uncertainty dimension and large $N$, albeit at the expense of introducing a statistical error due to sampling, in addition to the bias error. The sampling error is reduced as the number of realizations increases, but at the cost of higher computational burden.

We need to justify that, for each $t$, $\lim_{N\rightarrow\infty} f_{X^N(t)}(x)=f_{X(t)}(x)$, for $x\in\mathbb{R}$. This is a strong mode of convergence. Indeed, as we are working with density functions, almost everywhere convergence on $\mathbb{R}$ implies convergence in $\leb^1(\mathbb{R})$, due to Scheff\'e's lemma \cite[p.~55]{williams}, \cite{scheffe}. This lemma states that if a general sequence of integrable functions converges almost everywhere to another integrable function, then convergence in $\leb^1(\mathbb{R})$ is equivalent to convergence of the $\leb^1(\mathbb{R})$ norms. As density functions have $\leb^1(\mathbb{R})$ norms equal to $1$ by definition, we deduce that almost everywhere convergence of the density functions $f_{X^N(t)}(x)$ to $f_{X(t)}(x)$ as $N\rightarrow\infty$ implies that $\| f_{X(t)}-f_{X^N(t)} \|_{\leb^1(\mathbb{R})}=\int_{\mathbb{R}} |f_{X(t)}(x)-f_{X^N(t)}(x)|\,\dif x \rightarrow 0$ as $N\rightarrow\infty$. Convergence in $\leb^1(\mathbb{R})$ is also referred to as convergence in total variation \cite[p.~41]{totalvar}: $\| (\mathbb{P}\circ (X^N(t))^{-1})- (\mathbb{P}\circ (X(t))^{-1})\|_{\mathrm{TV}}\coloneqq\sup\{ |\mathbb{P}[X^N(t)\in F]-\mathbb{P}[X(t)\in F]|:F\in\mathcal{F}\}= \frac12 \| f_{X^N(t)}-f_{X(t)} \|_{\leb^1(\mathbb{R})}$.
It is also equivalent to convergence in terms of the Hellinger distance~\cite{hellinger},
$
 H( \mathbb{P}\circ (X^N(t))^{-1},\mathbb{P}\circ (X(t))^{-1})\coloneqq 
 (1/\sqrt{2}) \| f_{X^N(t)}^{1/2}-f_{X(t)}^{1/2}\|_{\leb^2(\mathbb{R})}$,
via the elementary inequalities $H^2\leq \|\cdot\|_{\mathrm{TV}}\leq \sqrt{2}\,H$.

In fact, convergence in $\leb^1(\mathbb{R})$ may be generalized to convergence in $\leb^p(\mathbb{R})$, for $1<p<\infty$, by imposing boundedness on $\mathbb{R}$ of $f_{Y_0}$. Indeed, in this case, taking a constant $C>0$ such that $|f_{X^N(t)}(x)|\leq C$ and $|f_{X(t)}(x)|\leq C$, for $N\geq0$, $t$ and $x\in\mathbb{R}$, the mean value theorem leads to $| \|f_{X^N(t)}\|_{\leb^p(\mathbb{R})}^p-\|f_{X(t)}\|_{\leb^p(\mathbb{R})}^p|\leq p\, C^{p-1} \| f_{X^N(t)}-f_{X(t)} \|_{\leb^1(\mathbb{R})}$,
therefore $\|f_{X^N(t)}\|_{\leb^p(\mathbb{R})} \rightarrow \|f_{X(t)}\|_{\leb^p(\mathbb{R})}$ as $N\rightarrow\infty$. By Scheff\'e's lemma, there is convergence in $\leb^p(\mathbb{R})$: $\| f_{X(t)}-f_{X^N(t)} \|_{\leb^p(\mathbb{R})}=(\int_{\mathbb{R}} |f_{X(t)}(x)-f_{X^N(t)}(x)|^p\,\dif x)^{1/p} \rightarrow0$ as $N\rightarrow\infty$.

The pointwise convergence is the object of the following important Theorem~\ref{te1}. The result is proved in the spirit of our contribution~\cite{pendol}, by utilizing the concept of Nemytskii operator \cite[Remark~2.6]{pendol}, \cite[pp.~15--17]{ambrosetti}, \cite[pp.~154--163]{vainberg}.

\begin{lemma} \label{lemma_nemyt}
Let $\{V_N\}_{N=1}^\infty$ be a sequence of random variables that converges to $V$ in $\leb^2(\Omega)$. If $\mathbb{P}[V\in \mathcal{D}_{f_{Y_0}}]=0$, where $\mathcal{D}_{f_{Y_0}}$ is the set of discontinuity points of $f_{Y_0}$, and if $f_{Y_0}(y)\leq \alpha+\beta y^2$, for certain constants $\alpha,\beta\geq0$, then $f_{Y_0}(V_N)\rightarrow f_{Y_0}(V)$ as $N\rightarrow\infty$ in $\leb^1(\Omega)$.
\end{lemma}
\begin{proof}
There is a result in point-set topology that states that, given a sequence, if every subsequence has a subsequence itself that converges to $z_0$, then the complete sequence converges to $z_0$. This follows by a simple contradiction argument. Thus, it suffices to prove that, for every subsequence $\{V_{N_k}\}_{k=1}^\infty$, there exists a subsequence $\{V_{N_{k_l}}\}_{l=1}^\infty$ such that $\lim_{l\rightarrow\infty} f_{Y_0}(V_{N_{k_l}})=f_{Y_0}(V)$ in $\leb^1(\Omega)$. Fix any subsequence $\{V_{N_k}\}_{k=1}^\infty$. Since $\lim_{k\rightarrow\infty} V_{N_k}=V$ in $\leb^2(\Omega)$, by \cite[Th.~4.9]{brezis} there exist a subsequence $\{V_{N_{k_l}}\}_{l=1}^\infty$ and a random variable $\overline{V}\in\leb^2(\Omega)$ such that $\lim_{l\rightarrow\infty} V_{N_{k_l}}(\omega)=V(\omega)$ and $|V_{N_{k_l}}(\omega)|\leq \overline{V}(\omega)$ almost surely. Since $\mathbb{P}[V\in \mathcal{D}_{f_{Y_0}}]=0$, the continuous mapping theorem \cite[p.~7, Th.~2.3]{vaart} guarantees that $\lim_{l\rightarrow\infty} f_{Y_0}(V_{N_{k_l}}(\omega))=f_{Y_0}(V(\omega))$ almost surely. As $f_{Y_0}(V_{N_{k_l}}(\omega))\leq\alpha+\beta (V_{N_{k_l}}(\omega))^2\leq \alpha+\beta (\overline{V}(\omega))^2\in\leb^1(\Omega)$, we can apply the dominated convergence theorem to conclude that the desired limit holds: $\lim_{l\rightarrow\infty} f_{Y_0}(V_{N_{k_l}})=f_{Y_0}(V)$ in $\leb^1(\Omega)$.
\end{proof}

\begin{remark}
As $S_0(t_0)=1$ and $S_0(t)$ is continuous in $\leb^\infty(\Omega)$, we can find a neighborhood of $t_0$, say $(t_0-\delta,t_0+\delta)$ for certain $\delta>0$, such that $\|S_0(t)-1\|_\infty<1/4$ for all $t\in (t_0-\delta,t_0+\delta)$. Hence, $S_0(t)>3/4>0$ almost surely, for $t\in (t_0-\delta,t_0+\delta)$. Notice that such neighborhood may be limited; for instance, the deterministic function $X(t)=\sin t$ satisfies $\ddot{X}(t)+X(t)=0$, $X(t_0=\pi/2)=1$ and $\dot{X}(t_0=\pi/2)=0$.

For $t\in (t_0-\delta,t_0+\delta)$ fixed, there exists an integer $N_t>0$ such that $\|S_0^N(t)-S_0(t)\|_\infty<1/4$, for all $N\geq N_t$. Then $\|S_0^N(t)-1\|_\infty \leq \|S_0^N(t)-S_0(t)\|_\infty+\|S_0(t)-1\|_\infty<1/2$. This implies that $S_0^N(t)>1/2$ almost surely, $N\geq N_t$. From now on, we will work with $t\in(t_0-\delta,t_0+\delta)$.
\end{remark}

\begin{theorem} \label{te1}
Suppose the conditions of Theorem~\ref{the_densitat}. If $f_{Y_0}$ is continuous on $\mathbb{R}$ and $f_{Y_0}(y)\leq \alpha+\beta y^2$, for certain constants $\alpha,\beta\geq0$, then $\lim_{N\rightarrow\infty} f_{X^N(t)}(x)=f_{X(t)}(x)$, for each $t\in (t_0-\delta,t_0+\delta)$ and for every $x\in\mathbb{R}$.
\end{theorem}
\begin{proof}
Fix $t\in (t_0-\delta,t_0+\delta)$ and $x\in\mathbb{R}$. Let 
\begin{equation} V_N=\frac{x-Y_1 S_1^N(t)}{S_0^N(t)},\quad V=\frac{x-Y_1 S_1(t)}{S_0(t)} 
\label{VN}
\end{equation}
(here we drop the explicit dependencies of $V_N$ and $V$ on $t$ and $x$). First, notice that $V_N\rightarrow V$ as $N\rightarrow\infty$ in $\leb^2(\Omega)$, as an easy consequence of the following facts: $S_0^N(t)>1/2$ almost surely, for all $N\geq N_t$, $S_0^N(t)\rightarrow S_0(t)$ and $S_1^N(t)\rightarrow S_1(t)$ as $N\rightarrow\infty$ in $\leb^\infty(\Omega)$, and $Y_1\in\leb^2(\Omega)$. 

The conditions imposed on $f_{Y_0}$ imply that the Nemytskii operator $V\mapsto f_{Y_0}(V)$ is continuous from $\leb^2(\Omega)$ to $\leb^1(\Omega)$, by Lemma~\ref{lemma_nemyt}. Hence, $\lim_{N\rightarrow\infty} f_{Y_0}(V_N)\rightarrow f_{Y_0}(V)$ in $\leb^1(\Omega)$. Since $S_0^N(t)>1/2$ almost surely, for all $N\geq N_t$, and $\lim_{N\rightarrow\infty} S_0^N(t)\rightarrow S_0(t)$ in $\leb^\infty(\Omega)$, we deduce that $f_{Y_0}(V_N)/S_0^N(t)\rightarrow f_{Y_0}(V)/S_0(t)$ as $N\rightarrow\infty$ in $\leb^1(\Omega)$. 

In particular, the sequence of expectations, $f_{X^N(t)}(x)=\mathbb{E}[f_{Y_0}(V_N)/S_0^N(t)]$, converges to the density $f_{X(t)}(x)=\mathbb{E}[f_{Y_0}(V)/S_0(t)]$, which completes the proof.
\end{proof}

In Section~\ref{num_examples}, the application of Theorem~\ref{te1} will be illustrated numerically on Examples~\ref{example1}--\ref{example2}. 

\begin{remark} \label{rmk_conv_var}
Having $\lim_{N\rightarrow\infty} f_{Y_0}(V_N)/S_0^N(t)=f_{Y_0}(V)/S_0(t)$ in $\leb^1(\Omega)$ assures the convergence of the expectations. If convergence of the variances is also needed, one needs to extend the convergence to $\leb^2(\Omega)$. In this case, the boundedness condition on $f_{Y_0}$ should be $f_{Y_0}(y)\leq \alpha+\beta |y|$ (apply an analogous proof to Lemma~\ref{lemma_nemyt}).
\end{remark}

\begin{remark}[Rate of convergence of the approximating density functions] \label{rmk_rate}
Notice that, under the conditions of Theorem~\ref{the_densitat}, if $f_{Y_0}$ is Lipschitz continuous on $\mathbb{R}$ (this assumption is stronger than the hypotheses of Theorem~\ref{te1}), then $f_{X^N(t)}(x)$ converges with $N$ exponentially to $f_{X(t)}(x)$, for $t\in (t_0-\delta,t_0+\delta)$ and $x\in\mathbb{R}$. This is because the Lipschitz condition allows for estimating $|f_{X^N(t)}(x)-f_{X(t)}(x)|$ via the following inequality: 
\[ |f_{X^N(t)}(x)-f_{X(t)}(x)|\leq C_t\left((|x|+1)\|S_0^N(t)-S_0(t)\|_\infty+\|Y_1\|_2\|S_1^N(t)-S_1(t)\|_\infty\right), \]
and as discussed in Section~\ref{stoch_solution}, the Fr\"obenius method converges exponentially. In the previous expression, $C_t$ is a constant depending on $t$. Unfortunately, the exponential convergence rate is not uniform with $t$ and $x$. As $|t-t_0|$ grows, one needs to increase $N$ to maintain the accuracy. The same occurs with $|x|$, which increases the bias error $\|S_0^N(t)-S_0(t)\|_\infty$ linearly.

In general, if $f_{Y_0}$ is $\gamma$-H\"older continuous on $\mathbb{R}$ with exponent $0<\gamma\leq 1$ (the case $\gamma=1$ corresponds to Lipschitz continuity), then 
\small
\[ |f_{X^N(t)}(x)-f_{X(t)}(x)|\leq C_t\left\{\|S_0^N(t)-S_0(t)\|_\infty\right.+ \left.\left(|x|\|S_0^N(t)-S_0(t)\|_\infty+\|Y_1\|_2\|S_1^N(t)-S_1(t)\|_\infty\right)^\gamma\right\}. \]
\normalsize
The same conclusion on the convergence holds in this case.
\end{remark}

Notice that the regularity of $f_{X^N(t)}(x)$ is inherited from $f_{Y_0}(y)$. These ideas are formalized in the following theorem:

\begin{theorem} \label{te2}
Under the assumptions of Theorem~\ref{te1}, if $f_{Y_0}$ is $C^1(\mathbb{R})$ with bounded derivative on $\mathbb{R}$, then $f_{X^N(t)}(x)$ and $f_{X(t)}(x)$ are $C^1(\mathbb{R})$, with bounded derivatives, and $f_{X^N(t)}'(x)\rightarrow f_{X(t)}'(x)$ as $N\rightarrow\infty$, for each $t\in (t_0-\delta,t_0+\delta)$ and for every $x\in\mathbb{R}$.
\end{theorem}
\begin{proof}
Fix $t\in (t_0-\delta,t_0+\delta)$. The following facts permit differentiating under the expectation operator that defines $f_{X^N(t)}(x)$ and $f_{X(t)}(x)$ \cite[p.~142]{klenke}: $f_{Y_0}$ is differentiable with bounded derivative, and $S_0^N(t)>1/2$ almost surely for all $N\geq N_t$. Hence, 
\[ f_{X^N(t)}'(x)=\mathbb{E}\left[f_{Y_0}'\left(\frac{x-Y_1 S_1^N(t)}{S_0^N(t)}\right)\frac{1}{\left(S_0^N(t)\right)^2}\right], \; 
f_{X(t)}'(x)=\mathbb{E}\left[f_{Y_0}'\left(\frac{x-Y_1 S_1(t)}{S_0(t)}\right)\frac{1}{\left(S_0(t)\right)^2}\right]. \]

The continuity and boundedness conditions imposed on $f_{Y_0}'$ entail that the Nemytskii operator $V\mapsto f_{Y_0}'(V)$ is continuous from $\leb^2(\Omega)$ to $\leb^1(\Omega)$, by Lemma~\ref{lemma_nemyt}. Thereby, as in the proof of Theorem~\ref{te1}, we deduce that $\lim_{N\rightarrow\infty} f_{X^N(t)}'(x)=f_{X(t)}'(x)$, $x\in\mathbb{R}$.
\end{proof}

\begin{remark} \label{rmk_Y1}
It is important to realize that the previous theory works exchanging the role of $Y_1$ and $Y_0$. Indeed, even though $S_1(t_0)=0$, in contrast with $S_0(t_0)=1$, we do have that $\dot{S}_1(t_0)=1$. We may choose a neighborhood of $t_0$, say $(t_0-\mu,t_0+\mu)$ for certain $\mu>0$, such that $\dot{S}_1(t)>3/4$ almost surely, for $t\in (t_0-\mu,t_0+\mu)$. We know that, in the sense of $\leb^\infty(\Omega)$, $S_1(t)=\int_{t_0}^t \dot{S}_1(r)\,\dif r$. Then $|S_1(t)|>\frac34|t-t_0|=m_t$ almost surely, for $t\in (t_0-\mu,t_0+\mu)$. In particular, as $m_t>0$ for $t\in (t_0-\mu,t_0+\mu)\backslash\{t_0\}$, the previous proofs work with $Y_1$ in place of $Y_0$. The previous theoretical results may be restated in a completely analogous fashion, as
\[ f_{X(t)}(x)=\mathbb{E}\left[f_{Y_1}\left(\frac{x-Y_0 S_0(t)}{S_1(t)}\right)\frac{1}{|S_1(t)|}\right],\quad f_{X^N(t)}(x)=\mathbb{E}\left[f_{Y_1}\left(\frac{x-Y_0 S_0^N(t)}{S_1^N(t)}\right)\frac{1}{|S_1^N(t)|}\right], \]
for $t\in (t_0-\mu,t_0+\mu)\backslash\{t_0\}$. In this case, one requires $Y_1$ to have an absolutely continuous probability law, with density function $f_{Y_1}$, and to be independent of the rest of the random input parameters in~\eqref{problem}. For convergence, one imposes continuity for $f_{Y_1}$ on $\mathbb{R}$ and boundedness $f_{Y_1}(y)\leq \alpha+\beta y^2$, for certain constants $\alpha,\beta\geq0$. If $f_{Y_1}$ is Lipschitz continuous on $\mathbb{R}$, then an exponential convergence holds. Finally, if $f_{Y_1}$ is also $C^1(\mathbb{R})$ with bounded derivative on $\mathbb{R}$, then both $f_{X^N(t)}(x)$ and $f_{X(t)}(x)$ are $C^1(\mathbb{R})$, with bounded derivative, and the sequence of derivatives converges. These cases are considered in Example~\ref{example3}.
\end{remark}

The continuity condition on $\mathbb{R}$ imposed in Theorem~\ref{te1} is somewhat restrictive, as we do not allow some common probability distributions for $Y_0$ whose density function possesses discontinuity points, such as the Uniform, Exponential or general truncated distributions. Notice that this assumption in Theorem~\ref{te1} may be relaxed to almost everywhere continuity on $\mathbb{R}$, by adding absolute continuity on $Y_1$. This fact is a consequence of the continuous mapping theorem \cite[p.~7, Th.~2.3]{vaart}. Indeed, for $t\in (t_0-\min\{\delta,\mu\},t_0+\min\{\delta,\mu\})\backslash\{t_0\}$, as $|S_1(t)|>m_t>0$ almost surely and $Y_1$ is absolutely continuous, then $V=\frac{x-Y_1 S_1(t)}{S_0(t)}$ is absolutely continuous, by Lemma~\ref{convo}. Therefore, the probability that $V$ lies in the discontinuity set of $f_{Y_0}$ is $0$. This assures that $f_{Y_0}(V_N)\rightarrow f_{Y_0}(V)$ in $\leb^1(\Omega)$ as $N\rightarrow\infty$, by Lemma~\ref{lemma_nemyt}. 

The precise restatement of Theorem~\ref{te1} is the following:
 
\begin{theorem} \label{te1super}
Suppose the conditions of Theorem~\ref{the_densitat}. If $f_{Y_0}$ is almost everywhere continuous on $\mathbb{R}$, $f_{Y_0}(y)\leq \alpha+\beta y^2$ for certain constants $\alpha,\beta\geq0$, $Y_1$ is absolutely continuous, and $Y_1$ is independent of $(A,B)$, then $\lim_{N\rightarrow\infty} f_{X^N(t)}(x)=f_{X(t)}(x)$, for $t\in (t_0-\min\{\delta,\mu\},t_0+\min\{\delta,\mu\})\backslash\{t_0\}$ and for every $x\in\mathbb{R}$. 
\end{theorem}

Theorem~\ref{te1super} will be applied in Example~\ref{example4}. 
An alternative version, with $Y_1$ playing the role of $Y_0$, can be formulated following Remark~\ref{rmk_Y1}. Notice that, nowhere in our theoretical exposition, we require independence between the coefficients of $A(t)$ and $B(t)$. We do not need any assumption on their probability distributions either, which might be discrete or continuous (but always bounded).

The methodology and theory presented in the paper do not cover all situations. For instance, let us study~\eqref{problem} involving discrete uncertainties. Other situations could be analogously analyzed.

\begin{theorem} \label{te1discrete}
Suppose the conditions of Theorem~\ref{the_densitat}. Assume that the coefficients $A_0,A_1,\ldots$, $B_0,B_1,\ldots$ are deterministic constants. If $f_{Y_0}$ has at most a countable number of discontinuities on $\mathbb{R}$, $f_{Y_0}(y)\leq \alpha+\beta y^2$ for certain constants $\alpha,\beta\geq0$, and $Y_1$ is a discrete random variable, then $\lim_{N\rightarrow\infty} f_{X^N(t)}(x)=f_{X(t)}(x)$ for almost every $x\in\mathbb{R}$, for each $t\in (t_0-\delta,t_0+\delta)$.
\end{theorem}
\begin{proof}
Fix $t\in (t_0-\delta,t_0+\delta)$. Let $V_N(x)=(x-Y_1 S_1^N(t))/S_0^N(t)$, $V(x)=(x-Y_1 S_1(t))/S_0(t)$
(now we make the dependence of $V_N$ and $V$ on $x$ explicit). We know that $V_N(x)\rightarrow V(x)$ in $\leb^2(\Omega)$ as $N\rightarrow\infty$, for all $x\in\mathbb{R}$. Given the discontinuity set of $f_{Y_0}$, $\mathcal{D}_{f_{Y_0}}$, we need to justify that $\mathbb{P}[V(x)\in \mathcal{D}_{f_{Y_0}}]=0$, for almost every $x\in\mathbb{R}$. In this case, $f_{Y_0}(V_N(x))\rightarrow f_{Y_0}(V(x))$ in $\leb^1(\Omega)$ as $N\rightarrow\infty$, for almost every $x\in\mathbb{R}$.

Write $\mathcal{D}_{f_{Y_0}}=\{d_1,d_2,d_3,\ldots\}$. As $Y_1$ is a discrete random variable, its support may be expressed as $\mathcal{S}_{Y_1}=\{y_1^1,y_1^2,y_1^3,\ldots\}$. Then the support of $V(x)$ is $\mathcal{S}_{V(x)}=\{(x-y_1^jS_1(t))/S_0(t):\,j=1,2,3,\ldots\}$. The problematic $x$'s are those such that $x=y_1^jS_1(t)+d_kS_0(t)$. Let $\Lambda=\{y_1^jS_1(t)+d_kS_0(t):\,j,k=1,2,3,\ldots\}$, which is a countable set. For every $x\notin\Lambda$, $\mathbb{P}[V(x)\in \mathcal{D}_{f_{Y_0}}]=0$. As a consequence, $\lim_{N\rightarrow\infty}f_{Y_0}(V_N(x))=f_{Y_0}(V(x))$ in $\leb^1(\Omega)$, $x\notin\Lambda$, by Lemma~\ref{lemma_nemyt}. This gives $\lim_{N\rightarrow\infty} f_{X^N(t)}(x)=f_{X(t)}(x)$, $x\notin\Lambda$, and we are done.
\end{proof}

Once again, one can state a similar version with $Y_1$ playing the role of $Y_0$ (see Remark~\ref{rmk_Y1}) and working on $(t_0-\mu,t_0+\mu)\backslash\{t_0\}$, instead. Example~\ref{example5} covers this situation.

\section{Crude Monte Carlo algorithm: Computational aspects} \label{comp_aspects}

We recast the proposed methodology in the form of Algorithm~\ref{algo}, which corresponds to the case of $Y_0$ having a density, see~\eqref{dens_Xt}; following Remark~\ref{rmk_Y1}, one can exchange the role of $Y_0$ and $Y_1$ in Algorithm~\ref{algo}, provided that $Y_1$ has a density. 

By judiciously exploiting its expression in~\eqref{dens_Xnt}, $f_{X^N(t)}(x)$ can be approximated \textit{via} a Monte Carlo procedure \cite[pp.~53--54]{xiu_llibre} to evaluate the expectation: using $M$ randomly generated realizations of $Y_1$, $S_0^N(t)$ and $S_1^N(t)$, we compute the sample average of $V_N(x,t)$ in~\eqref{VN}. 
Algorithm~\ref{algo} corresponds to symbolic computations with symbolic variables $x$ and $t$ \cite{symbolic2}; it computes a function $f_X^{N,M}(x,t)$, which is a complex closed-form expression approximating $f_{X(t)}(x)$. To speed up the execution of the algorithm, numerical values of $t$ and/or $x$ may be provided.

\begin{algorithm}[hbtp!]
    \begin{algorithmic}[1]
 
   \Statex \textbf{Inputs:} $t_0$; $N$; $f_{Y_0}$; probability distribution of $A_0,\ldots,A_N$, $B_0,\ldots,B_N$, $Y_1$; and number $M$ of realizations in the classical Monte Carlo procedure.
   \State $S_{0,0}\gets 1$, $S_{0,1}\gets 0$, $S_{1,0}\gets 0$, $S_{1,1}\gets 1$ \Comment{Initial conditions}
   \State $\Sigma\gets 0$\Comment{Initialize the samples sum}
\For {$i=1,\ldots,M$} \Comment{Monte Carlo loop}
            \State Draw randomly a realization of $(A_0,\ldots,A_{N-2},B_0,\ldots,B_{N-2})$ and $Y_1$
            \For {$n=0,\ldots,N-2$} 
                \State $S_{0,n+2}\gets\frac{-1}{(n+2)(n+1)}\sum_{m=0}^n [(m+1)A_{n-m}S_{0,m+1}+B_{n-m}S_{0,m}]$
                \State $S_{1,n+2}\gets\frac{-1}{(n+2)(n+1)}\sum_{m=0}^n [(m+1)A_{n-m}S_{1,m+1}+B_{n-m}S_{1,m}]$
            \EndFor
            \State $S_0^N(t) \gets 1+\sum_{n=1}^N S_{0,n}(t-t_0)^n$ \Comment{Realization of $S_0^N(t)$}
            \State $S_1^N(t) \gets \sum_{n=1}^N S_{1,n}(t-t_0)^n$ \Comment{Realization of $S_1^N(t)$}
            \State $\Sigma \gets \Sigma + f_{Y_0}\left(\frac{x-Y_1 S_1^N(t)}{S_0^N(t)}\right)\frac{1}{|S_0^N(t)|}$ \Comment{Update the samples sum}
\EndFor
\State $f_X^{N,M}(x,t)\gets \Sigma/M$\Comment{Set sample average}
\State \textbf{Return} $f_X^{N,M}(x,t)$\Comment{Approximation of $f_{X^N(t)}(x)$}
\end{algorithmic}
    \caption{Estimation of $f_{X^N(t)}(x)$ \textit{via} a crude Monte Carlo procedure.}
        \label{algo}
\end{algorithm}

The estimation error can be split into two contributions: $f_{X(t)}(x)-f_X^{N,M}(x,t)=\theta_{N}(x,t)+\mathcal{E}_{N,M}(x,t)$. The first contribution, $\theta_N(x,t) = f_{X(t)}(x)-f_{X^N(t)}(x)$, is the bias error caused by the truncation order $N$ in the Fr\"obenius method. It is deterministic and decays exponentially as $N\rightarrow\infty$ for each $t$ and $x$ by Remark~\ref{rmk_rate}. The second contribution is the sampling error $\mathcal{E}_{N,M}(x,t) = f_{X^N(t)}(x)-f_X^{N,M}(x,t)$, due to using a finite number $M$ of samples (statistical error). This contribution is random and $\mathcal{E}_{N,M}(x,t)\rightarrow 0$ with $M$ almost surely, as a consequence of the law of large numbers. If the variance
\begin{equation} \sigma^2_N(x,t)\coloneqq\mathbb{V}\left[f_{Y_0}\left(\frac{x-Y_1 S_1^N(t)}{S_0^N(t)}\right)\frac{1}{S_0^N(t)}\right] 
 \label{sigma2N}
\end{equation}
is finite, then the asymptotic probability distribution of $\mathcal{E}_{N,M}(x,t)$ as $M\rightarrow\infty$ is, by the central limit theorem, $\text{Normal}(0,\sigma^2_{N}(x,t)/M)$. The variance $\sigma^2_N(x,t)$ tends, as $N\rightarrow\infty$, to $\sigma^2(x,t)\coloneqq \mathbb{V}[f_{Y_0}(\frac{x-Y_1 S_1(t)}{S_0(t)})\frac{1}{S_0(t)}]$ (see Remark~\ref{rmk_conv_var}). In this case, we say that the sampling error is of order $1/\sqrt{M}$, and write $\mathcal{O}(1/\sqrt{M})$. On the contrary, if $\sigma_N^2(x,t)=\infty$, then the almost sure convergence $\mathcal{E}_{N,M}(x,t)\rightarrow 0$ with $M$ remains valid, although it might be much slower and affect the approximation to $f_{X^N(t)}(x)$ severely. See the forthcoming Example~\ref{example3} for an illustration of this issue.

Even though the bias error decays very fast, the sampling error is inevitable. In numerical computations, for fixed $M$, there is usually an index $N$ from which the global error does not go down anymore because the sampling error $\mathcal{O}(1/\sqrt{M})$ becomes dominant.

Within the main loop of Algorithm~\ref{algo} (loop over the samples), we first generate one realization for each random variable $A_0,\ldots,A_{N-2}$, $B_0,\ldots,B_{N-2}$ and $Y_1$; these realizations are used to compute by recursion the corresponding realizations of $S_0^N(t)$ and $S_1^N(t)$. In our implementation, this procedure is more efficient than expressing first $S_0^N(t)$ and $S_1^N(t)$ recursively in terms of symbolic variables $A_0,\ldots,A_{N-2}$, $B_0,\ldots,B_{N-2}$ and $Y_1$, and then evaluate for the realizations of $A_0,\ldots,A_{N-2}$, $B_0,\ldots,B_{N-2}$ and $Y_1$. This is due to the excessive complexity of the symbolic expressions of $S_0^N(t)$ and $S_1^N(t)$, which makes the computational time of their evaluation for specific realizations prohibitively large.

The computational complexity of Algorithm~\ref{algo} is at most $\mathcal{O}(MN^2)$ (the nested loop over $n$ demands $\sum_{n=0}^{N-2}\mathcal{O}(n)=\mathcal{O}(N^2)$ operations in general). As we show in the following Section~\ref{num_examples}, the implemented algorithm is certainly applicable and suitable for stochastic computations.

By taking $M=\mathcal{O}(1/\epsilon^2)$, the variance of the statistical error is $\mathbb{V}[\mathcal{E}_{N,M}(x,t)]=\mathcal{O}(\epsilon^2)$ (assuming the variance in~\eqref{sigma2N} finite). Under exponential convergence of the bias, by picking $N=\mathcal{O}(\log(1/\epsilon))+\mathcal{O}(1)$ the bias error is $|\theta_N(x,t)|=\mathcal{O}(\epsilon)$. Then the root mean square error of the algorithm is
$\|f_{X(t)}(x)-f_X^{N,M}(x,t)\|_2=\sqrt{\theta_N(x,t)^2+\mathbb{V}[\mathcal{E}_{N,M}(x,t)]}=\mathcal{O}(\epsilon)$,
with a computational complexity $\mathcal{O}(MN^2)=\mathcal{O}\left(\epsilon^{-2}\log^2\epsilon\right)$.

The complexity of Algorithm~\ref{algo} is significantly reduced if $A(t)$ and $B(t)$ are random polynomials, instead of infinite series. Suppose for instance that $A_j=0$ and $B_j=0$, for $j\geq N_0-1$. Then, within the nested loop over $n$, we actually sum $N_0$ terms, instead of $n$ terms. Therefore, the nested loop demands $N_0\mathcal{O}(N)=\mathcal{O}(N)$ operations. The whole algorithm then requires $\mathcal{O}(MN)$ operations only. If we take $M=\mathcal{O}(1/\epsilon^2)$ and $N=\mathcal{O}(\log(1/\epsilon))+\mathcal{O}(1)$ to ensure a root mean square error of order $\epsilon$, the computational complexity becomes $\mathcal{O}(MN)=\mathcal{O}(\epsilon^{-2}\log(\epsilon^{-1}))$. Notice that $0<\log(\epsilon^{-1})<(\log(\epsilon^{-1}))^2=\log^2\epsilon$, for $0<\epsilon<\e^{-1}$, so the complexity is lessened.

In the case in which $A(t)$ and $B(t)$ are deterministic expansions, the loop over $n$ and the assignments for $S_0^N(t)$ and $S_1^N(t)$ may be run once for all at the beginning of the algorithm and before the loop over the samples. 
The computational complexity then reduces even more to $\mathcal{O}(M)+\mathcal{O}(N^2)$ operations and the global cost is generally dominated by the sampling. To guarantee a global root mean square error of order $\epsilon$ with $M=\mathcal{O}(1/\epsilon^2)$ and $N=\mathcal{O}(\log(1/\epsilon))+\mathcal{O}(1)$, the computational complexity becomes $\mathcal{O}(\epsilon^{-2})+\mathcal{O}(\log^2\epsilon)=\mathcal{O}(\epsilon^{-2})$. This scenario allows for increasing $M$ and obtaining more accurate results by improving the statistical convergence. If $A(t)$ and $B(t)$ are simply deterministic polynomials, then the overall cost reduces further to $\mathcal{O}(M)+\mathcal{O}(N)$ operations, which yields in the end $\mathcal{O}(\epsilon^{-2})$ calculations.

In the view of computational applications, an important drawback of our exposition is the lack of awareness on the specific values of $\delta$ and $\mu$, which are necessary to prove the theoretical convergence. Given any $t$, one can apply Algorithm~\ref{algo} and check the convergence of the estimator with $M$ and $N$. The results can be validated using other stochastic methods and using statistics based on the estimated density. Notice that, in Algorithm~\ref{algo}, we have put $|S_0^N(t)|$ instead of $S_0^N(t)$. Even though we assume that $S_0^N(t)>0$ almost surely, for $t\in (t_0-\delta,t_0+\delta)$ and $N\geq N_t$, the absolute value ensures positiveness in numerical applications even if $|t-t_0|\geq\delta$.

\section{Crude Monte Carlo algorithm: Numerical examples} \label{num_examples}
In this section, we numerically illustrate our theoretical findings, using the crude Monte Carlo Algorithm~\ref{algo} to estimate the density of the solution to~\eqref{problem}. Several cases, differing by the probability distributions of the random input coefficients, are considered to cover a large class of situations and show the broad applicability of our theory. 

In each of these examples, we first check that the necessary theoretical conditions hold; we then estimate the density function $f_{X^N(t)}(x)$ for several increasing values of $N$ to highlight the convergence toward $f_{X(t)}(x)$. To this end, we employ the  Monte Carlo sampling procedure outlined in Algorithm~\ref{algo}. 

The theoretical results of Section~\ref{sec_dens} motivate the structure and the choice of the following five examples. In Example~\ref{example1}, we address the case where $A(t)$ and $B(t)$ are random polynomials; while Example~\ref{example2} concerns infinite expansions and infinite dimensionality. These first two examples showcase the applicability of Theorem~\ref{te1}. Example~\ref{example3} is designed to highlight Remark~\ref{rmk_Y1}. Up to this example, $f_{Y_0}$ or $f_{Y_1}$ are continuous on the whole real line. In contrast, Example~\ref{example4} considers experiments with $f_{Y_0}$ possessing discontinuity points, thus evoking Theorem~\ref{te1super}. Finally, Example~\ref{example5} considers the case where $A(t)$ and $B(t)$ are deterministic, so that Theorem~\ref{te1discrete} applies.

The implementations and computations are performed with Mathematica\textsuperscript{\tiny\textregistered}, version 11.2 \cite{mathematica}, owing to its capability to handle both symbolic and numeric computations. 
In general, Algorithm~\ref{algo} is applied with $M=20,000$ samples, as beyond this limit, the computational burden is becoming massive. 
The output function $f_X^{N,M}(x,t)$ is handled symbolically on $t$ and $x$. To simplify the notations,
we refer to the Monte Carlo estimate $f_X^{N,M}(x,t)$ as $\hat f_{X^N(t)}(x)$. 
We recall that the estimate $\hat f_{X^N(t)}(x)$ has two sources of error: bias and sampling. Although the bias error decays very fast (exponentially under the conditions of Remark~\ref{rmk_rate}), the sampling error is unavoidable and at least of order $\mathcal{O}(1/\sqrt{M})$. 

In each one of the following examples, we perform a complete analysis of the errors. As the exact density function $f_{X(t)}(x)$ is not known, we first analyze differences in consecutive (in $N$) estimates, both pointwise, using 
\begin{equation}\label{diff_loc}
\delta\epsilon^N(x,t)\coloneqq|\hat f_{X^{N+1}(t)}(x)-\hat f_{X^N(t)}(x)|, 
\end{equation}
and globally, using the norm 
\begin{equation}\label{diff_norm}
    \Delta\epsilon^N(t)\coloneqq\|\hat f_{X^{N+1}(t)}-\hat f_{X^N(t)}\|_{\leb^1(\mathbb{R})}. 
\end{equation}
As successive differences do not directly characterize the error, we also report 
\begin{equation}\label{err_estim}
    E^N(t)\coloneqq\|\hat f_{X^{L}(t)}-\hat f_{X^N(t)}\|_{\leb^1(\mathbb{R})} 
\end{equation}
for some pre-fixed $L\gg1$, selected such that $\hat f_{X^{L}(t)}$ plays the role of a bias-free estimate of the function $f_{X(t)}$. 
We set $L=30$ in the following. 
The $\leb^1(\mathbb{R})$ norms are computed by direct numerical integration, using a standard quadrature rule (standard \verb|NIntegrate| routine in Mathematica\textsuperscript{\tiny\textregistered}).

\begin{example} \label{example1} \normalfont
We start with the stochastic problem~\eqref{problem} where both $A(t)$ and $B(t)$ are random polynomials of degree $1$: $A(t) = A_0+A_1t$, and $B(t) = B_0+B_1t$.
We set $A_0=4$, $A_1\sim\text{Uniform}(0,1)$, $B_0\sim\text{Gamma}(2,2)$, $B_1\sim\text{Bernoulli}(0.35)$, $Y_0\sim\text{Normal}(2,1)$ and $Y_1\sim\text{Poisson}(2)$, all being independent random variables. In order for the hypotheses of Theorem~\ref{nostre} and Theorem~\ref{te1} to be satisfied, the Gamma distribution is truncated. For the Gamma distribution with shape and rate $2$, it can be checked that the interval $[0,4]$ contains approximately $99.7\%$ of the probability, so we actually consider $B_0\sim\text{Gamma}(2,2)|_{[0,4]}$.

By Theorem~\ref{nostre}, the unique mean square solution to the problem can be written as a random power series $X(t)=\sum_{n=0}^\infty X_n t^n$ that is mean square convergent for all $t\in\mathbb{R}$. With Theorem~\ref{te1}, we approximate pointwise the probability density function $f_{X(t)}(x)$ with $\hat f_{X^N(t)}(x)$, $N\geq0$, and use Algorithm~\ref{algo} taking advantage from the fact that $A(t)$ and $B(t)$ are random polynomials and not infinite expansions. 
We consider times $t=0.5$, 1 and 1.5. In Figure~\ref{figure1} we present the graphs of $\hat f_{X^N(t)}(x)$ at the corresponding times. Observe that the estimates are smooth, due to the regularity of the initial density $f_{Y_0}$, see Theorem~\ref{te2}. Observe also that, as $N$ increases, the density functions become closer, reflecting the theoretical convergence. The convergence is made clear in the corresponding successive differences $\delta\epsilon^N(x,t)$ (see~\eqref{diff_loc}) reported in Figure~\ref{figure1error}. 
Table~\ref{table1error} presents the $\leb^1(\mathbb{R})$ norms of the successive differences, $\Delta\epsilon^N(t)$ (see~\eqref{diff_norm}).

\begin{figure}[hbt!]
  \begin{center}
    \includegraphics[width=0.32\textwidth]{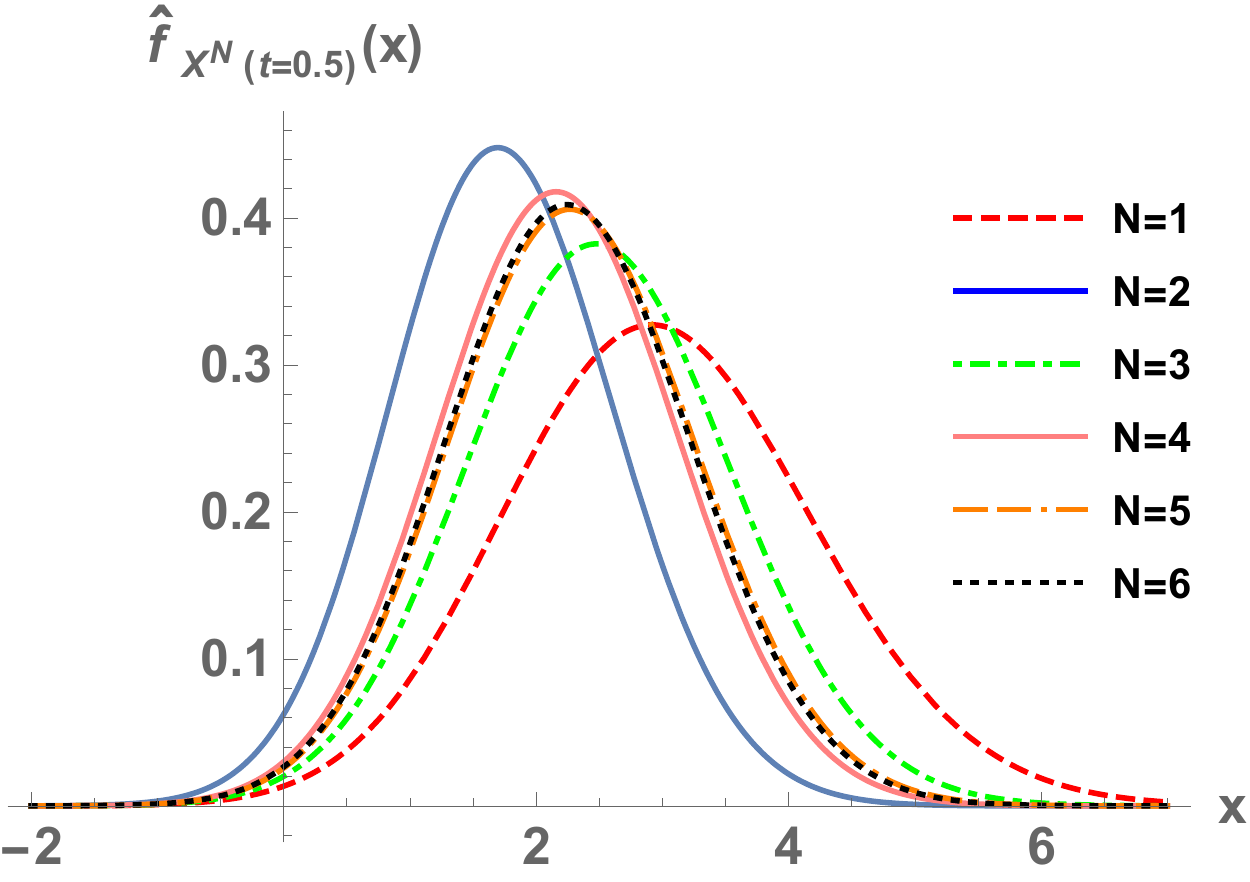}
    \includegraphics[width=0.32\textwidth]{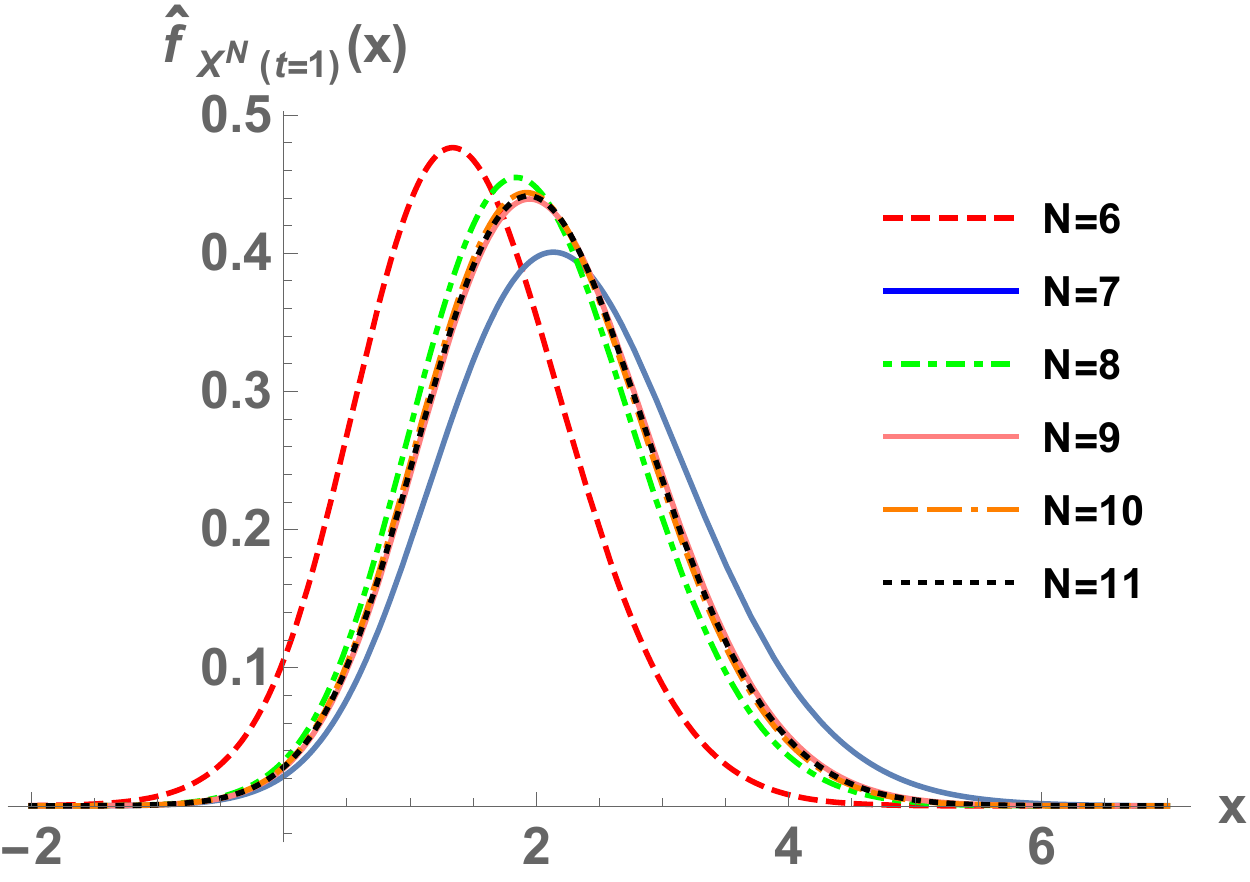}
    \includegraphics[width=0.32\textwidth]{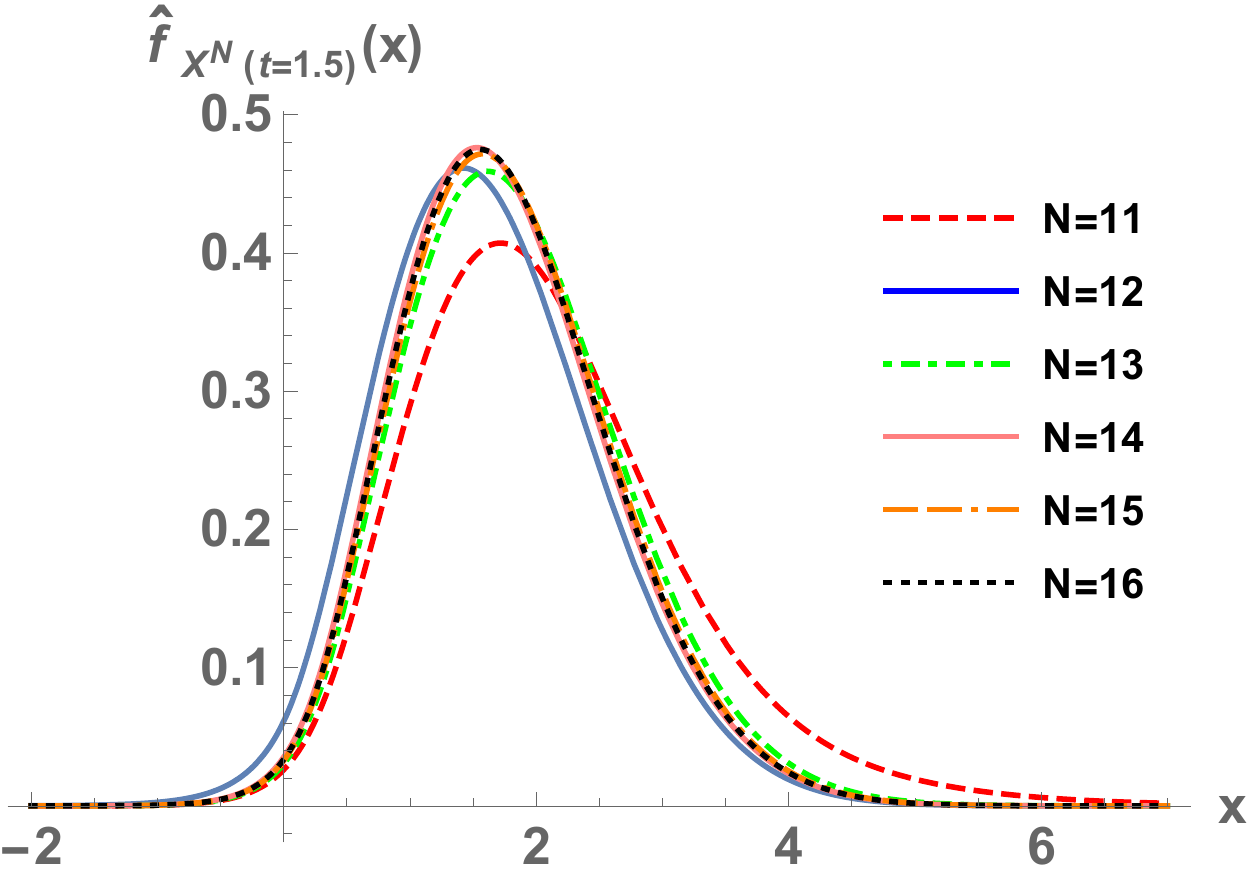}
    \caption{Graphical representations of the Monte Carlo estimates $\hat f_{X^N(t)} (x)$ 
    at $t=0.5$ (left), $t=1$ (center) and $t=1.5$ (right), 
    with orders of truncation $N=1\textendash6$, $N=6\textendash11$ and $N=11\textendash16$, respectively. 
    This figure corresponds to Example~\ref{example1}.}
        \label{figure1}
    \end{center}
  \end{figure}
    
    \begin{figure}[hbt!]
  \begin{center}
    \includegraphics[width=0.32\textwidth]{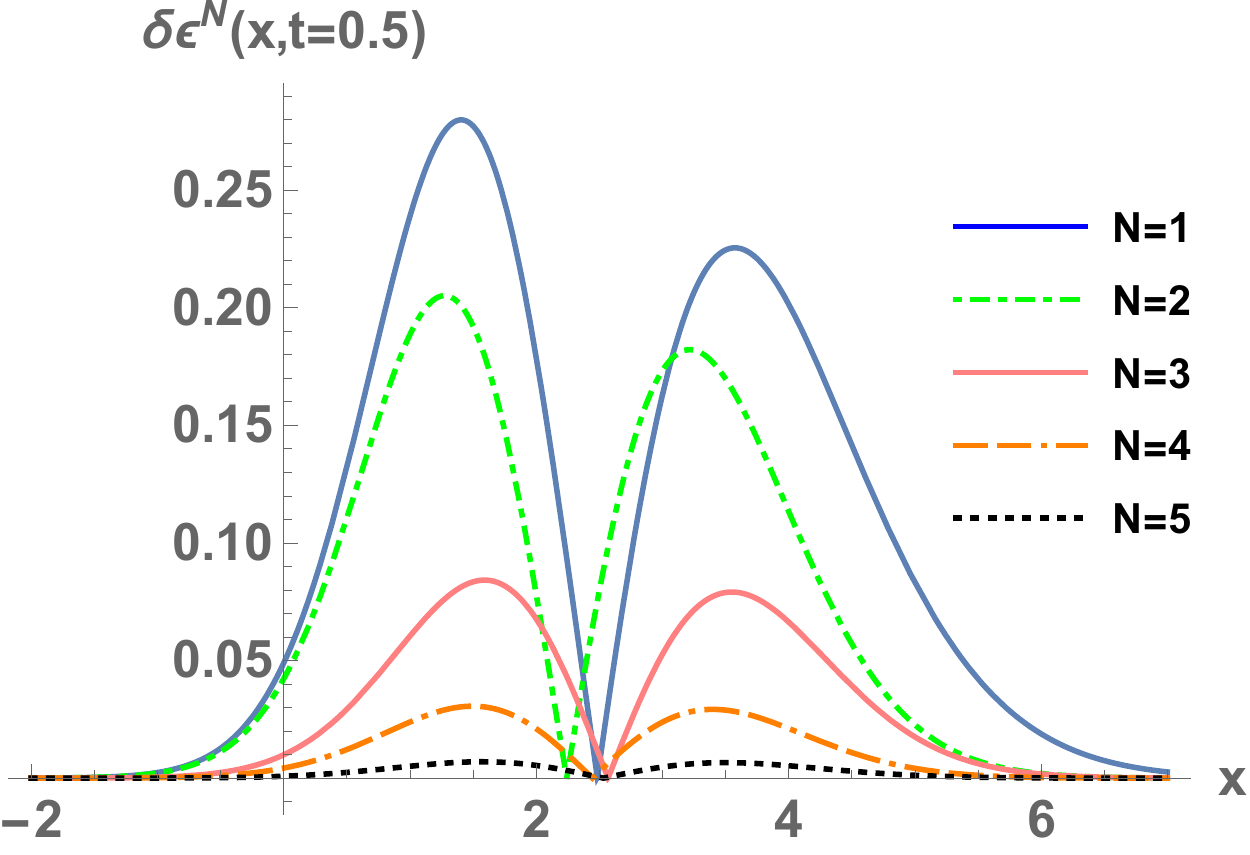}
    \includegraphics[width=0.32\textwidth]{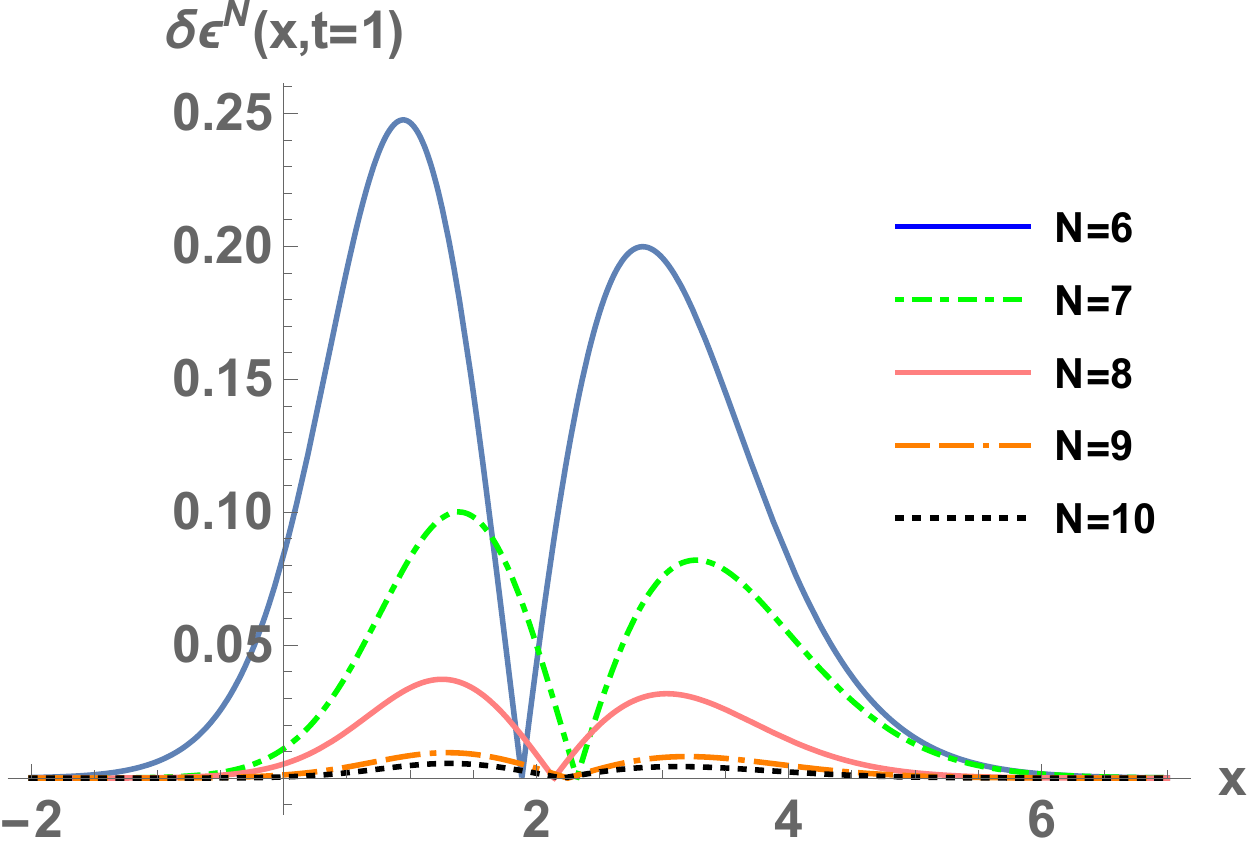}
    \includegraphics[width=0.32\textwidth]{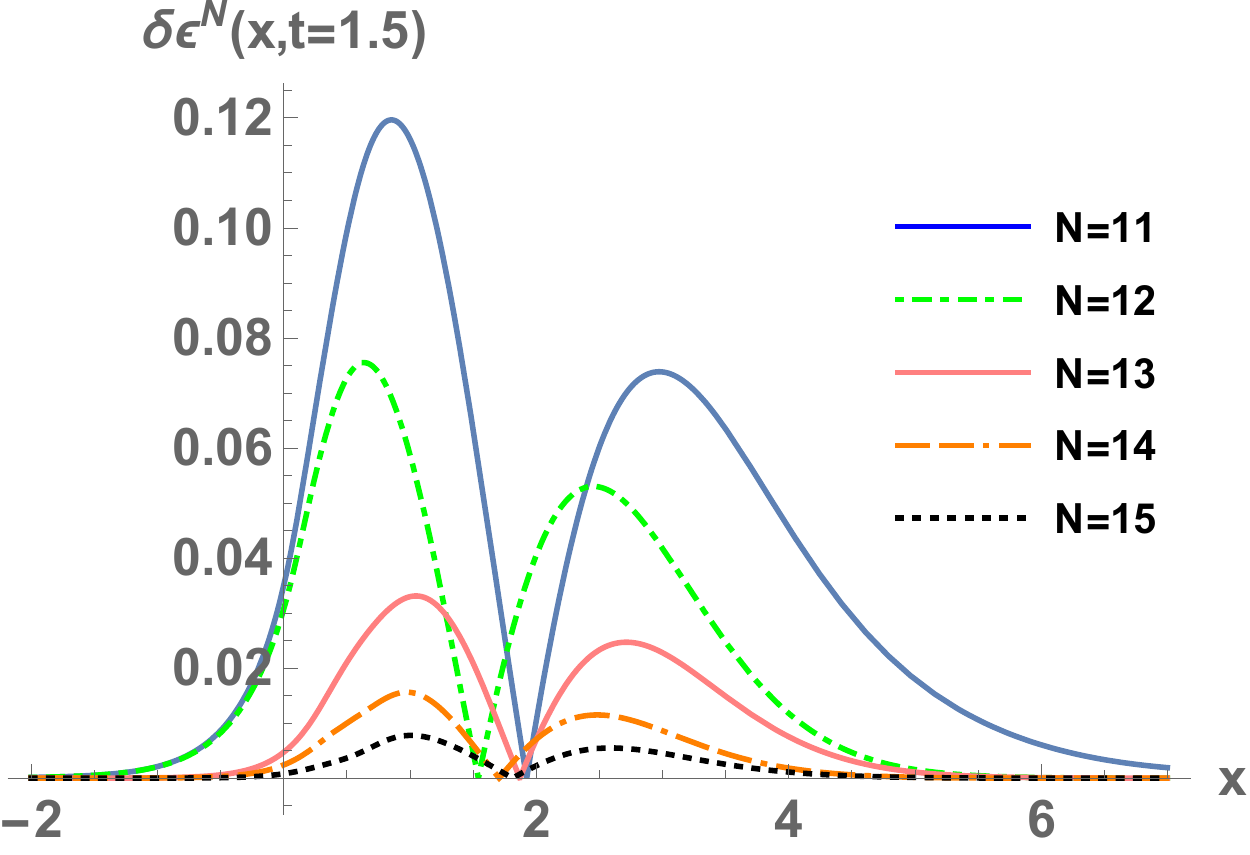}
    \caption{
    Differences in consecutive estimates $\delta\epsilon^N(x,t)$  (see~\eqref{diff_loc})
    at $t=0.5$ (left), $t=1$ (center) and $t=1.5$ (right), 
    with orders of truncation $N=1\textendash5$, $N=6\textendash10$ and $N=11\textendash15$, respectively. 
    This figure corresponds to Example~\ref{example1}.}
        \label{figure1error}
    \end{center}
  \end{figure}

\begin{table}[hbt!]
\footnotesize
\begin{center}
\begin{tabular}{|c|ccccc|}\hline
 & $N=1$ & $N=2$ & $N=3$ & $N=4$ & $N=5$  \\ 
$t=0.5$ & $0.903091$ & $0.622968$ & $0.270690$ & $0.0923362$ & $0.0178834$  \\
\hline
$$ & $N=6$ & $N=7$ & $N=8$ & $N=9$ & $N=10$  \\ 
$t=1$ & $0.691809$ & $0.263246$ & $0.0912177$ & $0.0345686$ & $0.026688$  
\\ \hline
$$ & $N=11$ & $N=12$ & $N=13$ & $N=14$ & $N=15$  \\
$t=1.5$ & $0.348643$ & $0.180075$ & $0.0721679$ & $0.0320314$ & $0.0198364$  \\ \hline
\end{tabular}
\caption{Norm $\Delta \epsilon^N(t)$ of differences in consecutive estimates (see~\eqref{diff_norm})
for different times $t$ and truncation orders $N$. 
This table corresponds to Example~\ref{example1}.}
\label{table1error}
\end{center}
\end{table}

The left plot in Figure~\ref{figure1error2} reports (in log-scale) the error estimate $E^N(t)$ (see \eqref{err_estim}), for $t=0.5$, $t=1$ and $t=1.5$. 
From the plot, it is clear that there is an index $N$ from which the error does not go down anymore, because of the sampling error (recall that we used a fixed number of samples $M=20,000$). 
Notice also that, as $|t-t_0|=|t|$ gets larger, a higher order of truncation $N$ is required to enhance the approximations of 
$f_{X(t)}(x)$. 
In the right plot of Figure~\ref{figure1error2}, we report the error estimate, $E^N(t)$, as a function of the consecutive difference, $\Delta\epsilon^N(t)$, for $t=0.5$, $t=1$ and $t=1.5$. We also plot a regression line through the data to reflect the exponential relationship between $E^N(t)$ and $\Delta \epsilon^N(t)$, at a given $t$,
\begin{equation} 
  \log E^N(t)\approx \log\beta(t)+\alpha(t)\log\Delta\epsilon^N(t).
 \label{colin1}
\end{equation}
There are three regression lines, one for each time $t$. We observe a strong linear relation with $N$ between the errors and the successive differences in log-scale, with slope $\alpha(t)$ being approximately~$1$, at least up to the truncation order at which the sampling error becomes dominant. This finding suggests that it is possible to estimate the norm of the bias error, $\|\theta_N(\cdot,t)\|_{\leb^1(\mathbb{R})}$, from the norm of the successive differences $\Delta\epsilon^N(t)$, provided that $M$ is large enough, and choose $N$ according to the targeted accuracy.

\begin{figure}[hbt!]
  \begin{center}
    \includegraphics[width=0.36\textwidth]{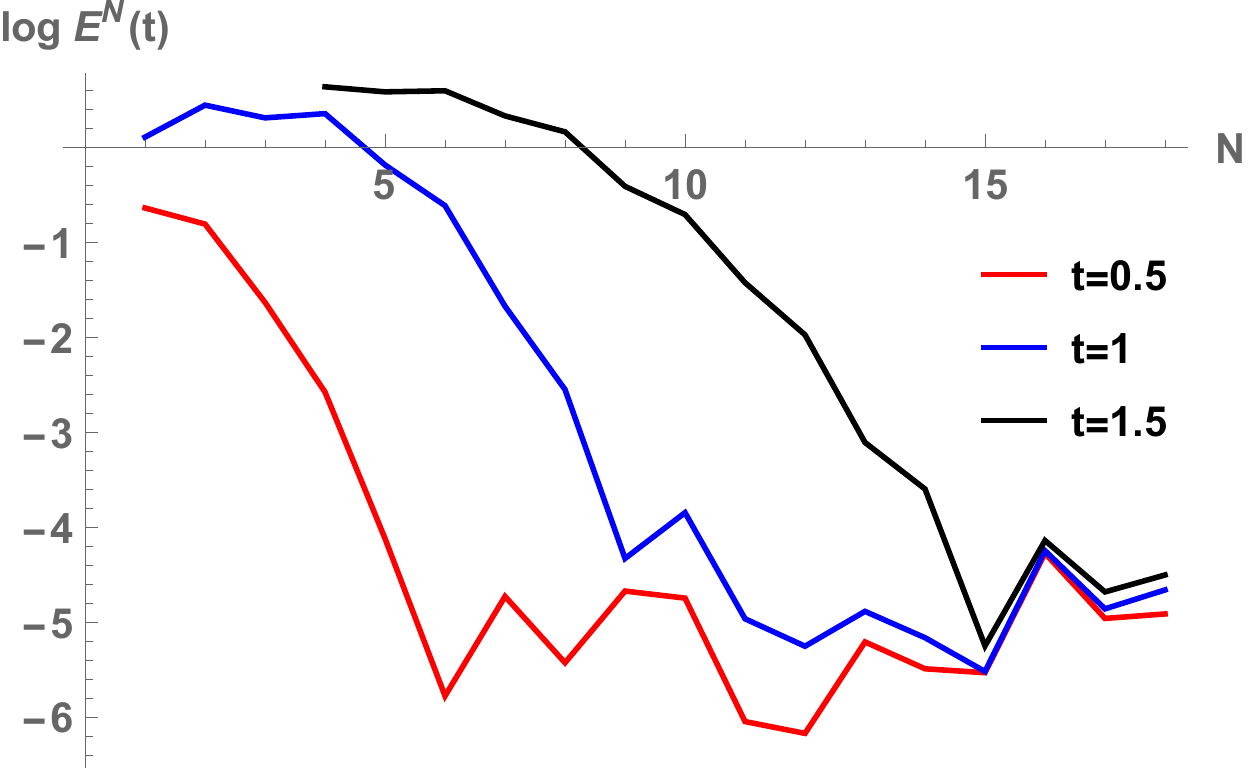}
    \includegraphics[width=0.36\textwidth]{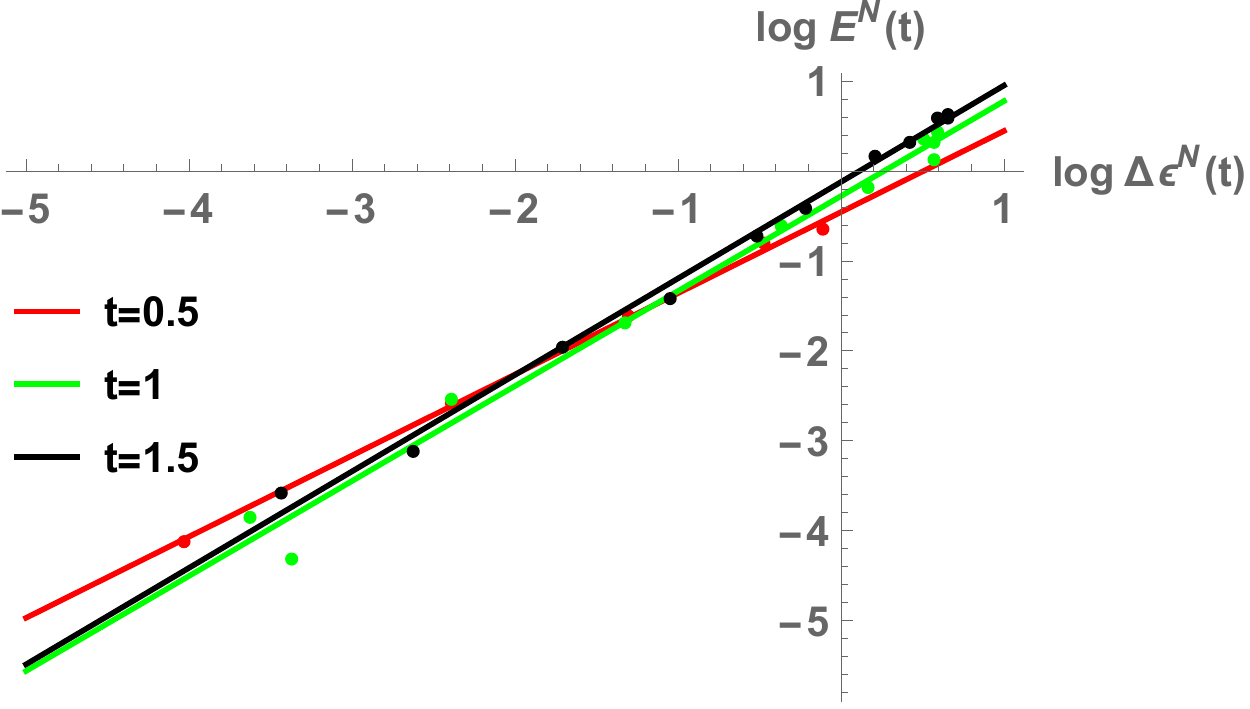}
    \caption{Left: error $E^N(t)$ (see \eqref{err_estim}), for different times as indicated.
    Right: relation between $\log E^N(t)$ and $\log \Delta\epsilon^N(t)$, 
    for $t=0.5$, $t=1$ and $t=1.5$. Also reported are linear regressions.
    This figure corresponds to Example~\ref{example1}.}
        \label{figure1error2}
    \end{center}
  \end{figure}
	
The decay of the sampling error as the inverse of the square root of the number of realizations can be documented. Let us fix a truncation order $N=20$. Define the Monte Carlo error for $P$ realizations at time $t$ as
\begin{equation}
 \text{MCE}^P(t)\coloneqq\| f_X^{N,P}(\cdot,t)-f_X^{N,M}(\cdot,t)\|_{\leb^1(\mathbb{R})}.
 \label{mcerr}
\end{equation}
Here $M=20,000$ is the maximum amount of realizations. Figure~\ref{figmcerr} shows $\text{MCE}^P(t)$ defined by~\eqref{mcerr} for $t=0.5$, $1$ and $1.5$, using logarithm scale for both axes. We have considered nested samples of length $P\in\{100,200,400,800,1600,3200,6400,12800\}$ growing geometrically. We observe the decline of the sampling error as $P$ grows. Obviously, the error $\text{MCE}^P(t)$ depends on the random numbers generated. Theoretically, the mean error $\mathbb{E}[\text{MCE}^P(t)]$ decreases at rate given by the inverse of the square root of the sample length $P$. In the figure, the line corresponding to the $1/\sqrt{P}$ decay has been drawn, for comparison with $\text{MCE}^P(t)$.

\begin{figure}[hbt!]
  \begin{center}
    \includegraphics[width=0.6\textwidth]{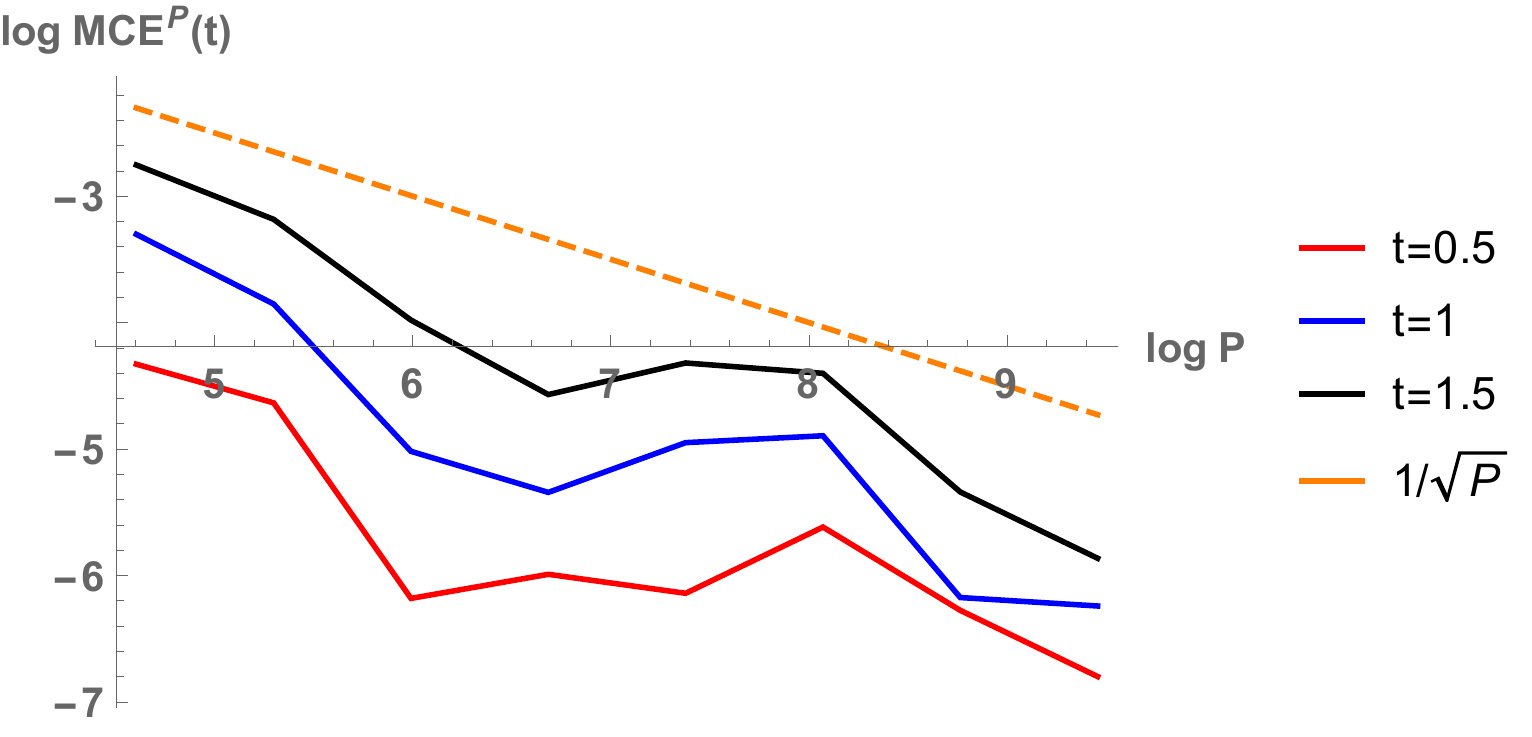}
    \caption{Sampling error~\eqref{mcerr} with the number of realizations $P$, for different times as indicated.
    This figure corresponds to Example~\ref{example1}.}
        \label{figmcerr}
    \end{center}
  \end{figure}

\end{example}


\begin{example} \label{example2} \normalfont
In the second example, we consider problem~\eqref{problem} with $A(t)$ and $B(t)$ having infinite expansions with coefficients $A_n\sim\text{Beta}(11,15)$ for $n\geq0$, $B_0=0$, $B_n=1/n^2$ for $n\geq1$, while $Y_0\sim f_{Y_0}(y)=\frac{\sqrt{2}}{\pi(1+y^4)}$ ($-\infty<y<\infty$) and $Y_1\sim \text{Poisson}(2)$. All these random quantities are assumed to be independent. 
The power series of $A(t)$ and $B(t)$ converge on $(-1,1)$ (that is for $r=1$), so the mean square solution $X(t)=\sum_{n=0}^\infty X_n t^n$ given by Theorem~\ref{nostre} is defined on $(-1,1)$. 
Theorem~\ref{te1} allows for approximating $f_{X(t)}(x)$ with $\hat f_{X^N(t)}(x)$, $N\geq0$. 

Figure~\ref{figure2} shows graphical representations of $\hat f_{X^N(t)}(x)$ for times $t=0.25$, $0.75$ and $0.99$, with orders of truncation $N=1\textendash 5$. The evident regularity of $\hat f_{X^N(t)}(x)$ is inherited from the smoothness of the density $f_{Y_0}$, by Theorem~\ref{te2}. 

\begin{figure}[hbt!]
  \begin{center}
    \includegraphics[width=0.32\textwidth]{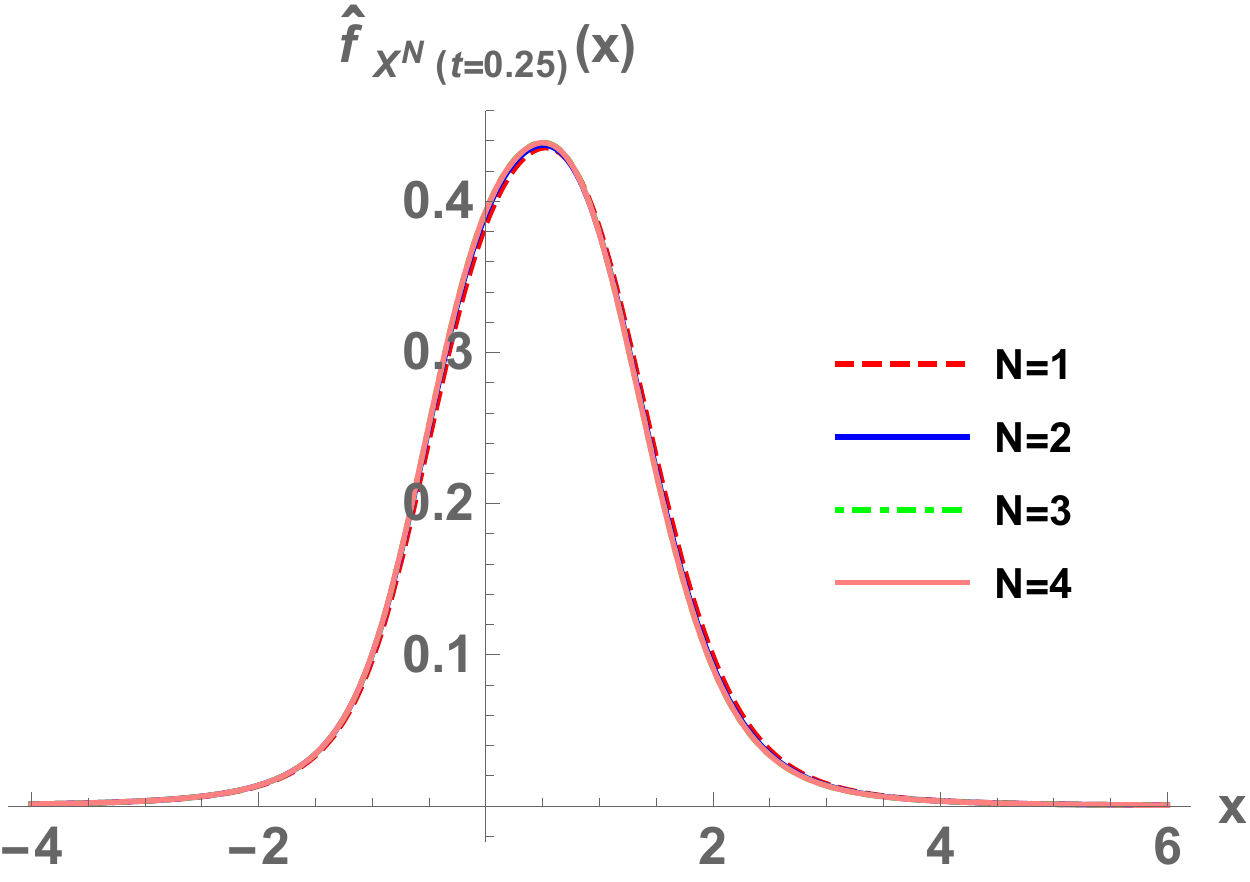}
    \includegraphics[width=0.32\textwidth]{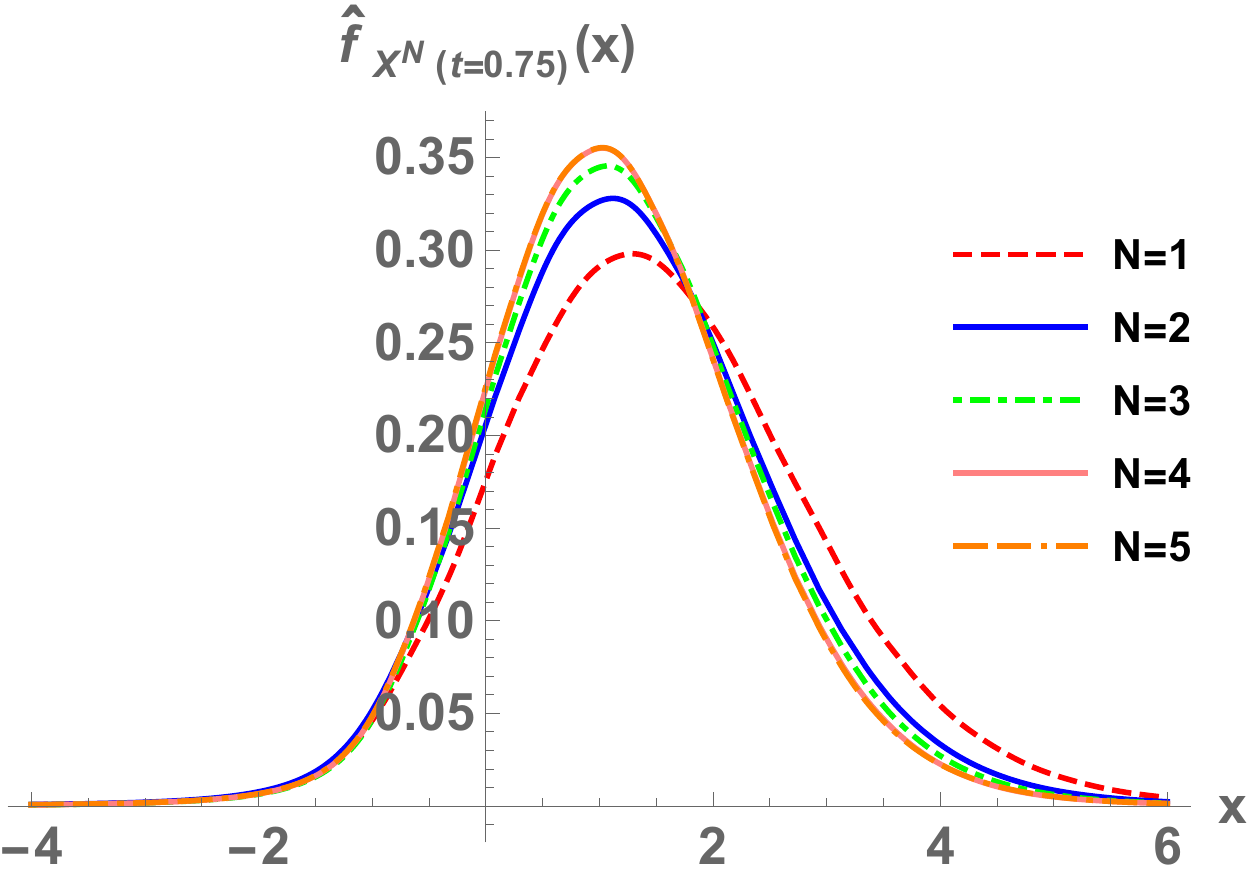}
    \includegraphics[width=0.32\textwidth]{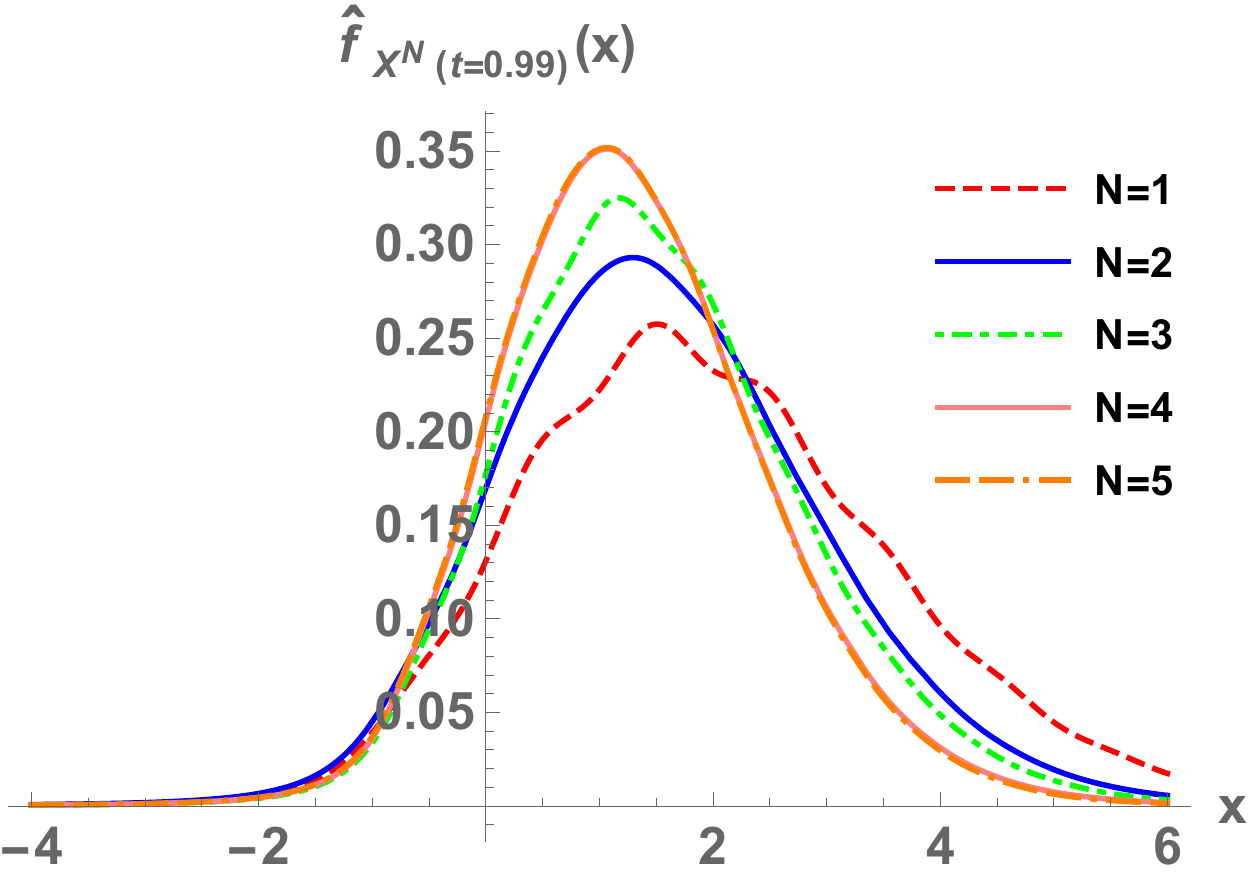}
    \caption{Graphical representations of the Monte Carlo estimates $\hat f_{X^N(t)} (x)$ at $t=0.25$ (left),  $t=0.75$ (center) and $t=0.99$ (right), with orders of truncation $N$ as indicated. 
    This figure corresponds to Example~\ref{example2}.}
        \label{figure2}
    \end{center}
  \end{figure}

To better assess the convergence, Figure~\ref{figure2error} shows the successive differences $\delta\epsilon^N(x,t)$ defined in~\eqref{diff_loc} at the same times as in Figure~\ref{figure2}; these differences are decreasing to $0$ pointwise as theoretically expected, see Theorem~\ref{te1}. 
As pointwise convergence of densities implies $\leb^1(\mathbb{R})$ convergence, we report in Table~\ref{table2error} the consecutive norms $\Delta\epsilon^N(t)$ defined by~\eqref{diff_norm}. The norms decay, albeit not monotonically; for instance, when $t=0.25$ the difference is larger for $N=4$ than in $N=3$. 

\begin{figure}[hbt!]
  \begin{center}
    \includegraphics[width=0.32\textwidth]{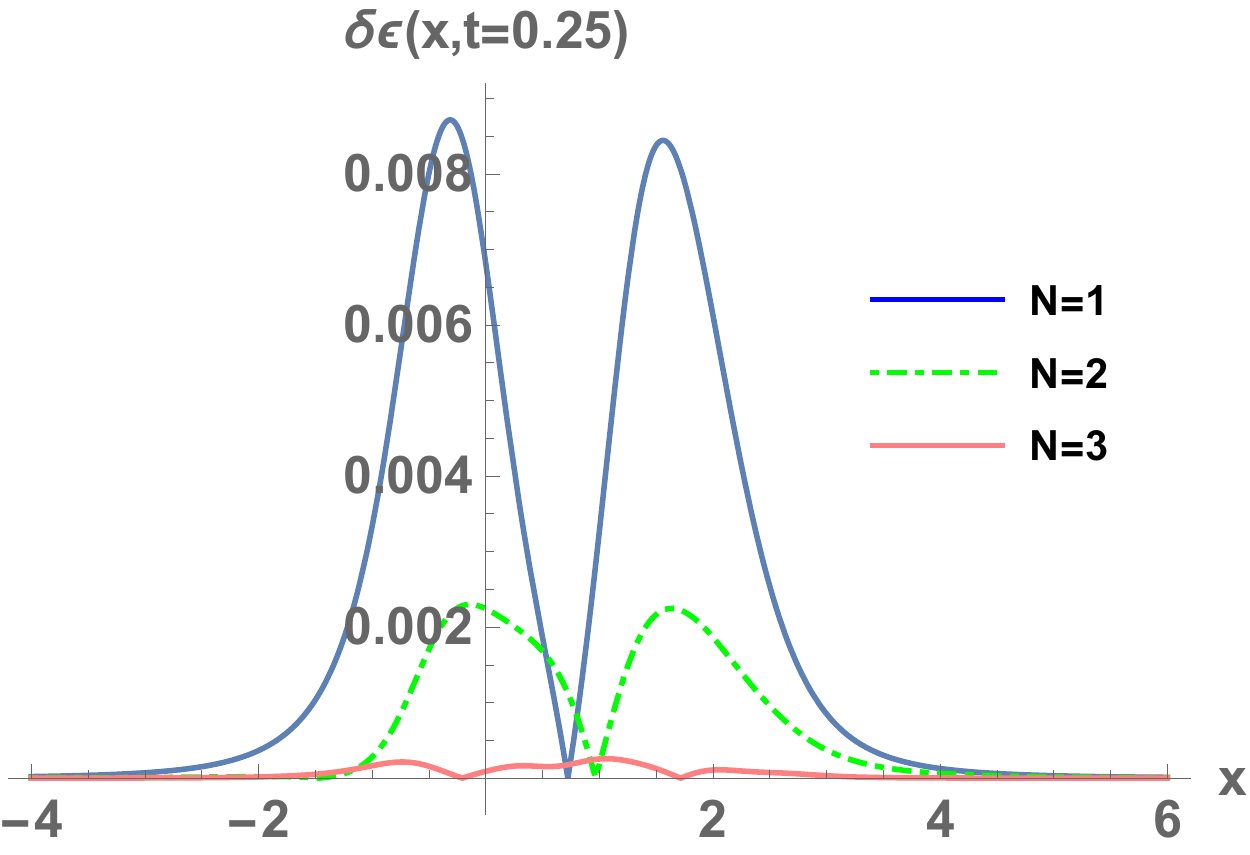}
    \includegraphics[width=0.32\textwidth]{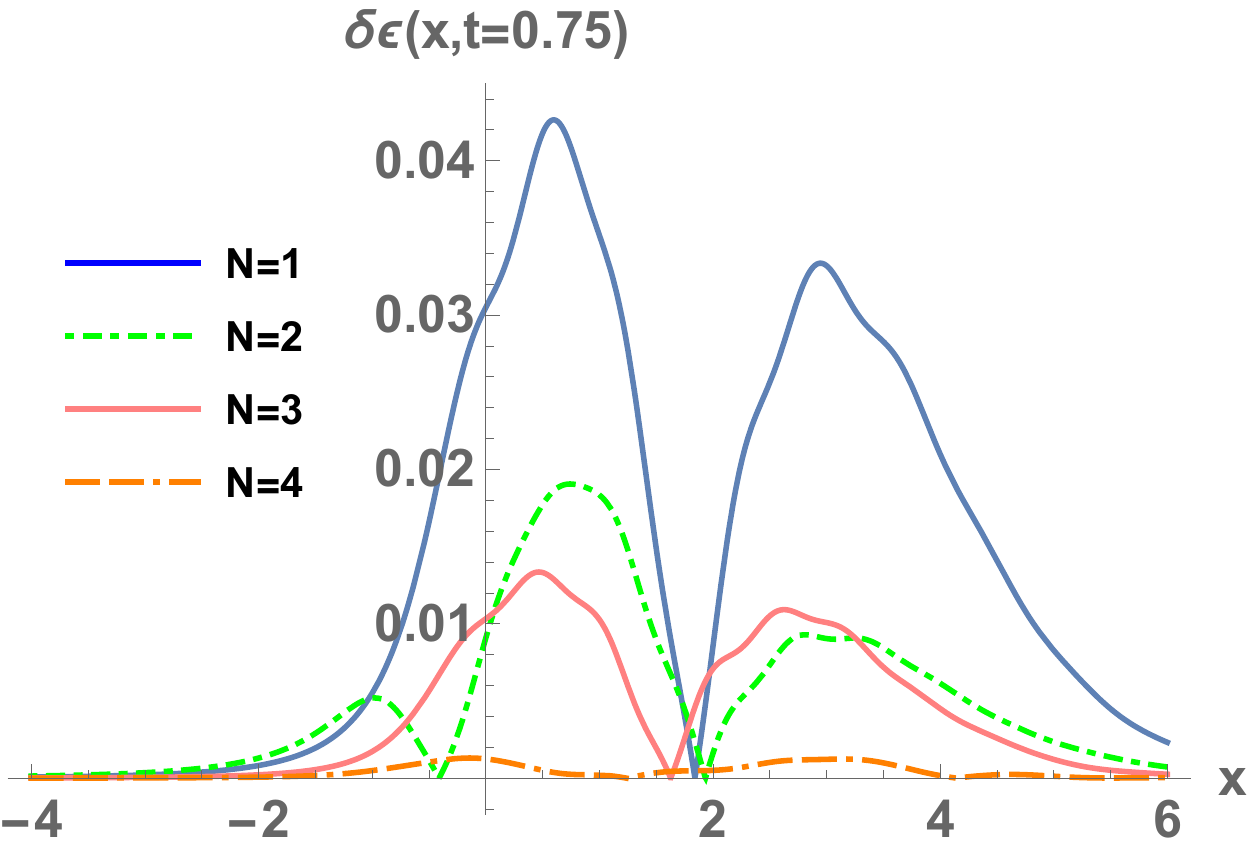}
    \includegraphics[width=0.32\textwidth]{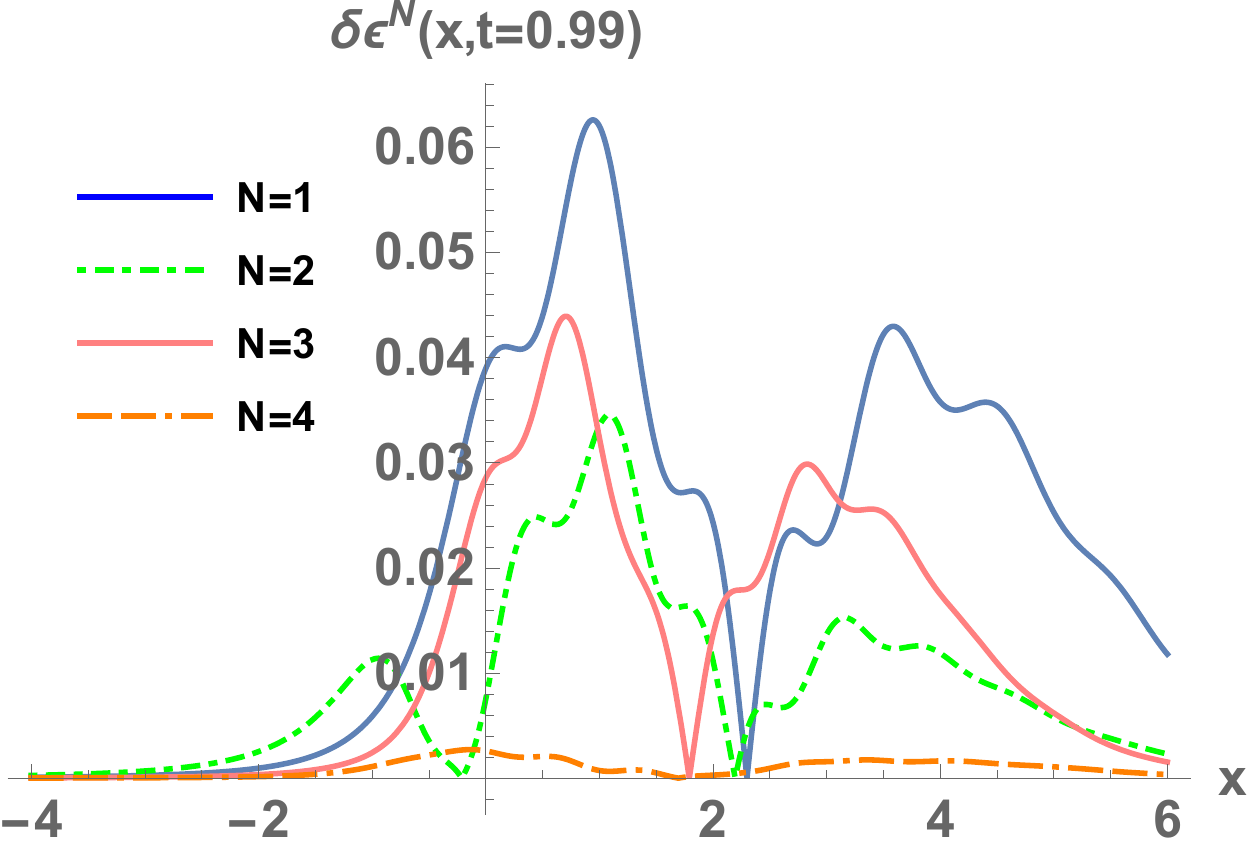}
    \caption{Differences in consecutive estimates $\delta\epsilon^N(x,t)$  (see~\eqref{diff_loc}) at $t=0.25$ (left),  $t=0.75$ (center) and $t=0.99$ (right), and for orders of truncation as indicated. 
    The plots correspond to Example~\ref{example2}.}
        \label{figure2error}
    \end{center}
  \end{figure}
    
\begin{table}[hbt!]
\footnotesize
\begin{center}
\begin{tabular}{|c|cccc|} \hline
 & $N=1$ & $N=2$ & $N=3$ & $N=4$  \\ 
$t=0.25$ & $0.0215530$ & $0.00607417$ & $0.000600201$ & $0.00167170$  \\ 
$t=0.75$ & $0.147952$ & $0.0545970$ & $0.0436801$ & $0.00419704$  \\
$t=0.99$ & $0.225868$ & $0.0945261$ & $0.127495$ & $0.00985133$  \\ \hline
\end{tabular}
\caption{Norm $\Delta \epsilon^N(t)$ of differences in consecutive estimates (see~\eqref{diff_norm}) for different times $t$ and truncation orders $N$. This table corresponds to Example~\ref{example2}.}
\label{table2error}
\end{center}
\end{table}
    
Figure~\ref{figure2error2} reports in the left plot the error estimates $\log E^N(t)$ defined in~\eqref{err_estim}. We see that the errors decrease quickly before stagnating because of the sampling error. 
This example, despite being more complex than the previous one in Example~\ref{example1}, in terms of dimensionality, requires smaller orders $N$, since for $t\in (-1,1)$ we have $|t-t_0|=|t|<1$, which implies $|t-t_0|^n\stackrel{n\rightarrow\infty}{\longrightarrow}0$.
The right plot of Figure~\ref{figure2error2} aims at showing the relation between the errors $E^N(t)$ and the successive differences $\Delta\epsilon^N(t)$. Specifically, for the times $t$ shown, a collinearity is found in log-scale through the model~(\ref{colin1}). In other words, the decay pattern of the consecutive differences characterizes the convergence of the global error as long as the bias error dominates the sampling error.

\begin{figure}[hbt!]
  \begin{center}
    \includegraphics[width=0.36\textwidth]{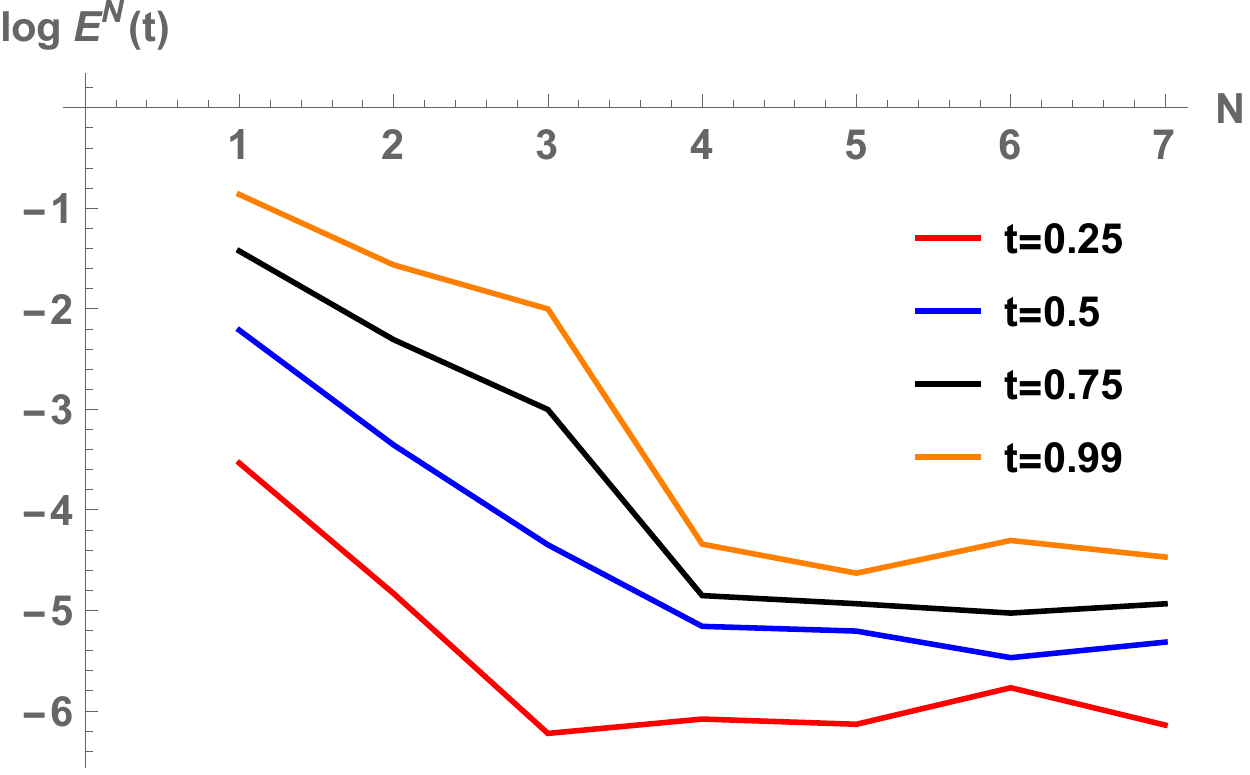}
    \includegraphics[width=0.36\textwidth]{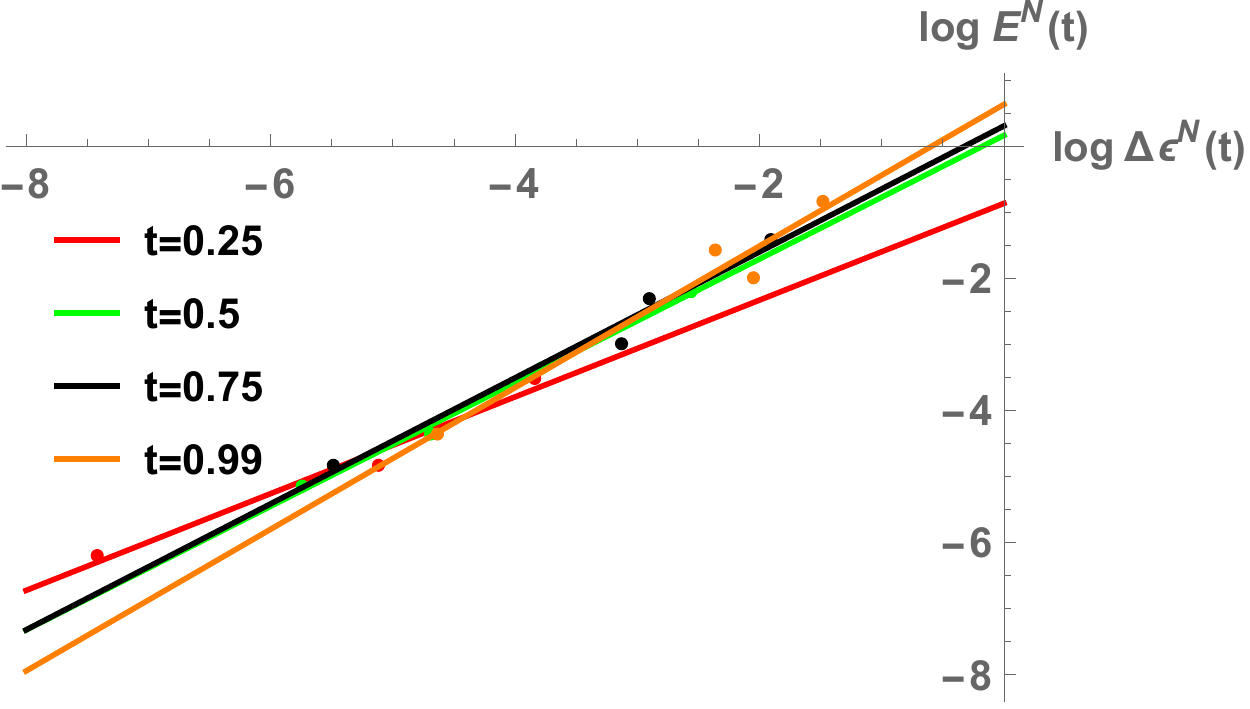}
    \caption{Left: error $E^N(t)$ in~\eqref{err_estim}, for different times as indicated. Right: relation between $\log E^N(t)$ and $\log \Delta\epsilon^N(t)$, 
    for $t=0.25$, 0.5, 0.75 and 0.99. 
    Also reported are linear regressions.
    This figure corresponds to Example~\ref{example2}.}
        \label{figure2error2}
    \end{center}
  \end{figure}
    
Figure~\ref{figmcerr2} plots the sampling error $\text{MCE}^P(t)$~\eqref{mcerr} for truncation order $N=7$, times $t=0.25$, 0.5, 0.75 and 0.99, nested samples of size $P\in\{100,200,400,800,1600,3200,6400,12800\}$, and $M=20,000$. Similar conclusions to Example~\ref{example1} are derived here. The decay pattern, although depending on the random numbers generated, is captured by the rate $1/\sqrt{P}$, as theoretically expected.

\begin{figure}[hbt!]
  \begin{center}
    \includegraphics[width=0.6\textwidth]{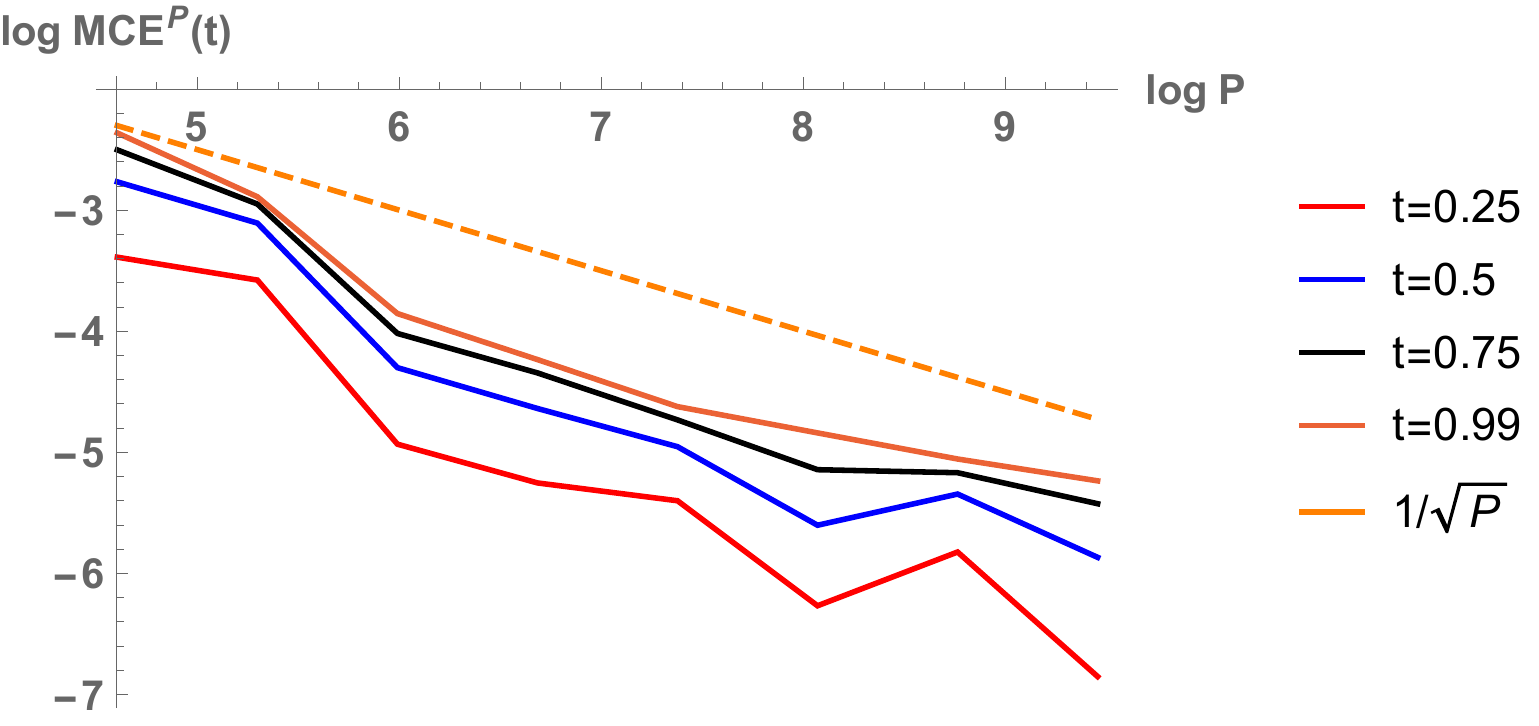}
    \caption{Sampling error~\eqref{mcerr} with the number of realizations $P$, for different times as indicated.
    This figure corresponds to Example~\ref{example2}.}
        \label{figmcerr2}
    \end{center}
  \end{figure}
    
\end{example}


\begin{example} \label{example3} \normalfont
In this example, we consider the previous degree one polynomial problem, with the following independent distributions: $A_0=4$, $A_1\sim\text{Uniform}(0,1)$, $B_0\sim\text{Gamma}(2,2)|_{[0,4]}$, $B_1\sim\text{Bernoulli}(0.35)$, $Y_0\sim\text{Poisson}(2)$ and $Y_1\sim\text{Normal}(2,1)$. 
This example coincides with Example~\ref{example1}, except that $Y_0$ and $Y_1$ have been interchanged: now $Y_0$ is discrete, while $Y_1$ is absolutely continuous. This exchange puts this example in a different theoretical case compared to Example~\ref{example1}.

By Theorem~\ref{nostre}, the unique mean square solution is expressible as a random power series $X(t)=\sum_{n=0}^\infty X_n t^n$ that is mean square convergent for all $t\in\mathbb{R}$. According to Remark~\ref{rmk_Y1}, we can approximate the probability density function of $X(t)$, $f_{X(t)}(x)$, for $t\neq0$. 

Figure~\ref{figure3} reports the approximations $\hat f_{X^N(t)}(x)$ at times $t=0.5$, $1$ and $1.5$. As $N$ grows, the graphical representations tend to overlap, denoting the convergence of the expansions.
The densities are all smooth, as expected from the smoothness of $f_{Y_1}$, except for $\hat f_{X^{N=12}(t=1.5)}(x)$ whose estimate presents noisy features. 

\begin{figure}[hbt!]
  \begin{center}
    \includegraphics[width=0.32\textwidth]{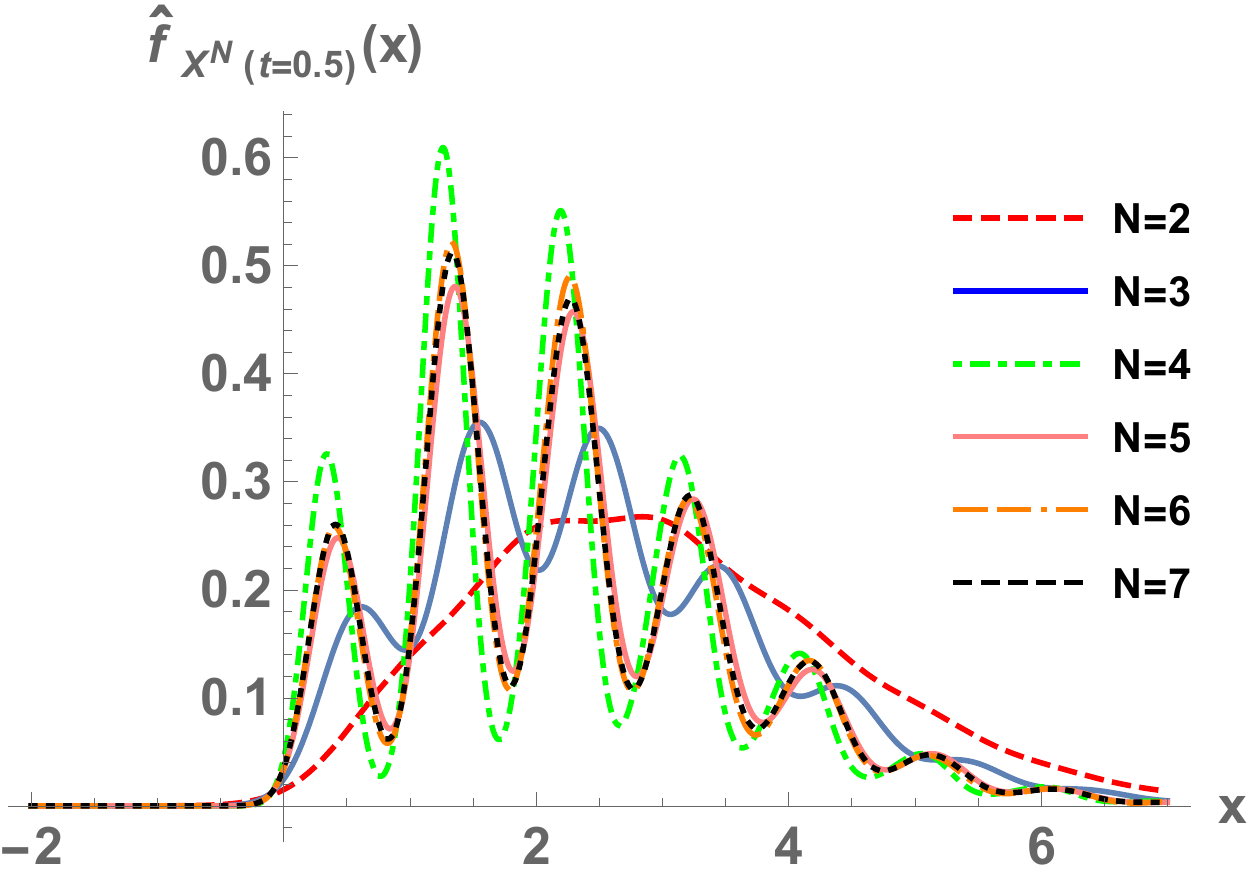}
    \includegraphics[width=0.32\textwidth]{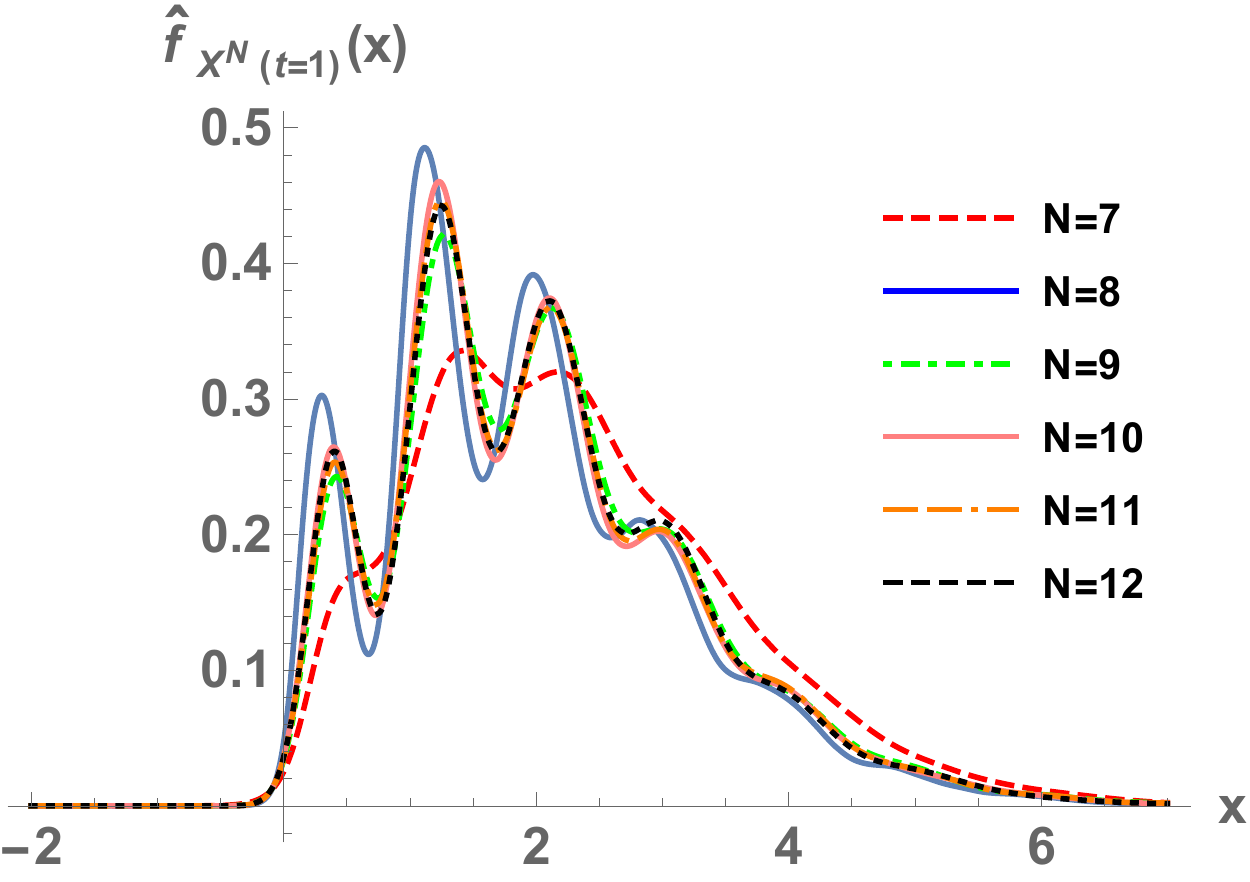}
    \includegraphics[width=0.32\textwidth]{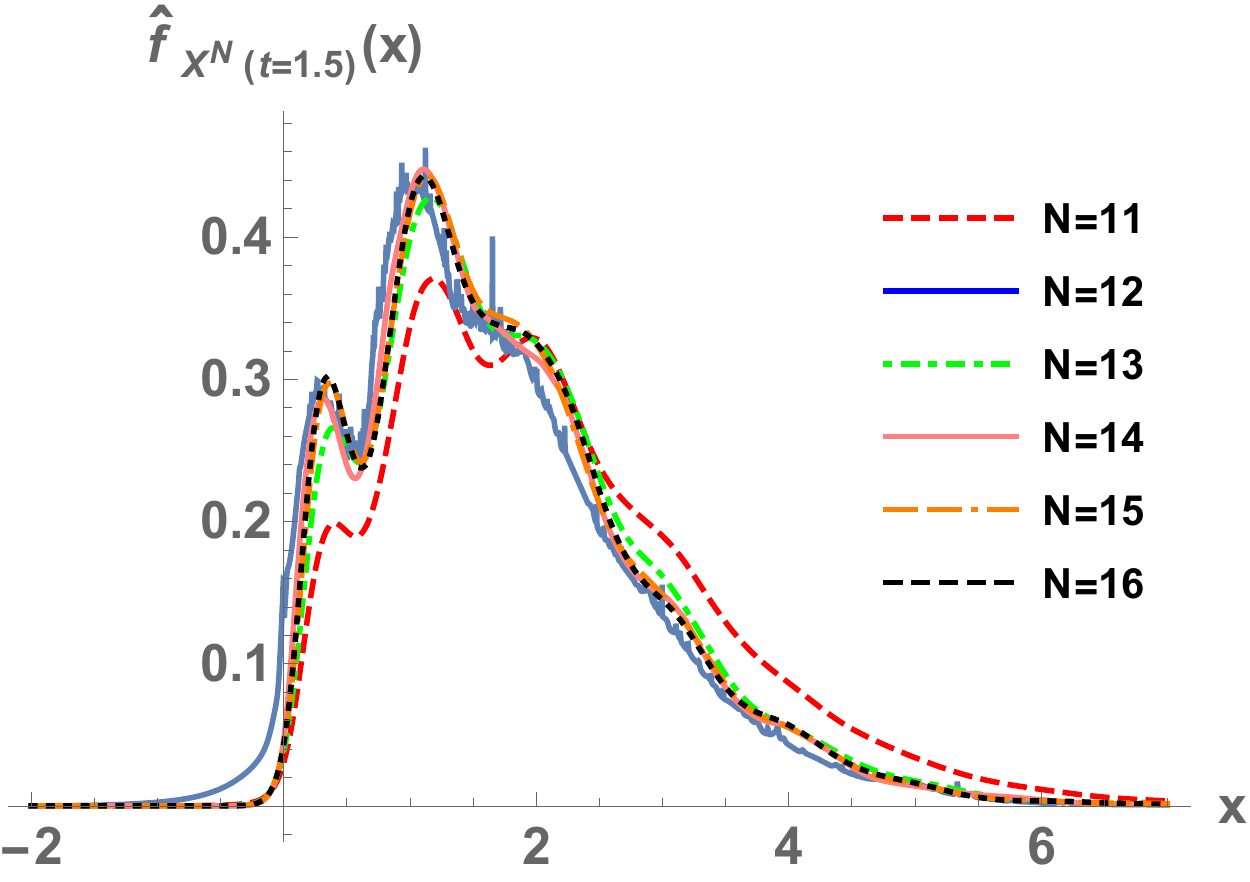}
    \caption{Graphical representations of the Monte Carlo estimates $\hat f_{X^N(t)} (x)$ at $t=0.5$ (left), $t=1$ (center) and $t=1.5$ (right), with varying orders of truncation $N$ as indicated. This figure corresponds to Example~\ref{example3}.}
        \label{figure3}
    \end{center}
  \end{figure}
    
The noisy features in $\hat f_{X^{N=12}(t=1.5)}(x)$ are due to several reasons. 
First, there is a computational issue of Mathematica\textsuperscript{\tiny\textregistered} caused by numerical overflow-underflow when too small or too large quantities are involved (for instance $\exp(z)$ for $|z|\gg1$). 
Some sample paths of $S_1^{N=12}(t)$ are vanishing near $t=1.5$, thus making the denominator $S_1^{N=12}(t)$ in the definition of $V_N(t)$ (in~\eqref{VN} by for the role of $Y_0$ and $Y_1$ exchanged) very small, with a loss of precision as a result. 
This is illustrated in Figure~\ref{sample_path_S2}, where we show some randomly generated sample paths of $S_1^{N=12}(t)$. We also report sample paths for $N=11$ and $N=13$ for comparison. 
Second, and not totally unrelated to the numerical overflow, we have $\mathbb{V}[1/|S_1^{N}(t=1.5)|]=\infty$ for $N=12$ when it remains finite for the other values of $N$ shown. As a result, the variance $\sigma^2_{N=12}$ in the Monte Carlo method (see \eqref{sigma2N}) is unbounded or very large for $N=12$, while it behaves well for other $N$, as illustrated in the last panel of Figure~\ref{sample_path_S2} (bottom right plot). 
As a result, for $N=12$, the convergence of the Monte Carlo procedure is slowed down due to the large or infinite variance, the rate $\mathcal{O}(1/\sqrt{M})$ is not obtained (see the discussion from Section~\ref{comp_aspects}), and some noisy features plague the estimator.

Luckily, the noise in $\hat f_{X^{N=12}(t=1.5)}(x)$ is not present for $N>12$. In situations where large or infinite variance occurs for some $N$, one should focus on the truncation orders $N$ for which the approximation $\hat f_{X^N(t)}(x)$ behaves nicely, without noise. In this manner, correct approximations to $f_{X(t)}(x)$ are obtained with a feasible number of samples.

\begin{figure}[hbt!]
  \begin{center}
    \includegraphics[width=0.35\textwidth]{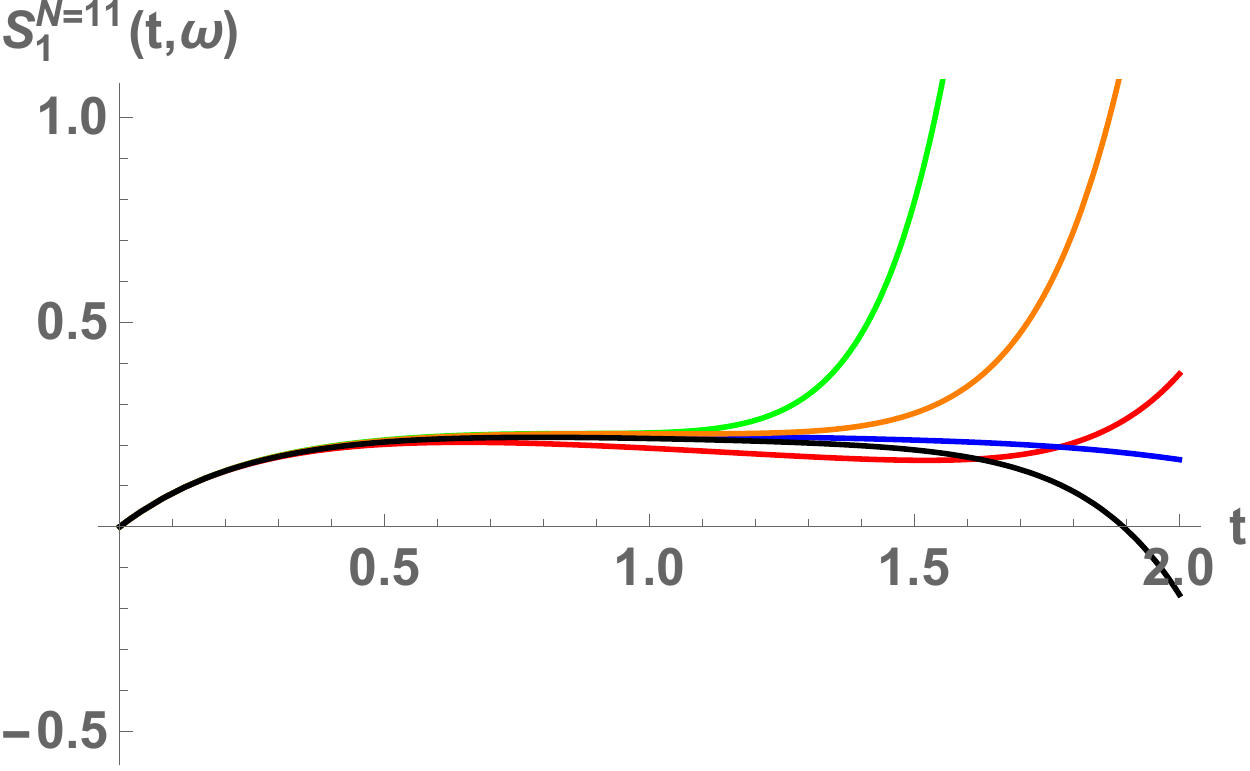}
    \includegraphics[width=0.35\textwidth]{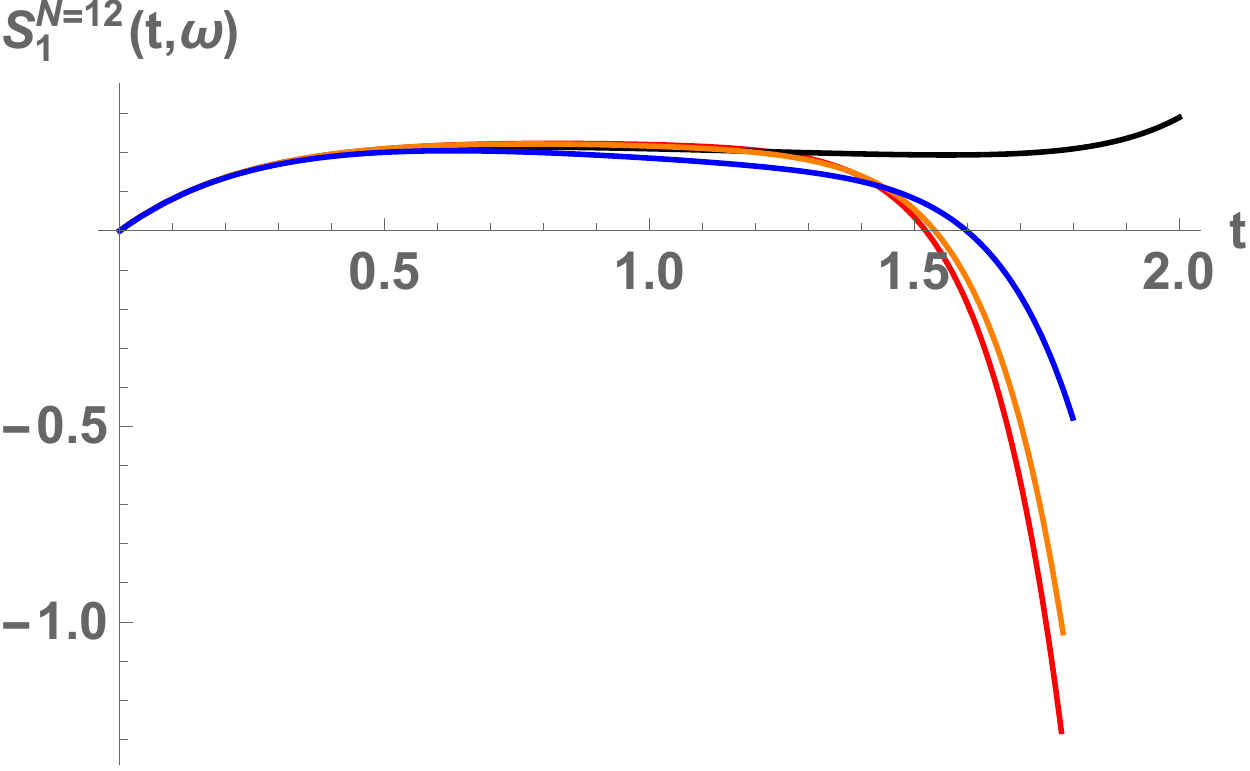}
    \includegraphics[width=0.35\textwidth]{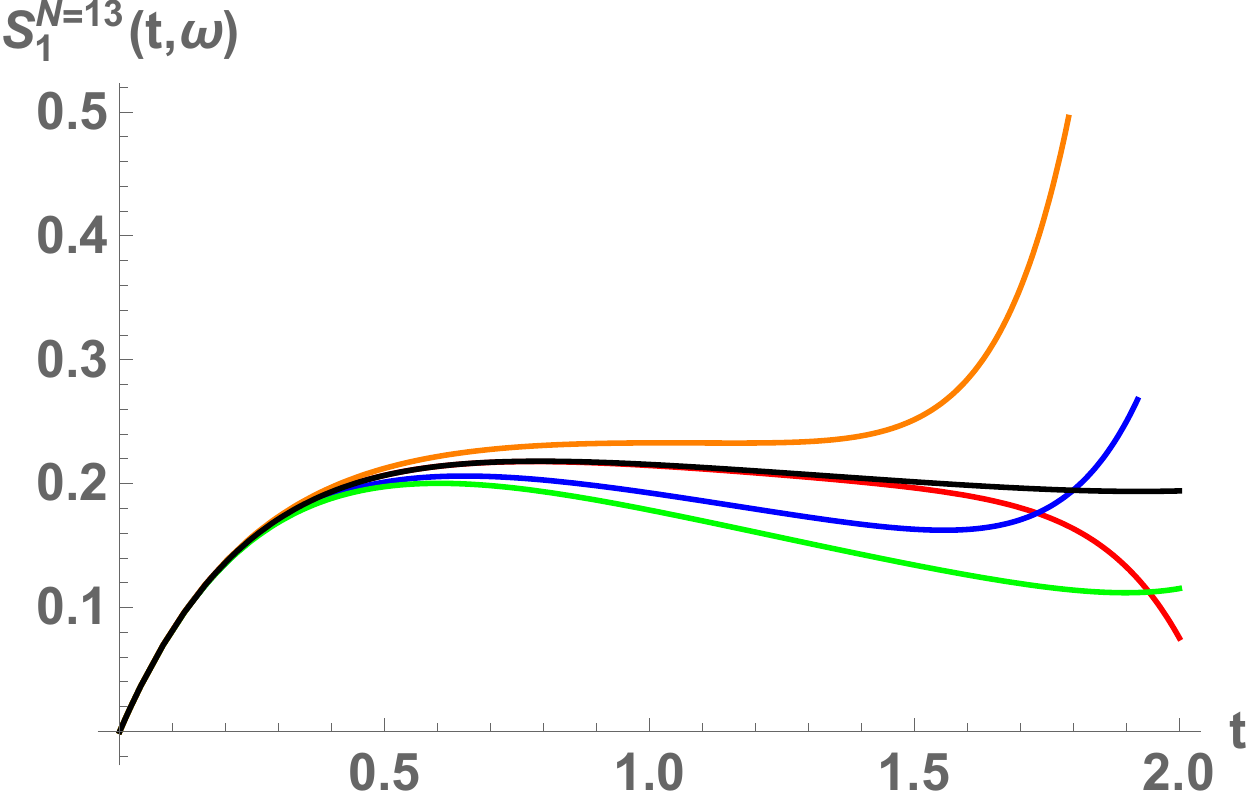}
    \includegraphics[width=0.35\textwidth]{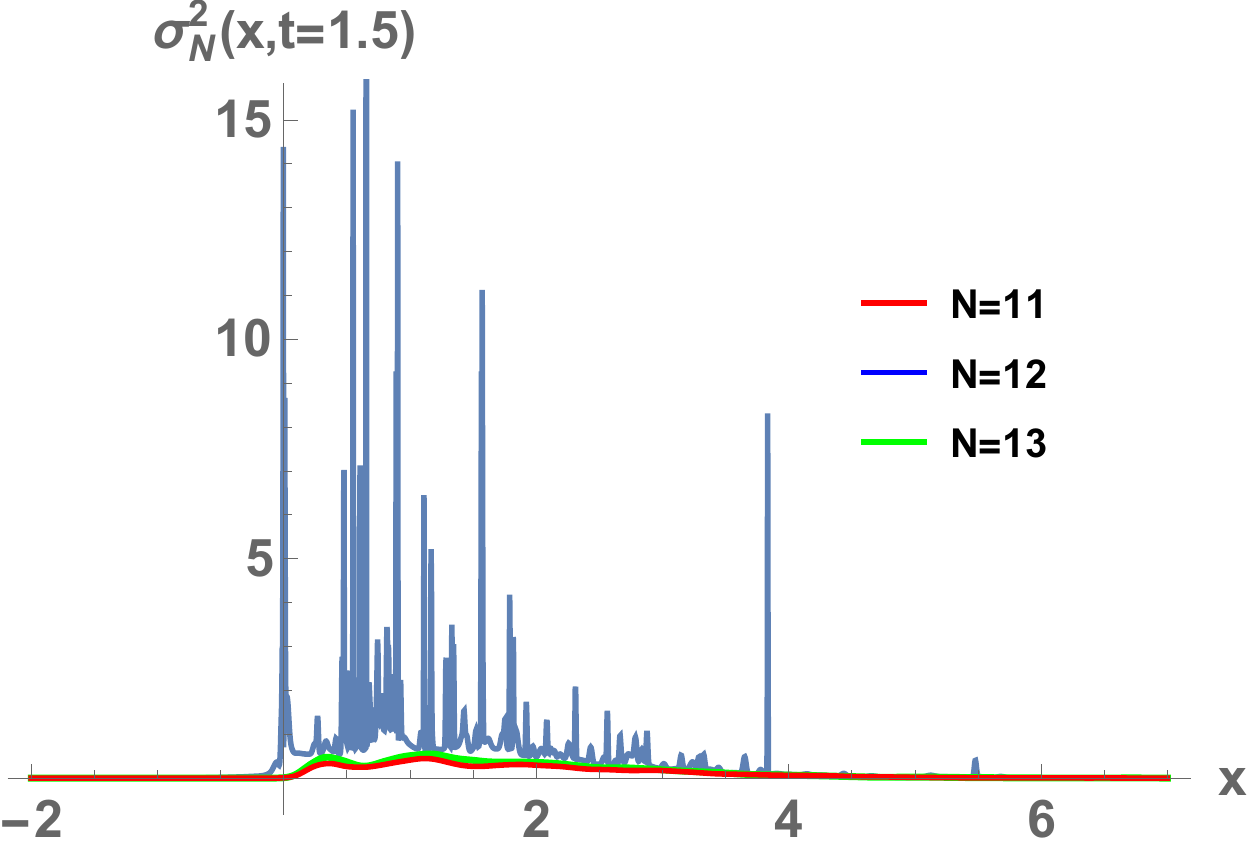}
    \caption{Random trajectories of $S_1^{N}(t)$ for $N=11$, 12 and $13$. 
    For $N=12$, observe that some trajectories vanish very close to $t=1.5$, while for $N\ne 12$ the trajectories remain away from 0. The plot in the bottom right panel shows the corresponding empirical estimates of $\sigma_N^2(x,t=1.5)$. Observe that for $N=11$ and $N=13$ the variances are small, while $\sigma^2_{N=12}$ is large (the range has been restricted to $15$).
    This figure corresponds to Example~\ref{example3}.}
        \label{sample_path_S2}
    \end{center}
  \end{figure}
    
    
Figure~\ref{figure3error} (left and center plots) presents the consecutive differences $\delta\epsilon^N(x,t)$ given by~\eqref{diff_loc}, for times $t=1$ and $1.5$. These consecutive differences are not monotonically decreasing with $N$, although a decay pattern towards $0$ is perceptible. Further, the impact of the noisy estimate $\hat f_{X^{N=12}(t=1.5)}(x)$ is clearly visible in the reported differences. The plots are entirely consistent with the theoretical results and Remark~\ref{rmk_Y1}. In Table~\ref{table3error}, we report the corresponding $\leb^1(\mathbb{R})$ norms $\Delta\epsilon^N(t)$ (see~\eqref{diff_norm}) as a summary of Figure~\ref{figure3error}. The last plot of Figure~\ref{figure3error} reports the estimate errors $\log E^N(t)$ in~\eqref{err_estim}. Again, the convergence and the sampling error are observed. This example also emphasizes that the Fr\"obenius method deteriorates for large times, as $N$ needs to increase with $t$ to maintain accurate approximations.


\begin{figure}[hbt!]
  \begin{center}
    \includegraphics[width=0.32\textwidth]{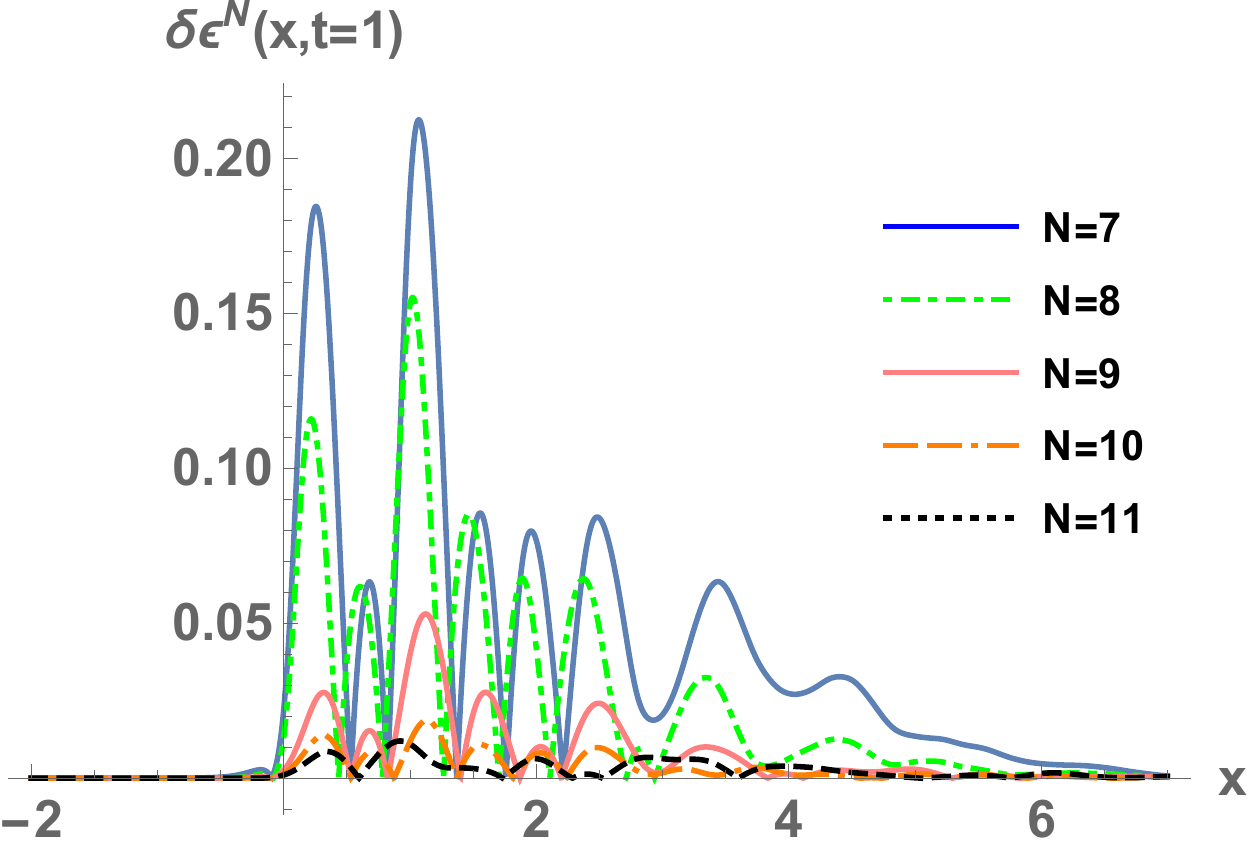}
    \includegraphics[width=0.32\textwidth]{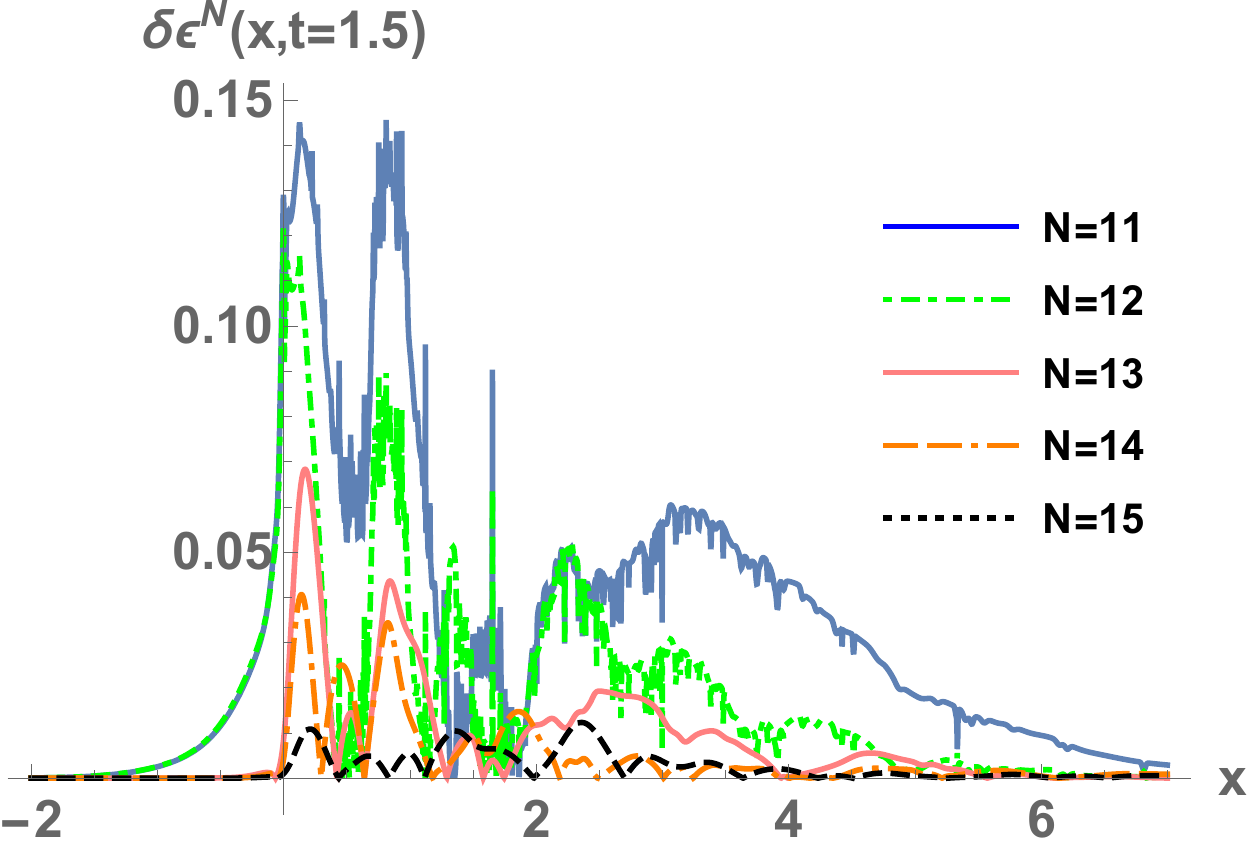}
		\includegraphics[width=0.32\textwidth]{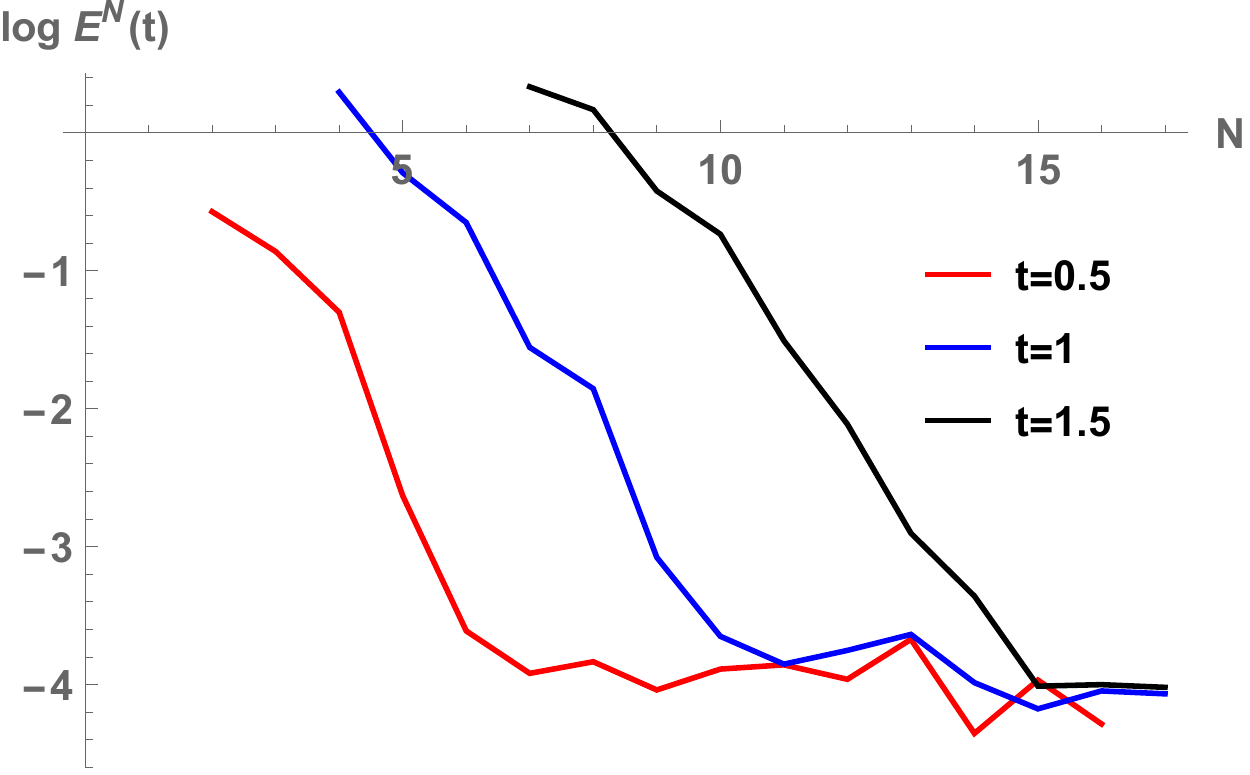}
    \caption{Differences in consecutive estimates $\delta\epsilon^N(x,t)$ at $t=1$ (left) and $t=1.5$ (center), with orders of truncation as indicated. The last plot presents the errors $E^N(t)$ (see~\eqref{err_estim}), for different times as indicated. This figure corresponds to Example~\ref{example3}.}
        \label{figure3error}
    \end{center}
  \end{figure}

\begin{table}[hbt!]
\footnotesize
\begin{center}
\begin{tabular}{|c|ccccc|} \hline
 & $N=2$ & $N=3$ & $N=4$ & $N=5$ & $N=6$  \\ 
$t=0.5$ & $0.320171$ & $0.618382$ & $0.333094$ & $0.0893759$ & $0.0291202$  \\ \hline
$$ & $N=7$ & $N=8$ & $N=9$ & $N=10$ & $N=11$  \\ 
$t=1$ & $0.308293$ & $0.185148$ & $0.0605758$ & $0.0256469$ & $0.0216692$  \\ \hline
$$ & $N=11$ & $N=12$ & $N=13$ & $N=14$ & $N=15$  \\ 
$t=1.5$ & $0.301694$ & $0.155185$ & $0.0677696$ & $0.0375570$ & $0.0202081$  \\ \hline
\end{tabular}
\caption{Norm $\Delta \epsilon^N(t)$ of differences in consecutive estimates (see~\eqref{diff_norm}) for different times $t$ and truncation orders $N$. This table corresponds to Example~\ref{example3}.}
\label{table3error}
\end{center}
\end{table}

\end{example}

\begin{example} \label{example4} \normalfont
Consider the problem with infinite expansions for $A$ and $B$, $A_n\sim\text{Beta} (11,15)$ for $n\geq0$, $B_0=0$, $B_n=1/n^2$ for $n\geq1$, $Y_0\sim \text{Uniform}(-1,1)$ and $Y_1\sim \text{Exponential}(2)$. These random inputs are again assumed to be independent. In contrast with Example~\ref{example2}, the probability density function of $Y_0$ has now two discontinuity points at $y_0 = \pm 1$, while $Y_1$ follows an absolutely continuous law. Hence, Theorem~\ref{te1} cannot be employed here.
The mean square analytic solution $X(t)=\sum_{n=0}^\infty X_n t^n$ given by Theorem~\ref{nostre} is defined on $(-1,1)$ and we must apply Theorem~\ref{te1super} to approximate $f_{X(t)}(x)$ for $t\neq0$. 
We compute the approximations at time $t=0.99$ (near the limit $1$), with orders of truncation $N=1\textendash5$. Figure~\ref{figure4} (left plot) depicts the graphs of $\hat f_{X^N(t)}(x)$. Promptly, the successive approximations of the density function tend to superimpose, thus entailing rapid convergence to the target density function $f_{X(t)}(x)$. In contrast to Example~\ref{example2}, the non-differentiability of the approximated density functions inherited from $f_{Y_0}$ is evident (here one cannot apply Theorem~\ref{te2}). Thereby, our method can capture peaks induced by non-differentiability.
This feature is a definite advantage of our method, compared to classical sample paths approximation methods where a kernel density estimation of the density would smear-out the approximation at the non-differentiability points.

A richer analysis of the convergence in this example is provided in the centered plot of Figure~\ref{figure4} and Table~\ref{table4error}, which depict consecutive differences $\delta\epsilon^N(x,t)$ (see \eqref{diff_loc}) and their norms $\Delta\epsilon^N(t)$ (see \eqref{diff_norm}), respectively. Even though the errors are not decreasing monotonically to $0$ (as in the previous Example~\ref{example3}), the convergence is evident and follows from Theorem~\ref{te1super}. Finally, in Figure~\ref{figure4} (last panel) we also plot the error estimate $\log E^N(t)$ (see \eqref{err_estim}), for distinct times. Similar to Example~\ref{example2}, the plot shows that this example needs small orders $N$ for all $t\in (-1,1)$ because of the decay of $|t|^N$. 

\begin{figure}[hbt!]
  \begin{center}
    \includegraphics[width=0.32\textwidth]{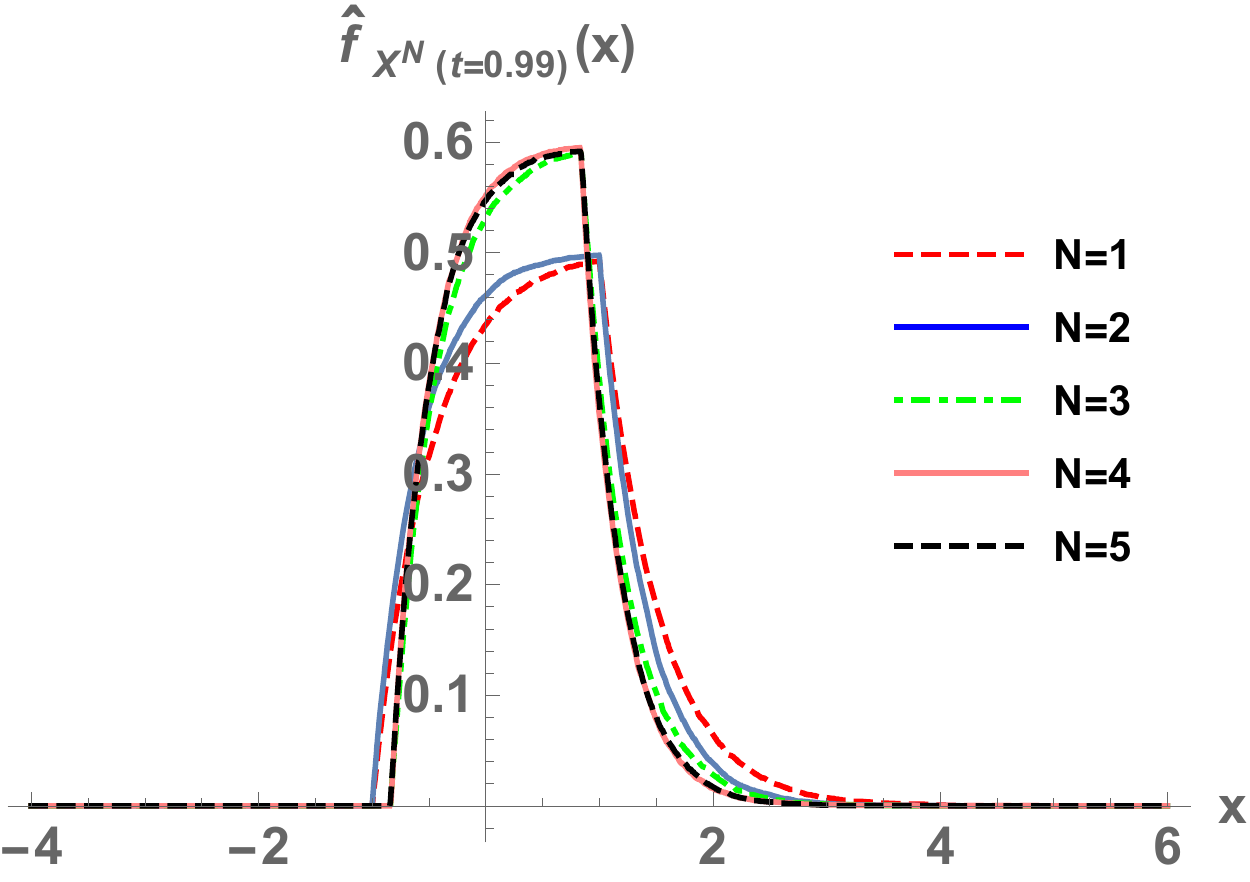}
    \includegraphics[width=0.32\textwidth]{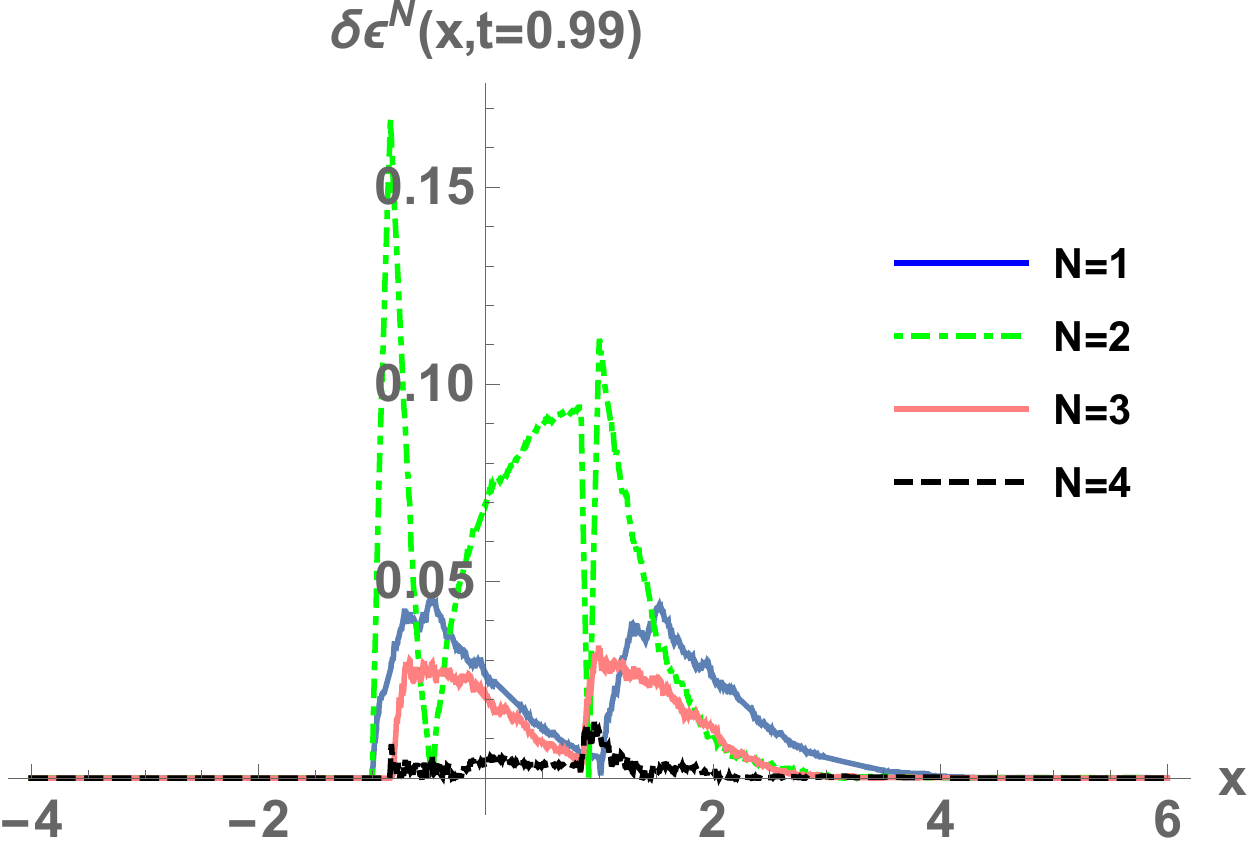}
    \includegraphics[width=0.32\textwidth]{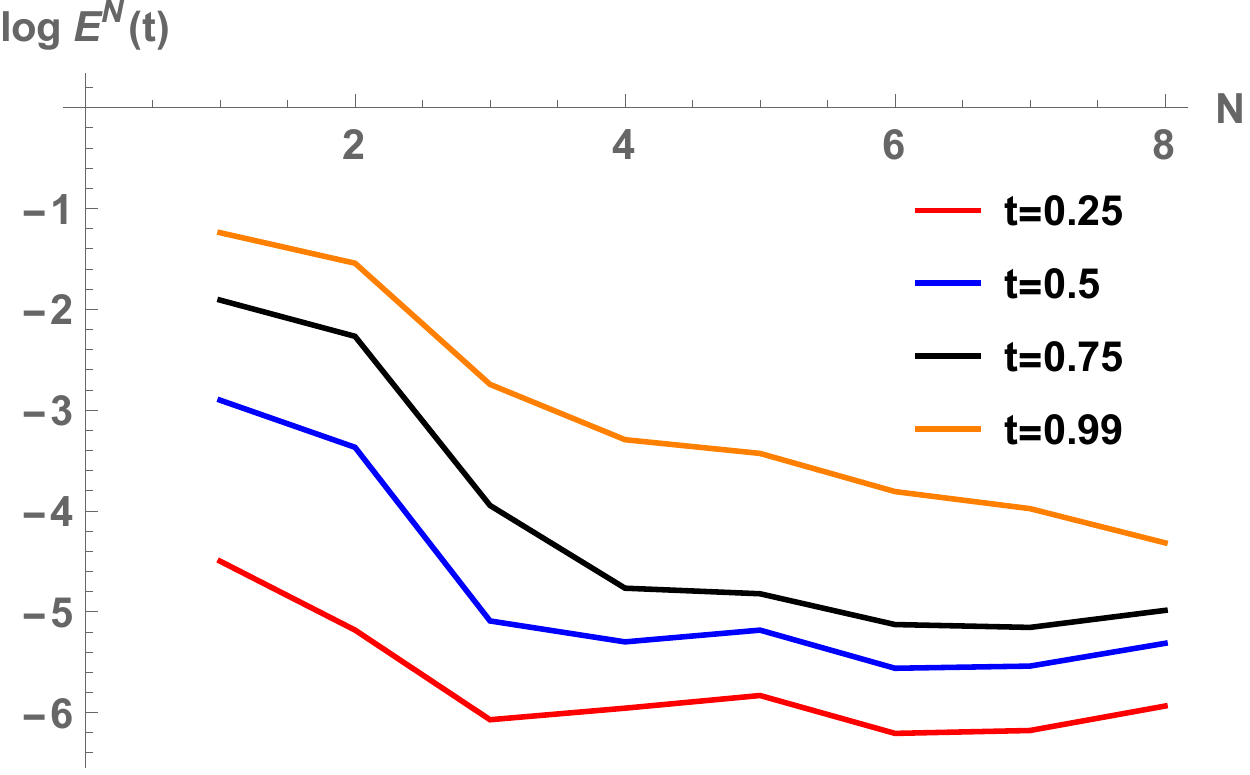}
    \caption{Graphical representation of the Monte Carlo estimates $\hat f_{X^N(t)} (x)$ at $t=0.99$ (left), with orders of truncation $N=1\textendash5$. Differences in consecutive estimates $\delta\epsilon^N(x,t)$ (see~\eqref{diff_loc}) at $t=0.99$ (center), with orders of truncation $N=1\textendash4$. Error $E^N(t)$ (see~\eqref{err_estim}), for different times as indicated (right). This figure corresponds to Example~\ref{example4}.}
        \label{figure4}
    \end{center}
  \end{figure}

    
    \begin{table}[hbt!]
\footnotesize
\begin{center}
\begin{tabular}{|c|cccc|} \hline
 & $N=1$ & $N=2$ & $N=3$ & $N=4$  \\ 
$t=0.25$ & $0.00666810$ & $0.00482414$ & $0.00169730$ & $0.00275401$  \\
$t=0.5$ & $0.0266447$ & $0.0301975$ & $0.00469989$ & $0.00524461$  \\ 
$t=0.75$ & $0.0582143$ & $0.0920184$ & $0.0176119$ & $0.00756466$  \\ 
$t=0.99$ & $0.0947395$ & $0.190933$ & $0.0597380$ & $0.0107980$  \\ \hline
\end{tabular}
\caption{Norm $\Delta \epsilon^N(t)$ of differences in consecutive estimates (see~\eqref{diff_norm}) 
at different times $t$ and for different truncation orders $N$. 
This table corresponds to Example~\ref{example4}.}
\label{table4error}
\end{center}
\end{table}
    
    
\end{example}

\begin{example} \label{example5} \normalfont
In this final example, we deal with discrete uncertainties, under the setting of Theorem~\ref{te1discrete}. Consider again the polynomial problem of Example~\ref{example1}, with $A_0=4$, $A_1=2$, $B_0=0$, $B_1=-1$, and now $Y_0\sim\text{Bernoulli}(0.4)$ and $Y_1\sim\text{Uniform}(-1,1)$, being all independent. 
By Theorem~\ref{nostre}, there is a unique mean square solution $X(t)=\sum_{n=0}^\infty X_n t^n$ on $\mathbb{R}$. 
For each $t\neq0$, the random variable $X(t)$ is absolutely continuous, due to the absolute continuity of $Y_1$. 
Theorem~\ref{te1discrete}, with $Y_1$ playing the role of $Y_0$ (see Remark~\ref{rmk_Y1}) allows for approximating $f_{X(t)}(x)$ by utilizing the convergence $\lim_{N\rightarrow\infty} f_{X^N(t)}(x)=f_{X(t)}(x)$, which holds for almost every $x\in\mathbb{R}$. 
In this particular example, Algorithm~\ref{algo} is used with $M=1,000,000$ iterations, because the deterministic values for $A(t)$ and $B(t)$ make the computational load much less demanding (see the discussion of Section~\ref{comp_aspects}). 
We will thus identify $\hat f_{X^N(t)}(x)=f_X^{N,M}(x,t)$.
For the time $t=1.5$, the numerical estimates $\hat f_{X^N(t)}(x)$ are displayed in Figure~\ref{figure5} (left plot). Observe that the computed density functions are completely different to those of the previous examples: they are discontinuous, in fact step functions, mainly due to the discontinuities in $f_{Y_1}$. This example highlights the ability of our method to capture discontinuities. The analysis of the convergence is completed with the centered plot from Figure~\ref{figure5} and Table~\ref{table5error}, where the consecutive differences $\delta\epsilon^N(x,t)$ (see \eqref{diff_loc}) and their norms $\Delta\epsilon^N(t)$ (see \eqref{diff_norm}) are reported, respectively. Table~\ref{table5error} considers times $t=0.5$, $1$ and $1.5$. Finally, the last panel from Figure~\ref{figure5} plots $\log E^N(t)$ (see~\eqref{err_estim}) as a function of $N$. The lower bound for the global error is the sampling error, which is smaller than in the previous four examples, owing to the larger number of samples considered. Comparing the last plot from Figure~\ref{figure5} with the corresponding figures from the previous four examples, the non-monotonic decay of the error is also more pronounced for the three times. 
The discontinuity in the graph of the target distribution $f_{X(t)}(x)$ can explain the highly non-monotonic decay with the truncation order. 
Moreover, $f_{Y_1}$ is not Lipschitz continuous, so the exponential convergence rate discussed in Remark~\ref{rmk_rate} is not applicable in the present example. 
Finally, as for the other examples, the truncation order needed to reduce the error to the sampling contribution increases as we move away from the origin $t_0=0$. This behavior may pose severe challenges for large times $t$.



\begin{figure}[hbt!]
  \begin{center}
    \includegraphics[width=0.32\textwidth]{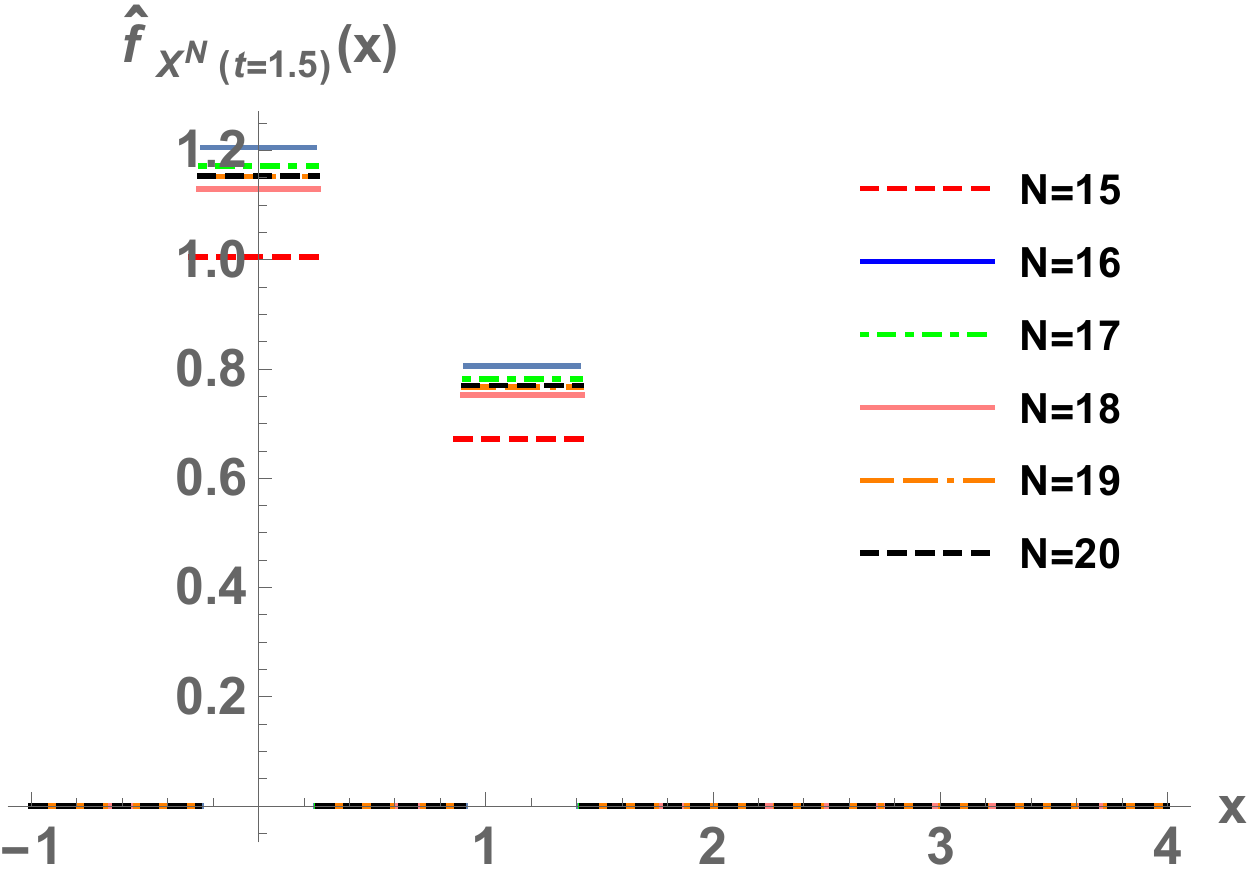}
		\includegraphics[width=0.32\textwidth]{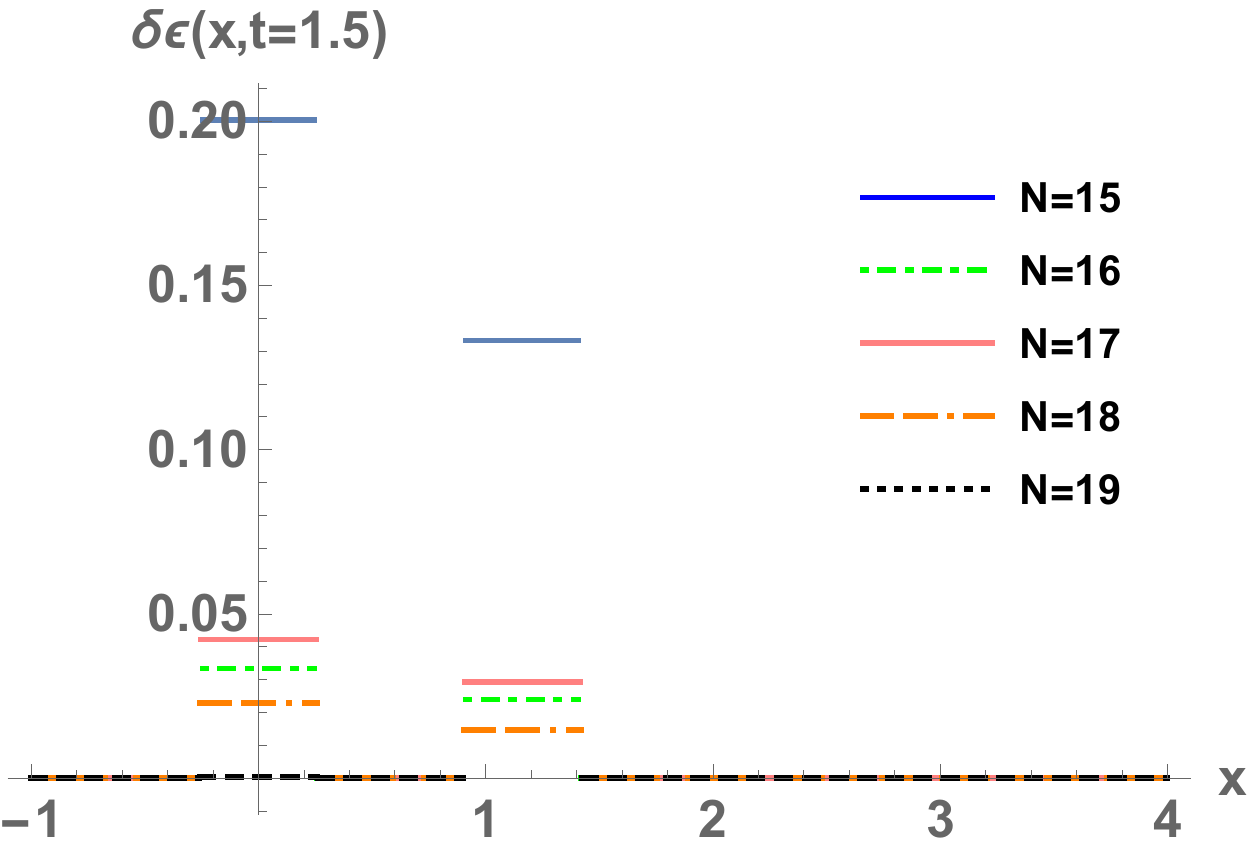}
		\includegraphics[width=0.32\textwidth]{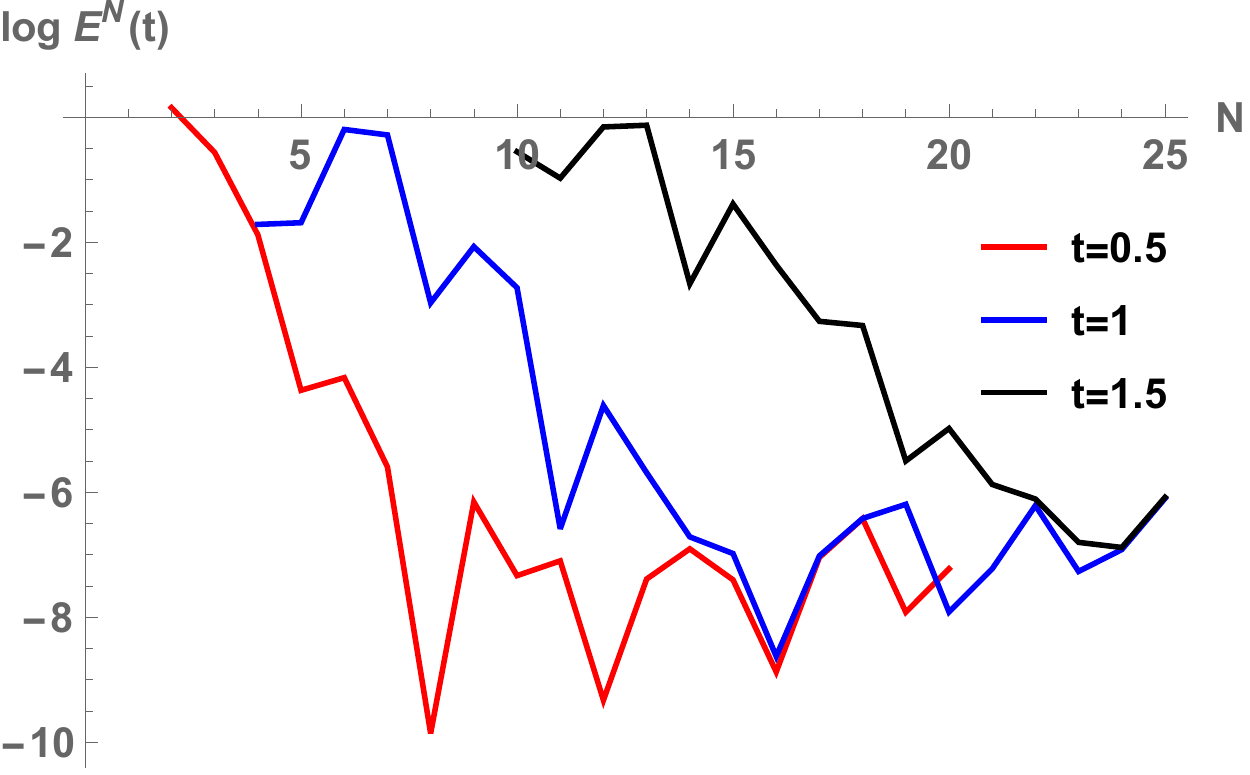}
    \caption{In the left plot, graphical representations of the Monte Carlo estimates $\hat f_{X^N(t)} (x)$ at $t=1.5$, with orders of truncation $N$ as indicated. In the center plot, differences in consecutive estimates $\delta\epsilon^N(x,t)$ (see~\eqref{diff_loc}) at $t=1.5$. In the last plot, error $E^N(t)$ (see~\eqref{err_estim}), for different times as indicated. This figure corresponds to Example~\ref{example5}.}
        \label{figure5}
    \end{center}
  \end{figure}
	
	\begin{table}[hbt!]
\footnotesize
\begin{center}
\begin{tabular}{|c|ccccc|} \hline
 & $N=2$ & $N=3$ & $N=4$ & $N=5$ & $N=6$  \\ 
$t=0.5$ & $0.833333$ & $0.678571$ & $0.140704$ & $0.0280760$ & $0.0190810$  \\ \hline
$$ & $N=7$ & $N=8$ & $N=9$ & $N=10$ & $N=11$  \\ 
$t=1$ & $0.785425$ & $0.0767401$ & $0.187360$ & $0.0665995$ & $0.00866170$  \\ \hline
$$ & $N=15$ & $N=16$ & $N=17$ & $N=18$ & $N=19$  \\ 
$t=1.5$ & $0.331941$ & $0.0569325$ & $0.0732120$ & $0.0390644$ & $0.00346696$  \\ \hline
\end{tabular}
\caption{Norm $\Delta \epsilon^N(t)$ of differences in consecutive estimates (see~\eqref{diff_norm}) at different times $t$ and truncation orders $N$. 
This table corresponds to Example~\ref{example5}.}
\label{table5error}
\end{center}
\end{table}

\end{example}

\section{Control variates method: Computational aspects and numerical analysis} \label{sec_control_var}

In Section~\ref{comp_aspects} we saw that the estimation error is split into two contributions: bias error caused by the Fr\"obenius method and statistical error due to using a finite number $M$ of samples. The statistical error, denoted as $\mathcal{E}_{N,M}(x,t) = f_{X^N(t)}(x)-f_X^{N,M}(x,t)$, has mean zero and variance $\sigma_N^2(x,t)/M$ asymptotically with $M$, where $\sigma_N^2(x,t)$ is defined by~\eqref{sigma2N} (assuming that $\sigma_N^2(x,t)<\infty$). Thus, the Monte Carlo error depends on the sample length $M$ and on the variance of the estimator.

Variance reduction techniques aim at improving the Monte Carlo estimates for a given computational effort~\cite{botev}. In what follows, we describe the control variates method applied for this setting. Let
\[
 Z^N(x,t)=f_{Y_0}\left(\frac{x-Y_1S_1^N(t)}{S_0^N(t)}\right)\frac{1}{|S_0^N(t)|},
 \]
such that $f_{X^N(t)}(x)=\mathbb{E}[Z^N(x,t)]$ and $\sigma_N^2(x,t)=\mathbb{V}[Z^N(x,t)]$. Let $S^{N_0}(t)$ be the control variate constructed from the power series $S_0^{N_0}(t)$ and/or $S_1^{N_0}(t)$ truncated at level $N_0<N$. The statistics of $S^{N_0}(t)$ can be exactly computed with no difficulties. Indeed, $S^{N_0}(t)$ is just a polynomial in $t$, so its statistical moments are calculated using the linearity of the expectation and the precomputed moments of the random inputs $A_0,\ldots,A_{N_0}$, $B_0,\ldots,B_{N_0}$. Let $\tau^{N_0}(t)=\mathbb{E}[S^{N_0}(t)]$. Let
\begin{equation}
 Z^{N,*}(x,t)=Z^N(x,t)+c^*(x,t)\left(S^{N_0}(t)-\tau^{N_0}(t)\right) 
 \label{ZNstar}
\end{equation}
be an unbiased estimator of $f_{X^N(t)}(x)$, where $c^*(x,t)$ is a function that minimizes the variance 
\[ \mathbb{V}[Z^{N,*}(x,t)]= \mathbb{V}[Z^N(x,t)]+\left(c^*(x,t)\right)^2\mathbb{V}[S^{N_0}(t)]+2c^*(x,t)\Cov[Z^N(x,t),S^{N_0}(t)]. \]
The function $c^*(x,t)$ is usually referred to as the control coefficient. It is easy to see that 
\begin{equation}
 c^*(x,t)=-\frac{\Cov[Z^N(x,t),S^{N_0}(t)]}{\mathbb{V}[S^{N_0}(t)]}. 
 \label{cStar}
\end{equation}
The variance of the new estimator results
\[ \mathbb{V}[Z^{N,*}(x,t)]=\left(1-\rho_{Z^N(x,t),S^{N_0}(t)}^2\right)\mathbb{V}[Z^N(x,t)]<\mathbb{V}[Z^N(x,t)], \]
where $\rho_{Z^N(x,t),S^{N_0}(t)}\in(-1,1)$ is the correlation coefficient.

From the computational viewpoint, \eqref{cStar} is calculated as follows: the variance of $S^{N_0}(t)$ from the denominator is determined exactly, as previously explained; whereas the covariance from the numerator is not available in general and is estimated by executing crude Monte Carlo simulation, using a moderate number of realizations. Once~\eqref{cStar} is estimated, a crude Monte Carlo procedure, similar to that from Algorithm~\ref{algo}, is conducted for $Z^{N,*}(x,t)$~\eqref{ZNstar}. The steps are briefly described in Algorithm~\ref{algooo}. We do not explicitly write the parts corresponding to crude Monte Carlo procedures, as they were already detailed in Algorithm~\ref{algo}. Here we recommend to set numeric values for $t$ and $x$, otherwise the computational burden is prohibitive.

\begin{algorithm}[hbtp!]
    \begin{algorithmic}[1]
 
   \Statex \textbf{Inputs:} $t_0$; $N_0<N$; $f_{Y_0}$; probability distribution of $A_0,\ldots,A_N$, $B_0,\ldots,B_N$, $Y_1$; control variate $S^{N_0}(t)$; and number $M$ of realizations.
	 \State Calculate $\tau^{N_0}(t)\gets\mathbb{E}[S^{N_0}(t)]$, and $\mathbb{V}[S^{N_0}(t)]$ \Comment{Exact}
	 \State Crude Monte Carlo simulation for $\Cov[Z^N(x,t),S^{N_0}(t)]$ \Comment{\parbox[t]{.36\linewidth}{Low number of samples $\ll M$; Analogous to Algorithm~\ref{algo}}}
	\State $c^*(x,t)\gets -\frac{\widehat{\Cov}[Z^N(x,t),S^{N_0}(t)]}{\mathbb{V}[S^{N_0}(t)]}$ \Comment{See~\eqref{cStar}}
	\State Crude Monte Carlo simulation for $\mathbb{E}[Z^{N,*}(x,t)]$ \Comment{See~\eqref{ZNstar}; Analogous to Algorithm~\ref{algo}}
	\State \textbf{Return} $f_X^{N,M}(x,t)\gets \hat{\mathbb{E}}[Z^{N,*}(x,t)]$\Comment{Approximation of $f_{X^N(t)}(x)$}
\end{algorithmic}
    \caption{Estimation of $f_{X^N(t)}(x)$ \textit{via} a control variates method.}
        \label{algooo}
\end{algorithm}

Algorithm~\ref{algooo} is more efficient than Algorithm~\ref{algo}, as the variance of the Monte Carlo estimator has been lowered. In numerical computations, this entails several consequences. Suppose that the number $M$ of realizations is fixed, and we vary the truncation order $N$. The error decreases exponentially with $N$ until the sampling error is reached. The variance reduction technique allows for a lower sampling error for the same sample length $M$, so the exponential decay of the error with $N$ is prolonged. On the other hand, suppose that $N$ is fixed and $M$ varies. In log-log scale, the Monte Carlo error is approximately a line with slope $-1/2$; the control variates method translates the error line lower according to the difference of the logarithms of the standard deviations of the estimators $Z^N(x,t)$ and $Z^{N,*}(x,t)$. Obviously, the performances depend on the particular random numbers generated.

\begin{example} \label{exCVM} \normalfont
The setting is the same as Example~\ref{example1}: $A(t) = A_0+A_1t$ and $B(t) = B_0+B_1t$, where $A_0=4$, $A_1\sim\text{Uniform}(0,1)$, $B_0\sim\text{Gamma}(2,2)|_{[0,4]}$, $B_1\sim\text{Bernoulli}(0.35)$, $Y_0\sim\text{Normal}(2,1)$ and $Y_1\sim\text{Poisson}(2)$, all being independent random variables. The goal of this example is to show that the control variates method entails a lower statistical error. Due to the previous theoretical exposition, this fact is generalizable to any other situation.

We choose as control variate $S^{N_0}(t)=S_0^{N_0}(t)$, $N_0=10$. The covariance $\Cov[Z^N(x,t),S^{N_0}(t)]$ is estimated by using crude Monte Carlo simulation, with 2500 realizations. Given $M=20,000$, the output function $f_X^{N,M}(x,t)$ is denoted as $\hat f_{X^N(t)}(x)$. In Table~\ref{tabCVM}, we tabulate the differences in consecutive (in $N$) estimates, $\Delta\epsilon^N(t)$, defined in~\eqref{diff_norm}. For $t=1.5$, the results for both Algorithms~\ref{algo} and~\ref{algooo} are reported, for comparison. Notice that, while the consecutive differences start to stabilize when $N\geq14$ for the crude Monte Carlo algorithm, they keep decreasing for the control variates method. This can be visually observed in Figure~\ref{figCVM}. In the first panel, we report the errors $E^N(t)$ defined by~\eqref{err_estim} (we set $L=30$). The semi-log plot shows exponential decay with $N$ until the statistical error is attained. For the simple Monte Carlo approach, the decay stops and stabilizes earlier than for the control variates method. The second panel of the figure presents, for $N=20$ fixed, how the statistical error $\text{MCE}^P(t)$ decreases with the sample length $P$, see~\eqref{mcerr}. The log-log plot reflects the decrease rate $\mathcal{O}(1/\sqrt{P})$; the control variates method has lower constant corresponding to $\mathcal{O}$ and becomes more efficient as $P$ grows.

\begin{table}[hbt!]
\footnotesize
\begin{center}
\begin{tabular}{|c|ccccc|}\hline
$ $ & $N=11$ & $N=12$ & $N=13$ & $N=14$ & $N=15$  \\
 Crude Monte Carlo  & $0.348643$ & $0.180075$ & $0.0721679$ & $0.0320314$ & $0.0198364$  \\
Control variates & $0.32902$ & $0.171683$ & $0.0769344$ & $0.0296144$ & $0.00371231$  \\ \hline
\end{tabular}
\caption{Norm $\Delta \epsilon^N(t=1.5)$ of differences in consecutive estimates (see~\eqref{diff_norm})
for different truncation orders $N$, for the crude Monte Carlo and the control variates methods.
This table corresponds to Example~\ref{exCVM}.}
\label{tabCVM}
\end{center}
\end{table}

\begin{figure}[hbt!]
  \begin{center}
	\includegraphics[width=0.4\textwidth]{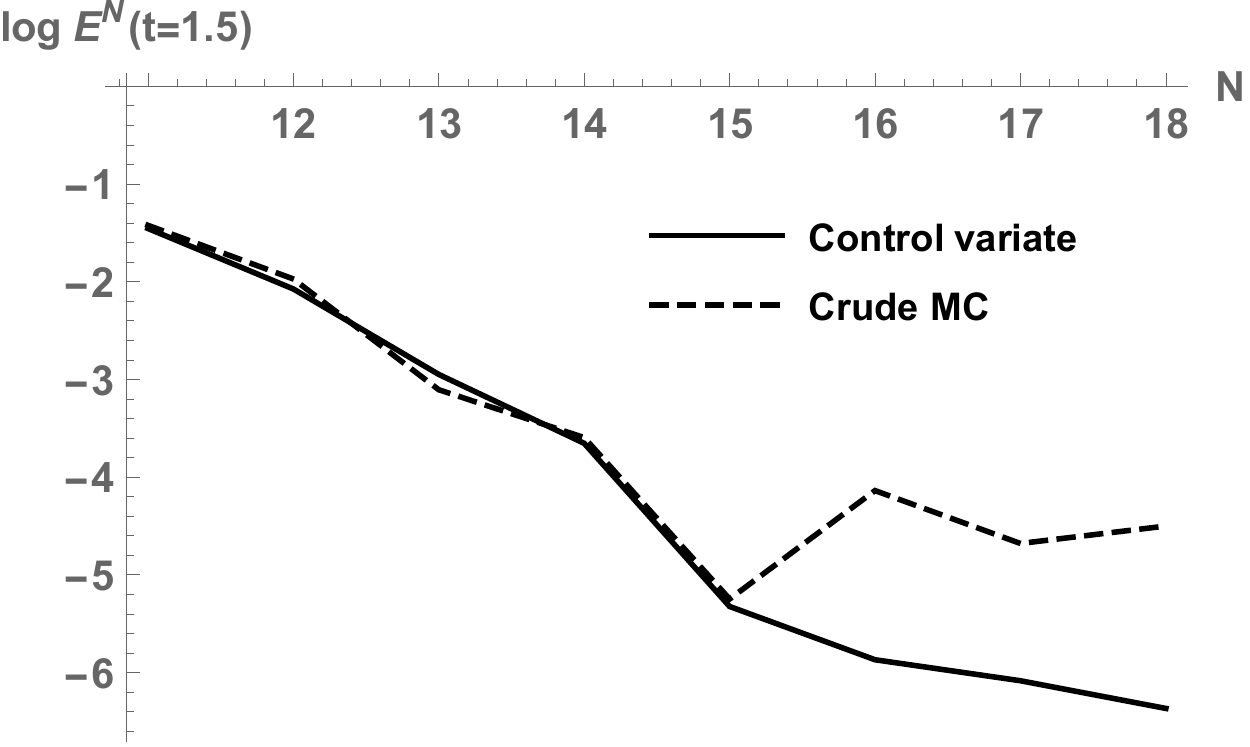}
    \includegraphics[width=0.4\textwidth]{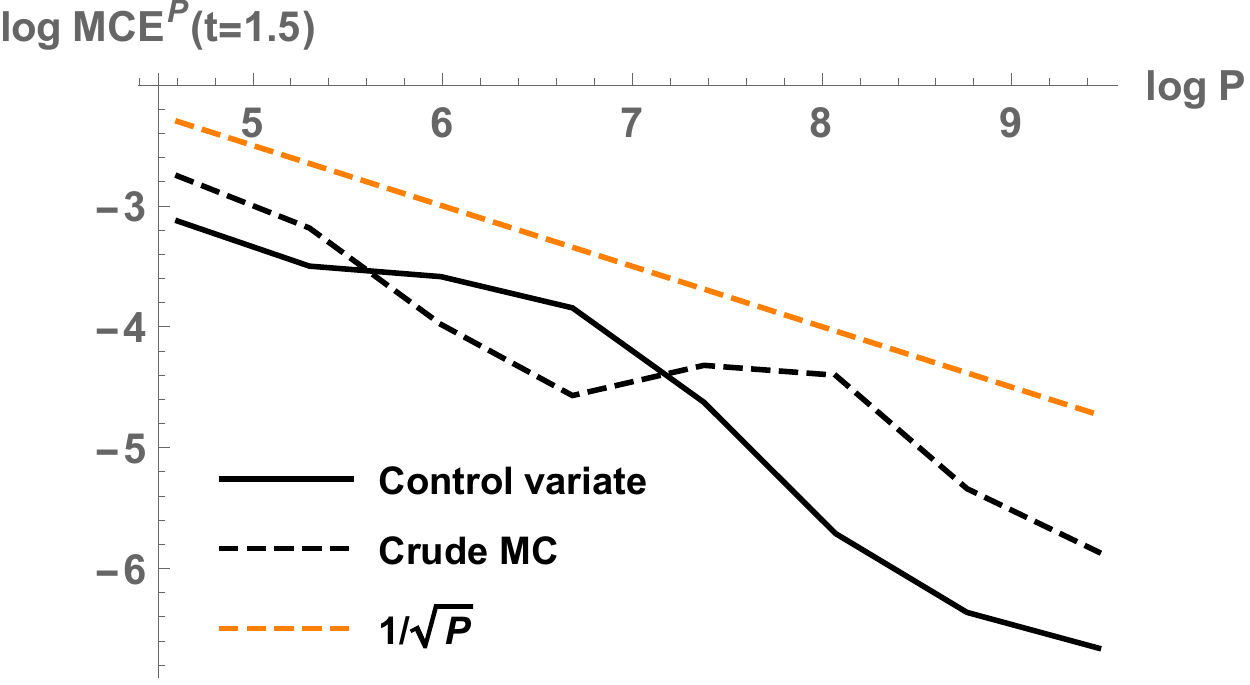}
    \caption{Left: error $E^N(t=1.5)$ (see~\eqref{err_estim}). Right: sampling error $\text{MCE}^P(t=1.5)$ (see~\eqref{mcerr}) with the number of realizations $P$.
    This figure corresponds to Example~\ref{exCVM}.}
        \label{figCVM}
    \end{center}
  \end{figure}
	
We finish this example by showing in Figure~\ref{figrho} the correlation $\rho(x,t)=\rho_{Z^N(x,t),S^{N_0}(t)}$ and the control coefficient $c^*(x,t)$, for $N=20$, $N_0=10$ and $t=1.5$. Observe that the correlation coefficient lies within $(-1,1)$ for all $x$. The variance reduction in the pointwise estimation of the density at $x$ depends on the magnitude of the correlation coefficient. On the other hand, the control coefficient is non-negligible on the interval $(-1,5)$, approximately, which corresponds to the domain where the density is also not negligible, see the third panel from Figure~\ref{figure1}.

\begin{figure}[hbt!]
  \begin{center}
	\includegraphics[width=0.4\textwidth]{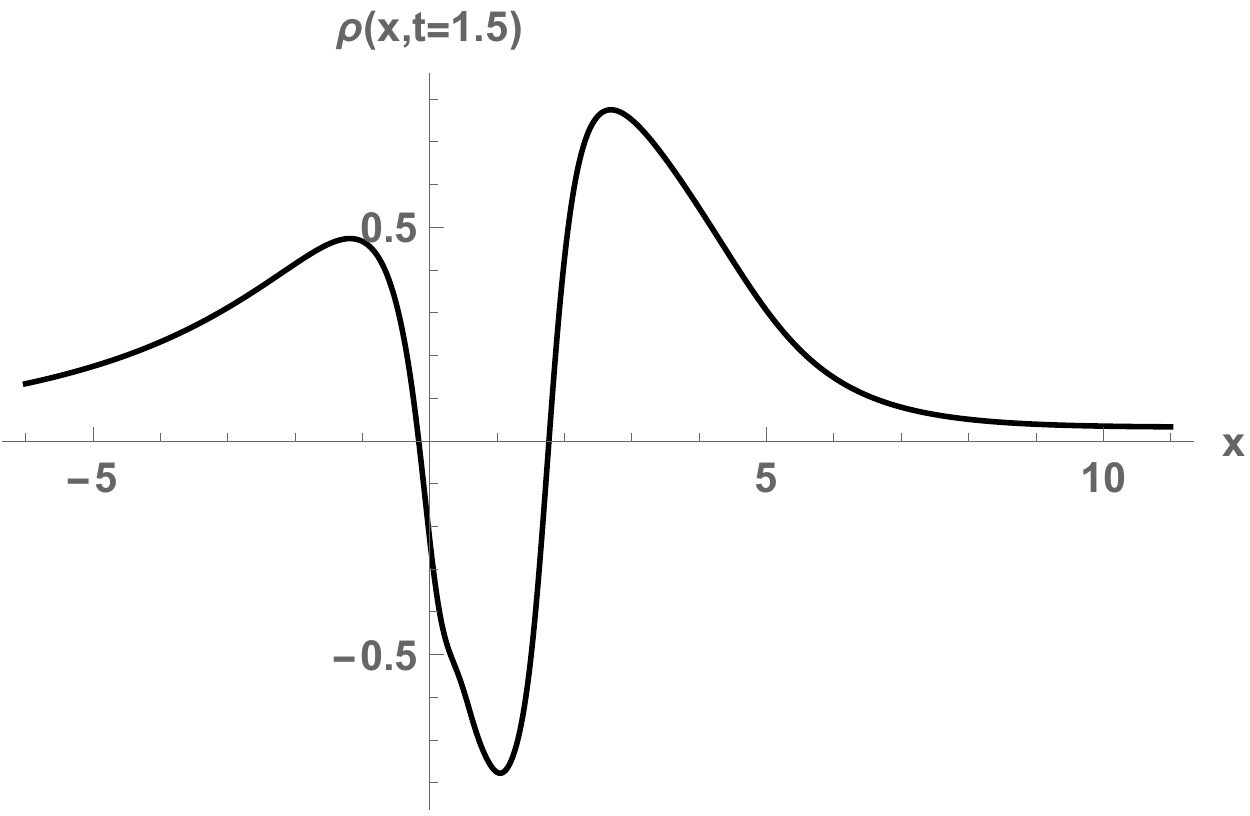}
    \includegraphics[width=0.4\textwidth]{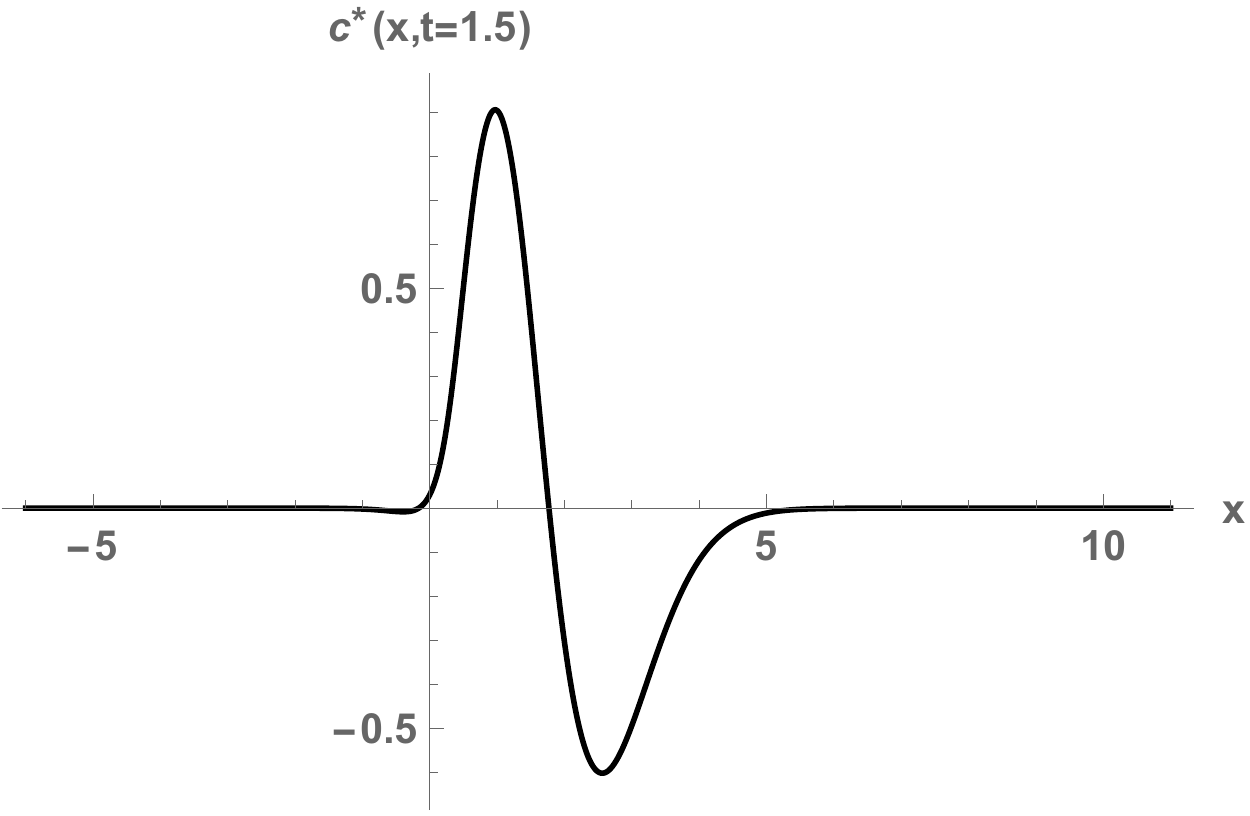}
    \caption{Left: correlation coefficient. Right: control coefficient.
    This figure corresponds to Example~\ref{exCVM}.}
        \label{figrho}
    \end{center}
  \end{figure}

\end{example}

\section{Conclusions and perspectives} \label{concl}

In this paper, we address the analysis of the random non-autonomous second order linear differential equation. When the data $A(t)$ and $B(t)$ are given by random power series on $(t_0-r,t_0+r)$ in $\leb^\infty(\Omega)$, and the initial conditions $Y_0$ and $Y_1$ belong to $\leb^2(\Omega)$, it is possible to construct a random power series solution $X(t)$ on $(t_0-r,t_0+r)$ in the mean square sense, whose coefficients satisfy a random difference equation. This approach is the generalization of the Fr\"obenius method to the random setting. The convergence rate of the power series of $X(t)$ is exponential for each time $t$, but not uniformly on the whole time domain $(t_0-r,t_0+r)$. For a fixed tolerance on the mean square error of $X(t)$, the order of truncation of the power series needs to be increased, in general, when $|t-t_0|$ grows.

We proved the existence of solution $X(t)$ when $A(t)$ and $B(t)$ are bounded. In the unbounded case, we present in this paper two counterexamples of existence. In practice, the hypotheses can be fulfilled by truncating unbounded random coefficients.

The probability density function of $X(t)$ can be expressed as the expectation of a random process $Z(x,t)$,  $f_{X(t)}(x)=\mathbb{E}[Z(x,t)]$, using the law of total probability. 
A closed-form expression for $Z$ is derived in terms of the fundamental set by exploiting the linearity of the system. 
However, to compute this expectation, one needs to perform a dimension reduction of the problem, by truncating the series of the fundamental set used to express the solution $X(t)$. Denoting $X^N(t)$, $N\geq0$, the truncation of $X(t)$, we show that $f_{X^N(t)}(x)$ converges to $f_{X(t)}(x)$ pointwise as $N\rightarrow\infty$ under certain conditions (regarding Nemytskii operators); in some cases, an exponential convergence may be achieved for each $t$ and $x$. The pointwise convergence also implies convergence in $\leb^p(\mathbb{R})$, $1\leq p<\infty$. In particular, the convergence in $\leb^1(\mathbb{R})$ is equivalent to the convergence in the total variation and the Hellinger distances, which are instances of $f$-divergences. 

From a numerical standpoint, the expectation defining $f_{X^N(t)}(x)=\mathbb{E}[Z^N(x,t)]$ is computable \textit{via} a Monte Carlo sampling strategy. We propose a crude Monte Carlo algorithm for that purpose, which estimates $f_{X^N(t)}(x)$. This algorithm is implemented in the software Mathematica\textsuperscript{\tiny\textregistered}, and it can be used to compute pointwise approximations of the density function $f_{X(t)}(x)$. One key feature of the algorithm is that it handles discontinuity and non-differentiability points of $f_{X(t)}(x)$ appropriately, without smoothing them out. To improve the efficiency of the algorithm, several variance reduction methods can be conducted. We employ the control variates method to lower the sampling error for the same computational cost.

To the best of our knowledge, this paper is the first one to provide such analysis of random second order linear differential equations. However, we point out certain limitations of our methodology, which constitute potential avenues for future developments. 

To start with, despite the exponential convergence rate, the approximations substantiated on the Fr\"obenius method may deteriorate for large $|t-t_0|$. This fact is inherent to Taylor series-based methods and also plagues other types of stochastic computations, such as PC expansions. Following~\cite{olm}, using random time-transformations may help to improve the convergence of the Fr\"obenius method and mitigate this issue.

Another point requiring a more in-depth analysis is the ignorance of the specific values of $\delta$ and $\mu$. In particular, we showed that if the truncated processes from the fundamental set vanish for some trajectories near the time $t$ of interest, the numerical estimate of the density becomes very noisy (see Example~\ref{example3}). This effect is due to the variance of $Z^N(x,t)$ that may be very large or infinite, with a severely deteriorated Monte Carlo convergence in these situations. We are currently exploring different strategies to sort out this issue, such as the path-wise selection of the variable ($Y_0$ or $Y_1$) used in the expression of $Z^N(x,t)$, in order to control its variance.

Efforts to weaken or modify the theoretical hypotheses and enlarge the applicability of our method shall also be carried out. As an example, the extension of the method to the case of $Y_0$ and $Y_1$ not absolutely continuous would also present a valuable achievement. 
Similarly, an extension of the present methodology to linear systems of second order random differential equations may be of great interest,
while the application to other stochastic models of our expertise on random expansions and density approximations could be interesting.

At the computational level, the Monte Carlo estimation of $f_{X^N(t)}(x)$ introduces a statistical error since, in numerical computations, we are restricted to a finite number $M$ of realizations. Therefore, an error of order $1/\sqrt{M}$ is unavoidable, even for $N$ very large. The results presented in the paper have highlighted the crucial importance of bias and sampling errors. In the future, it would be beneficial to rely on improved sampling strategies, such as multilevel Monte Carlo \cite{giles,giles2}, to balance the bias and sampling errors, while reducing the computational cost of the Monte Carlo estimates of the density. This topic is the focus of our current efforts.

\section*{Acknowledgements}
This work is supported by the Spanish ``Ministerio de Econom\'{i}a y Competitividad'' grant MTM2017--89664--P. Marc Jornet acknowledges the doctorate scholarship granted by PAID, as well as ``Ayudas para movilidad de estudiantes de doctorado de la Universitat Polit\`ecnica de Val\`encia para estancias en 2019'', for financing his research stay at CMAP. Julia Calatayud acknowledges ``Fundaci\'o Ferran Sunyer i Balaguer'', ``Institut d'Estudis Catalans'' and the award from ``Borses Ferran Sunyer i Balaguer 2019'' for funding her research stay at CMAP. All authors are also grateful to Inria (Centre de Saclay, DeFi Team), which hosted Marc Jornet and Julia Calatayud during their research stays at \'Ecole Polytechnique. The authors thank the reviewers for the valuable comments and suggestions, which have greatly enriched the quality of the paper.

\end{document}